\documentclass[a4paper,oneside]{amsart}

\pdfoutput=1

\usepackage[T1]{fontenc}
\usepackage[utf8]{inputenc}

\usepackage{lmodern}
\usepackage{microtype}

\usepackage[ngerman,german,british]{babel}

\usepackage{amssymb}
\usepackage{amsxtra}

\usepackage[svgnames]{xcolor}
\usepackage[unicode=true,citecolor=ForestGreen,pdfusetitle,bookmarksdepth=subsection]{hyperref}
\usepackage[ocgcolorlinks]{ocgx2} 

\usepackage[style=alphabetic,date=year,autolang=langname,backend=biber]{biblatex}
\usepackage[autostyle=true]{csquotes}
\usepackage{tikz}
\usetikzlibrary{positioning}

\usepackage{mathtools}

\usepackage{enumitem}

\usepackage{xfrac}
\let\origfrac\frac
\renewcommand{\frac}[2]{\mathchoice
    {\origfrac{#1}{#2}}
    {\sfrac{#1}{#2}}
    {\sfrac{\scriptstyle #1}{\scriptstyle #2}}
    {\sfrac{\scriptstyle #1}{\scriptstyle #2}}
}

\usepackage{calc}
\usepackage{prettyref}
\newrefformat{sec}{Section~\ref{#1}}
\newrefformat{sub}{\S~\ref{#1}}
\newrefformat{thm}{Theorem~\ref{#1}}
\newrefformat{prop}{Proposition~\ref{#1}}
\newrefformat{cor}{Corollary~\ref{#1}}
\newrefformat{lem}{Lemma~\ref{#1}}
\newrefformat{conj}{Conjecture~\ref{#1}}
\newrefformat{def}{Definition~\ref{#1}}
\newrefformat{nota}{Notation~\ref{#1}}
\newrefformat{rem}{Remark~\ref{#1}}
\newrefformat{exa}{Example~\ref{#1}}
\newrefformat{exas}{Examples~\ref{#1}}
\newrefformat{fact}{Fact~\ref{#1}}
\newrefformat{main}{Theorem~\ref{#1}}
\newrefformat{item}{(\ref{#1})}
\newrefformat{tab}{Table~\ref{#1}}
\newrefformat{eq}{equation~\eqref{#1}}

\theoremstyle{plain}

\newtheorem{main}{Theorem}

\newtheorem{thm}{Theorem}
\newtheorem{prop}[thm]{Proposition}
\newtheorem{cor}[thm]{Corollary}
\newtheorem{lem}[thm]{Lemma}
\newtheorem{conj}[thm]{Conjecture}

\theoremstyle{definition}
\newtheorem{defn}[thm]{Definition}
\newtheorem*{defn*}{Definition}
\newtheorem{nota}[thm]{Notation}
\newtheorem{fact}[thm]{Fact}
\newtheorem*{que*}{Question}

\theoremstyle{remark}
\newtheorem{rem}[thm]{Remark}
\newtheorem{exa}[thm]{Example}
\newtheorem{exas}[thm]{Examples}

\numberwithin{thm}{subsection}
\numberwithin{table}{subsection}

\DeclareMathOperator{\dom}{dom}
\DeclareMathOperator{\Frac}{Frac}
\DeclareMathOperator{\hatplus}{\hat{+}}
\DeclareMathOperator{\hatdot}{\hat{\cdot}}
\DeclareMathOperator*{\iprod}{\dot\prod}

\DeclareMathOperator{\osupp}{\overline{supp}}

\DeclareMathOperator{\ot}{ot}
\DeclareMathOperator{\p}{p}
\DeclareMathOperator{\PRV}{P}
\DeclareMathOperator{\Red}{R}
\DeclareMathOperator{\res}{res}
\DeclareMathOperator{\RV}{RV}
\DeclareMathOperator{\rv}{rv}
\DeclareMathOperator{\Suc}{S}

\DeclareMathOperator{\supp}{supp}

\newcommand{\vj}{v_J}

\newcommand{\Nb}{\mathbb{N}}
\newcommand{\Qb}{\mathbb{Q}}
\newcommand{\Rb}{\mathbb{R}}
\newcommand{\Zb}{\mathbb{Z}}

\newcommand{\Abf}{\mathbf{A}}
\newcommand{\Bbf}{\mathbf{B}}
\newcommand{\Gbf}{\mathbf{G}}
\newcommand{\Hbf}{\mathbf{H}}
\newcommand{\Jbf}{\mathbf{J}}
\newcommand{\Kbf}{\mathbf{K}}
\newcommand{\Lbf}{\mathbf{L}}
\newcommand{\Rbf}{\mathbf{R}}
\newcommand{\Sbf}{\mathbf{S}}
\newcommand{\Zbf}{\mathbf{Z}}

\newcommand{\no}{\mathbf{No}}
\newcommand{\on}{\mathbf{On}}
\newcommand{\oz}{\mathbf{Oz}}

\newcommand{\sRV}{\widehat\RV}
\newcommand{\sPRV}{\widehat\PRV}

\newcommand{\KR}{{\Kbf(\Rb^{\leq 0})}}

\newcommand{\kg}{{K(G^{\leq 0})}}
\newcommand{\KG}{{\Kbf(\Gbf^{\leq 0})}}

\newcommand{\kRR}{{K((\Rb^{\leq 0}))}}
\newcommand{\KRR}{{\Kbf((\Rb^{\leq 0}))}}

\newcommand{\kgg}{{K((G^{\leq 0}))}}

\newcommand{\KGG}{{\Kbf((\Gbf^{\leq 0}))}}
\newcommand{\ZKGG}{{\Zbf + \Kbf((\Gbf^{< 0}))}}

\newcommand{\KHH}{{\Kbf((H^{\leq 0}))}}
\newcommand{\ZKHH}{{\Zbf + \Kbf((H^{< 0}))}}
\newcommand{\ZKHHH}{{\Zbf + \Kbf((\Hbf^{< 0}))}}
\newcommand{\KH}{{\Kbf(H^{\leq 0})}}

\newcommand{\LHH}{{\Lbf((H_\sigma^{\leq 0}))}}
\newcommand{\SLHH}{{\Sbf_\sigma + \Lbf_\sigma((H_\sigma^{< 0}))}}

\newcommand*{\lspan}[1]{\langle#1\rangle_\Qb}

\newcommand*{\refarchquot}{\textup{(\hyperref[item:arch-quot]{A1})}\textsubscript{$\sigma$}}
\newcommand*{\refbigcof}{\textup{(\hyperref[item:big-cof]{A2})}\textsubscript{$\sigma$}}
\newcommand*{\refZpS}{\textup{(\hyperref[item:Z-pS]{A3})}}

\title[Factorisation theory for omnific integers]{A factorisation theory for generalised power series and omnific integers}

\author{Sonia L'Innocente}
\address{School of Science and Technology, Mathematics Division, University of Camerino, Camerino, Italy}
\email{sonia.linnocente@unicam.it}

\author{Vincenzo Mantova}
\address{School of Mathematics, University of Leeds, LS2 9JT Leeds, United Kingdom}
\email{v.l.mantova@leeds.ac.uk}

\date{22\textsuperscript{nd} January 2024}

\thanks{No data are associated with this article. For the purpose of open access, the authors have applied a Creative Commons Attribution (CC BY) licence to any Author Accepted Manuscript version arising from this submission.}
\thanks{The authors were supported by: FIRB2010 ``New advances in the Model theory of exponentiation'' RBFR10V792 (S.I., V.M.); PRIN 2017 ``Mathematical logic: models, sets, computability'' 2017NWTM8R, INdAM-GNSAGA (S.I.); ERC AdG ``Diophantine Problems'' 267273, EPSRC grant EP/T018461/1 (V.M.).}

\subjclass[2020]{Primary 13F25, 13F15, secondary 13A05, 03E10}

\keywords{omnific integers, surreal numbers, pre-Schreier domain, valued ring}

\begin{document}

\begin{abstract}
  We prove that in every ring of generalised power series with non-positive real exponents and coefficients in a field of characteristic zero, every series admits a factorisation into finitely many irreducibles of infinite support, the number of which can be bounded in terms of the order type of the series, and a unique product, up to multiplication by a unit, of factors of finite support.

  We deduce analogous results for the ring of omnific integers within Conway's surreal numbers, using a suitable notion of infinite product. In turn, we solve Gonshor's conjecture that the omnific integer $\omega^{\sqrt{2}} + \omega + 1$ is prime.

  We also exhibit new classes of irreducible and prime generalised power series and omnific integers, generalising previous work of Berarducci and Pitteloud.
\end{abstract}

\maketitle

\addtocontents{toc}{\protect\setcounter{tocdepth}{0}}

\section{Introduction}

\subsection{Conway's conjectures on omnific integers}\label{sub:conjectures}

Conway's ordered field $\no$ of surreal numbers \cite{Con1976} is a proper class containing the field $\Rb$ of real numbers as well as the class $\on$ of all ordinals equipped with Hessenberg's natural operations. $\no$ contains, for instance, the least infinite ordinal $\omega$, as well as $\sqrt{2}\omega$, $1 - \omega$, $\omega^2 + \pi$, $\frac{1}{\sqrt{\omega}}$, $\omega^{-\omega}$. Surreal numbers can be represented as series with coefficients in $\Rb$ and exponents in $\no$ itself, such as the following surreals:
\[ 1 + \sum_{n \in \Nb} \omega^{\frac{1}{n+1}}, \quad \sum_{n \in \Nb} \omega^{\frac{1}{n+1}} + \sum_{n \in \Nb} \omega^{\frac{1}{(n+1)\omega}} + \omega^{\frac{1}{\varepsilon_0}}.\]
More precisely, every surreal number can be expressed in a canonical way as a formal sum $\sum_{x \in \no} b_x\omega^x$, where each $b_x$ is in $\Rb$, and $\{ x \in \no : b_x \neq 0 \}$ is a set which is reverse well ordered, namely every non-empty subset has a maximum. See \prettyref{sub:surreal} for a minimal introduction to $\no$.

Inside $\no$, Conway isolated the subring $\oz$ of the \textbf{omnific integers}: those are the surreal numbers $\sum_{x \in \no} b_x\omega^x$ where $b_0 \in \Zb$ and $b_x = 0$ for all $x < 0$. For instance, the numbers $\omega^\omega$, $\sqrt{2}\omega$, $1 - \omega$ are omnific integers, while $\omega^2 + \pi, \omega^{\frac{-1}{2}}, \omega^{-\omega}$ are not. $\oz$ is an \emph{integer part} of $\no$, namely for every $b \in \no$ there is a unique $c \in \oz$ such that $c \leq b < c + 1$. By \cite{She1964}, $\oz^{\geq 0}$ is a model of \emph{Open Induction} (see \prettyref{rem:open-induction} for some details).

Not every element of $\oz$ admit a factorisation into irreducibles, for instance:

\begin{enumerate}
\item[(a)] $\omega = 2 \cdot \frac{\omega}{2} = 3 \cdot \frac{\omega}{3} = \omega^{\frac{1}{2}} \cdot \omega^{\frac{1}{2}} = \omega^{\frac{1}{3}} \cdot \omega^{\frac{1}{3}} \cdot \omega^{\frac{1}{3}} = \ldots$;
\item[(b)] $\omega^2 - 1 = (\omega + 1)(\omega - 1) = (\omega + 1)(\omega^{\frac{1}{2}} + 1)(\omega^{\frac{1}{2}} - 1) = \ldots$;
\item[(c)] $\omega^\omega = (\omega + 1)(\omega^{\omega-1} - \omega^{\omega-2} + \omega^{\omega-3} - \ldots) = (\omega^{\sqrt{2}} + 1)(\omega^{\omega-\sqrt{2}} - \ldots)$.
\end{enumerate}

Inspired by examples like the above ones, Conway made the following conjectures.
\begin{conj}[{\cite[p46]{Con1976}}]\label{conj:conway}
  \phantom{[empty line]}
  \begin{enumerate}
  \item\label{item:conj-conway-prime} $1 + \sum_{n \in \Nb} \omega^{\frac{1}{n+1}}$ is prime in $\oz$.
  \item\label{item:conj-conway-crd} $\oz$ has the \textbf{refinement property}: for all $b, c, d, e \in \oz$, if $bc = de$, then there are $f, g, h, i \in \oz$ such that $b = fg$, $c = hi$, $d = fh$, $e = gi$.
  \end{enumerate}
\end{conj}
Gonshor also asked if $\omega^{\sqrt{2}} + \omega + 1$ is a prime \cite[p117]{Gon1986}. Integral domains with the refinement property are usually called \textbf{pre-Schreier domains} from \cite{Zaf1987}.

\subsection{Generalised power series}
\label{sub:intro-power-series}
Conway's conjectures are intimately connected to analogous problems on certain rings of power series. Given a field $K$ and a totally ordered abelian group $G$, a \textbf{generalised power series} with \textbf{coefficients} in $K$ and \textbf{exponents} in $G$ is a formal sum $b = \sum_{x \in G} b_xt^x$ with $b_x \in K$ such that the \textbf{support} $\supp(b) \coloneqq \{x \in G : b_x \neq 0\}$ is well-ordered, namely every non-empty subset has a minimum. We may also rewrite such sums as $\sum_{i < \alpha} b_it^{x_i}$ where $\alpha$ is the unique ordinal isomorphic to $\supp(b)$ as a well ordered set, $(b_i : i < \alpha)$ is a sequence of non-zero elements of $K$, and $(x_i : i < \alpha)$ is a strictly increasing sequence in $G$. We call $\alpha$ the \textbf{order type} of $b$, denoted by $\ot(b)$.

Let $K((G))$ denote the set of such series, which has a natural ring structure, and is in fact a field; and given a subset $S \subseteq G$, let $K((S))$ consist of the series with support contained in $S$. An ubiquitous example is the field of formal Laurent series $\mathbb{C}((\mathbb{Z})) = \mathbb{C}((t))$, containing for instance $\sum_{n \in \Nb} nt^{n - 3}$, and the subring of Taylor series $\mathbb{C}((\mathbb{N})) = \mathbb{C}[[t]] \subseteq \mathbb{C}((\Zb))$. $\no$ can be seen as the union of the sets $\Rb((S))$ for $S$ subset of $\no$, thus as the class of series with coefficients in $\Rb$ and exponents in $\no$, after identifying $t$ with $\omega^{-1}$ (see \prettyref{sub:surreal} for more details). For instance, $\omega = t^{-1}$, $1 + \sum_{n \in \Nb} \omega^{\frac{1}{n+1}} = 1 + \sum_{n \in \Nb} t^{\frac{-1}{n+1}}$. We will consistently use series in $\omega^{-1}$ to represent surreal numbers, while keeping $t$ for series in other fields $K((G))$.

Gonshor \cite[Ch.\ 8]{Gon1986} observed that one can attack Conway's conjectures by studying factorisations in rings of the form $\kRR$, where $K$ is a field of characteristic $0$. This leads to similar questions in $\kRR$: is $1 + \sum_{n \in \Nb} t^{\frac{-1}{n+1}}$ prime in $\kRR$? Is $\kRR$ a pre-Schreier domain?

Note that the question has negative answer for other rings of series: for instance, $\Rb((\Qb^{\leq 0}))$ is \emph{not} pre-Schreier (so not a GCD domain). See \prettyref{sub:not-pre-schreier} for more details.

\subsection{Previous results}
\label{sub:previous-results}

\prettyref{conj:conway}\prettyref{item:conj-conway-prime} was settled in a sequence of papers. Let $K$ be a field of characteristic $0$ and $G$ be a divisible ordered abelian group.

Berarducci \cite{Ber2000} proved that for $b  \in K((\Rb^{<0}))$ of the form $b = \sum_{i < \omega^{\omega^\beta}} b_i t^{r_i}$, where $\sup \{r_i : i < \omega^{\omega^\beta}\} = 0$, both $b$ and $1 + b$ are irreducible in $\kRR = K + K((\Rb^{<0}))$ (see \prettyref{sub:ordinal-arithmetic} for the definition of ordinal exponentiation). In particular, so is $1 + \sum_{n \in \Nb} t^{\frac{-1}{n+1}}$. He also proved that $t^{-\sqrt{2}} + t^{-1} + 1$ is irreducible in $\kRR$, and deduced that $1 + \sum_{n \in \Nb} \omega^{\frac{1}{n+1}}$ and $\omega^{\sqrt{2}} + \omega + 1$ are irreducible in $\oz$.

Pitteloud \cite{Pit2001} proved that if moreover $b = \sum_{i < \omega} b_i t^{r_i}$ (that is, when $\beta = 0$), both $b$ and $1 + b$ are prime in $\kRR$, and deduced the existence of prime elements in $\kgg$ when $G$ admits a maximal convex subgroup. Then Biljakovic, Kochetov, and Kuhlmann \cite[Thm.\ 4.12]{BKK2006} observed that the primes $1+b$ as above remain prime in $\kgg$ for any $G \supseteq \Rb$. These results imply that $1 + \sum_{n \in \Nb} \omega^{\frac{1}{n+1}}$ is prime in $\oz$, proving \prettyref{conj:conway}\prettyref{item:conj-conway-prime}. On the other hand, the primality of $\omega^{\sqrt{2}} + \omega + 1$ was not addressed.

Moreover, Biljakovic, Kochetov and Kuhlmann proved that irreducibles and primes are unbounded in $\kgg$ and showed various results on some \emph{truncation closed} (\prettyref{rem:open-induction}) subrings of $\kgg$ \cite{BKK2006}. Pitteloud proved that the ideal $J$ generated by $t^x$ for $x \in G^{<0}$ is prime in $\kgg$ and that every non-zero and non-unit element of the quotient $\kgg_{/J}$ admits a factorisation into irreducibles \cite{Pit2002}. For series of the form $b = \sum_{i < \omega^k} b_i t^{r_i} \in K((\Rb^{<0}))$, again with $\sup \{r_i : i < \omega^k\} = 0$, both $b$ and $1+b$ have unique factorisation into irreducibles in $\kRR$ for $k=2$ (by Pommersheim and Shahriari \cite{PS2006}) and $k=3$ (by the authors of this paper \cite{LM2017}); in those papers, choices of $b$ are exhibited so that both $b$ and $1+b$ are irreducible.

We also remark that in \cite{Rib1995}, developed independently of Conway's conjectures, Ribenboim studied the rings $R((S))$ of sums $\sum_x b_xt^x$ with coefficients in a ring $R$ and the exponents in a partially ordered monoid $S$, with suitable restrictions on the supports. Then \cite[6.10]{Rib1995} implies that $\oz$ is \textbf{seminormal}, that is, if $b^2 = c^3$ for some $b, c \in \oz$, there is $d \in \oz$ such that $b = d^3$ and $c = d^2$. However, the rest of \cite{Rib1995} and subsequent works on $R((S))$ (such as \cite{KP2001,Liu2004,Zho2004,BR2008}) do not give information about $\oz$: they typically restrict the order of $S$ in a way that exclude the monoid $G^{\leq 0}$. On the other hand, most results in those papers apply to $G^{\geq 0}$, and one has for instance that $K((G^{\geq 0}))$ is a GCD domain (see \prettyref{rem:KGG-geq0-GCD} for a very short proof).

\subsection{Complete factorisations}

In this paper, we provide a theory of \emph{complete} factorisations in $\oz$, in the sense that we express omnific integers as (possibly infinite) products of factors which are either irreducible, or reducible but with a well understood factorisation theory. We first prove a factorisation theorem in the ring $\KRR$, where $\Kbf$ is a field of characteristic $0$ that may be a proper class.

Given a set $A \subseteq \Rb$ of real numbers, let $\langle A \rangle_\Qb$ denote the linear span of $A$ over $\Qb$, seen as a $\Qb$-vector space.

\begin{main}
  \label{main:KRR}
  Let $b \in \KRR$ with $b \neq 0$. Then there are $r \in \Rb^{\leq 0}$, $n \in \Nb$, and $c_1, \dots, c_n \in \KRR$ such that $b = t^rc_1 \cdots c_n$, where each $c_i$ satisfies exactly one of the following:
  \begin{enumerate}
    \item\label{item:main-KRR-finite-irred} $c_i$ is irreducible with finite support and $\langle\supp(c_i)\rangle_\Qb$ has dimension $\geq 2$;
    \item\label{item:main-KRR-finite-dim-1} $c_i$ has finite support, $\langle\supp(c_i)\rangle_\Qb$ has dimension $1$, and $0 \in \supp(c_i)$;
    \item\label{item:main-KRR-infinite-irred} $c_i$ is irreducible with infinite support;
  \end{enumerate}
  and the supports of the $c_i$'s satisfying \prettyref{item:main-KRR-finite-irred} are pairwise $\Qb$-linearly independent.

  Moreover, $n$ can be bounded in terms of $\ot(b)$, $r$ is unique, and the factors $c_i$ with finite support are unique up to reordering and up to multiplication by elements of $\Kbf$.
\end{main}

The factors in \prettyref{item:main-KRR-finite-dim-1} and $t^r$ are the only ones that can be written as $Q(t^s)$ for some polynomial $Q$ over $\Kbf$ and some $s \in \Rb^{\leq 0}$. Such factors may be irreducible, or may also have no irreducible divisors; however, their factorisation theory is easy to describe (see \prettyref{rem:factorisations-of-fractional-polynomials}). The key step in the proof of \prettyref{main:KRR} is that every series $b \in \KRR$ admits a maximal divisor $p$ with finite support, which is unique up to multiplication by elements of $\Kbf$ (\prettyref{prop:pb-KRR}), and the quotient of $b$ by such $p$ splits into irreducible factors (\prettyref{thm:KRR-fact-unique}); the factorisation of $p$ is then a version of Ritt's factorisation theory \cite{Rit1927} (see \prettyref{prop:ritt}).

The irreducible factors with finite support are prime, by the uniqueness clause of \prettyref{main:KRR}. More generally, every series $p$ with finite support is \textbf{primal} (from \cite{Coh1968}) in the ring $\KRR$ (\prettyref{cor:KR-primal-KRR}), namely, if $p$ divides a product $cd \in \KRR$, then we can write $p = p_1p_2$ with $p_1,p_2 \in \KRR$ so that $p_1$ divides $c$ and $p_2$ divides $d$. Moreover, $p_1, p_2$ have both finite support.

In particular, since $t^{-\sqrt{2}} + t^{-1} + 1$ is irreducible (\cite{Ber2000}), it is prime in $\KRR$. Using a suitable primality transfer result (\prettyref{prop:lift-primal-easy-reduced}, inspired by \cite[Cor.\ 4.3]{BKK2006}), we obtain the positive answer to Gonshor's question.

\begin{main}
  \label{main:Gonshor}
  The omnific integer $\omega^{\sqrt{2}} + \omega + 1$ is prime in $\oz$.
\end{main}

\prettyref{main:KRR} cannot be generalised verbatim to $\oz$, as an omnific integer may well have infinitely many pairwise coprime irreducible divisors. We side step this issue by introducing an appropriate infinite product of (some) omnific integers.

We call an omnific integer $b \in \oz$ \emph{reduced} if $\supp(b) \cap \supp(b-1)$ is contained in a single Archimedean class (see \prettyref{def:reduced}). For instance, $\sum_{n \in \Nb} \omega^{\frac{1}{n+1}}$, $\omega^{\sqrt{2}} + \omega + 1$, $3$ are reduced, while $\omega^\omega + \omega^{\sqrt{2}}$, $\omega^{\sqrt{2}} + \omega + 3$ are not. A key reason for this choice is that every omnific integer admits reduced factors, for example $\omega^{\sqrt{2}} + \omega + 3 = (\frac{1}{3}\omega^{\sqrt{2}} + \frac{1}{3}\omega + 1)3$ (see also \prettyref{fact:gonshor-arch-classes}).

By construction, reduced omnific integers can be identified, in a canonical way, with series in rings of the form $\Zb + \Kbf((\Rb^{<0}))$ for some field $\Kbf$ (see \prettyref{fact:isomorphism-K-H-sigma}, \prettyref{def:rho}), or more precisely in $\Kbf((\Rb^{<0})) \cup (1 + \Kbf((\Rb^{<0})))$. Thus we may apply \prettyref{main:KRR} and obtain a factorisation theory for reduced integers in the same style, with the appropriate changes (\prettyref{prop:reduced-fact-unique}); for instance, series with finite support are replaced by \emph{pseudo-polynomials}, which are reduced omnific integers identified with series of finite support in $\Zb + \Kbf((\Rb^{<0}))$ (\prettyref{def:pseudo-stuff}). Reduced omnific integers with finite support, such as $\omega^{\sqrt{2}} + \omega + 1$, are examples of pseudo-polynomials.

We then introduce a partially defined infinite product operator that maps (some) families $(c_i : i \in I)$ of reduced omnific integers, where $I$ is a set, to $\iprod_{i \in I} c_i \in \oz$ (see \prettyref{def:infinite-product}; note that it differs from the usual infinite product of surreal numbers, see \prettyref{rem:infinite-products}). When an omnific integer $\sum_{x \in \no^{\geq 0}} b_x\omega^x$ is such that $b_0 \neq 0$, or in other words, when $0$ in its support, we obtain the following.

\begin{main}
  \label{main:Oz}
  Let $b \in \oz$ with $0 \in \supp(b)$. Then there exist a family $(c_i : i \in I)$ of reduced omnific integers such that $b = \iprod_{i \in I}c_i$ and each $c_i$ is irreducible, a pseudo-polynomial, or $\pm 1$.
\end{main}

Moreover, the irreducible factors in $\Zb$ are unique up to sign, and the pseudo-polynomials are unique under restrictions similar to the ones in \prettyref{main:KRR}. If $0 \notin \supp(b)$, there are additional issues to consider, as for instance such a $b$ is divisible by every non-zero $n \in \Zb$. We deal with this by allowing \emph{pseudo-irreducible} and \emph{pseudo-monomial} factors (\prettyref{def:pseudo-stuff}), which have a very simple factorisation theory (Propositions~\ref{prop:fact-pseudo-monomial},~\ref{prop:fact-almost-irreducible}).

More generally, we state in \prettyref{thm:ZKGG} a general factorisation theorem for rings of the form $\ZKGG_\kappa$, where $\Kbf$ is a field of characteristic $0$, $\Zbf$ is a subring of $\Kbf$, $\Gbf$ is a divisible ordered abelian group, all of which may be proper classes, and the subscript $\kappa$ means that we restrict to those series whose support has size less than some uncountable cardinal $\kappa$, or whose support is a set when $\kappa = \on$ (see \prettyref{sub:generalised-power}). In this generality, we need to allow for a further type of factor (\emph{almost irreducible}, see again \prettyref{def:pseudo-stuff}) which does not appear in omnific integers. \prettyref{main:Oz} follows as a special case.

We also find new primes and primal elements in $\oz$ (see \prettyref{cor:primal-in-oz}), for instance we prove that $\sum_{n \in \Nb}\omega^{\frac{1}{n+1}}$ is primal (see also \prettyref{exas:primes-primal}) and that every omnific integer with finite support is primal (\prettyref{cor:finite-supp-primal}, \prettyref{exa:finite-supp-primal}). For this, we prove a generalisation of \cite[Cor.\ 4.3]{BKK2006} in \prettyref{prop:lift-primal}, where we show how to construct primal elements in $\ZKGG_\kappa$ starting from primal elements in rings of the form $\Lbf((H^{\leq 0}))$, where $H$ is a divisible Archimedean group, and $\Gbf$, $\kappa$ and $\Zbf$ satisfy specific assumptions described at the beginning of \prettyref{sec:on-primal-series}. We also show that if any of those assumptions are violated, then $\ZKGG_\kappa$ is \emph{not} a pre-Schreier domain (\prettyref{thm:non-pS}), as we are able to exhibit non-primal series (\prettyref{sub:not-pre-schreier}). We note that in the particular case $\Kbf((H^{\leq 0}))$, for $H$ a divisible Archimedean ordered group, one could still ask whether the subring of $\Kbf((H))$ of the series with support bounded from above is pre-Schreier.

\subsection{A new valuation}
\label{sub:intro-new-valuation}

The main ingredient of \cite{Ber2000} and many of the papers cited in \prettyref{sub:previous-results} is a specially crafted function $v_J : \kRR \to \on$, which happens to be a \textbf{multiplicative semi-valuation}, namely it satisfies $v_J(bc) = v_J(b) \odot v_J(c)$ and $v_J(b + c) \leq \max\{v_J(b),v_J(c)\}$ for all $b, c \in \kRR$ (see \prettyref{sub:order-value}; we use the word `multiplicative' to emphasise that the ultrametric inequality uses the maximum, rather than the minimum which is more common in valuation theory). Here $\odot$ denote Hessenberg's natural product on $\on$.  See \prettyref{sub:ordinal-arithmetic} for a minimal introduction to the ordinal arithmetic needed in this paper; here we let $\oplus$ denote Hessenberg's natural sum, and $\alpha \mapsto \omega^\alpha$ denote ordinal exponentiation in base $\omega$. To dispel any potential ambiguity, we also use $\hatplus$, $\hatdot$ for the classical (non-commutative) ordinal operations.

In this paper, we use a different function $\KRR \to \on \cup \{-\infty\}$ that also happens to be a valuation.

\begin{defn*}
  Given $b \in \KRR$ with $b \neq 0$, let the \textbf{degree} of $b$, denoted by $\deg(b)$, be the maximum ordinal $\alpha$ such that $\omega^\alpha \leq \ot(b)$. We also let $\deg(0) \coloneqq -\infty$.
\end{defn*}

The series of degree $0$ are precisely the ones with finite support, such as $t^{-\sqrt{2}} + t^{-1} + 1$, except for $0$, which has degree $-\infty$. Those of degree $1$ have order types between $\omega$ (included) and $\omega^2$ (excluded). For instance, the series $\sum_{n \in \Nb} t^{-1 - \frac{1}{n+1}} + \sum_{n \in \Nb} t^{\frac{-1}{n+1}}$ has degree $1$ because its order type is $\omega \hatplus \omega$, whereas $\sum_{(m,n) \in \Nb} t^{\frac{-1}{(n+1)(m+1)}}$ has degree $2$, as its order type is $\omega^2$. We shall prove that the degree is a \textbf{multiplicative valuation} in the following sense.

\begin{main}
  \label{main:degree}
  For all non-zero $b, c \in \KRR$,
  \begin{enumerate}
  \item\label{item:main-deg-ultra} $\deg(b + c) \leq \max\{\deg(b), \deg(c)\}$ (ultrametric inequality);
  \item\label{item:main-deg-mult} $\deg(bc) = \deg(b) \oplus \deg(c)$ (multiplicativity);
  \item\label{item:main-deg-infty} $\deg(b) = -\infty$ if and only if $b = 0$.
  \end{enumerate}
\end{main}

An earlier proof of condition \prettyref{item:main-deg-mult} can be found in \cite[Rem.\ 4.3, Cor.\ 4.7]{LM2017} from the authors of this paper, however with a considerably different notation based on a combination of tools from \cite{Ber2000}. Here we shall describe a more self-contained proof based only on \cite[Cor.\ 9.9]{Ber2000}.

Thus the degree is quite similar to $v_J$, with $\oplus$ replacing $\odot$, and $-\infty$ replacing $0$. A key difference is that $v_J$ does not satisfy condition \prettyref{item:main-deg-infty}: the set $\{b \in \KRR : v_J(b) = 0\}$ is the ideal generated by $t^x$ for $x \in \Rb^{<0}$. We have that $v_J(b) \leq \omega^{\deg(b)} \leq \ot(b)$ for all series $b$, but both inequalities are often strict (see an example in \prettyref{sub:order-value}).

We remark that the map $\sup : \KRR \to \Rb^{\leq 0} \cup \{-\infty\}$ defined by $\sup(b) \coloneqq \sup(\supp(b))$ for $b \neq 0$, and $\sup(0) \coloneqq -\infty$ (\prettyref{def:sup}) is also a multiplicative valuation (\prettyref{prop:sup-valuation}). This is an almost immediate consequence of \cite[Cor.\ 9.8]{Ber2000}, but we give it a full proof for completeness. We discuss how the naive generalisations of $\deg$ and $\sup$ fail to be multiplicative in $\KGG_\kappa$ for arbitrary $\Gbf$ in Remarks~\ref{rem:deg-not-mult},~\ref{rem:sup-G}, and how they can be corrected in \prettyref{cor:deg-sup-sigma-mult}. With the latter Corollary, we recover a short proof of the main result of \cite{Pit2002}, which is that the ideal generated by the monomials $t^x$ for $x \in \Gbf^{<0}$ is prime (\prettyref{cor:pitteloud}). We also show how to construct a true multiplicative valuation $\sup$ on $\KGG_\kappa$ taking values in a quotient of a suitable version of the Dedekind-MacNeille completion of $\Gbf$ (\prettyref{sub:sup-valuation}).

\subsection{New criteria for irreducibility and primality}

The proof of \prettyref{main:KRR} yields additional information about irreducibles and primes. In particular, we obtain a strengthening of Berarducci's criterion for irreducibility \cite[Thm.\ 10.5]{Ber2000}.

\begin{main}
  \label{main:irreducible}
  For all $b \in \KRR$, if the order type of $b$ is of the form $\omega^{\omega^{\alpha}} \hatplus \beta$ with $\beta < \omega^{\omega^{\alpha}}$, and $b$ is not divisible by $t^x$ for any $x \in \Rb^{<0}$, then $b$ is irreducible.
\end{main}

For comparison, the criterion in \cite{Ber2000} is the special case $\beta = 0,1$ and with $0$ an accumulation point of $\supp(b)$. Thus, the above shows for instance that the series $\sum_{n \in \Nb} t^{-1-\frac{1}{n+1}} + t^{\frac{-1}{3}} + 1$, which has order type $\omega \hatplus 2$, is irreducible. The above is a consequence of a more technical criterion (\prettyref{cor:irreducibility}) which allows us to find further irreducible series; for instance, we can prove that for every ordinal $\beta < \omega_1$ and every ordinal $\omega^{\omega^\beta} \leq \alpha < \omega^{\omega^\beta \hatplus 1}$ there exists an irreducible series $b \in \KRR$ of order type $\alpha$ (\prettyref{prop:irreducible-principal-degree}; see also \prettyref{exa:irreducible-omega2+k+1}). On the other hand, neither \prettyref{main:irreducible} nor \prettyref{cor:irreducibility} detect irreducibles of order type $\omega^2$ and $\omega^3$ as the ones in respectively \cite{PS2006}, \cite{LM2017}, which require different techniques.

In a similar way, we strengthen the Pitteloud's primality criterion \cite{Pit2001}.

\begin{main}
  \label{main:prime}
  For all $b \in \KRR$, if the order type of $b$ is of the form $\omega \hatplus k$ with $k < \omega$, and $b$ is not divisible by $t^x$ for any $x \in \Rb^{<0}$, then $b$ is prime.
\end{main}

In particular, the aforementioned example $\sum_{n \in \Nb} t^{-1-\frac{1}{n+1}} + t^{\frac{-1}{3}} + 1$ of order type $\omega \hatplus 2$ is also prime. For comparison again, the main result of \cite{Pit2001} is the special case $b = 0, 1$ and with $0$ an accumulation point of $\supp(b)$. Our proof goes through a technical criterion (\prettyref{cor:b-irred-frac-sRV1-prime}) that can be used to find further primes. For instance, we show that for every ordinal $\omega \leq \alpha < \omega^2$, there is a prime series $b \in \KRR$ of order type $\alpha$ (\prettyref{prop:prime-degree-1}; see also \prettyref{exa:irreducible-omega2+k+1}).

The above results can be used to find new irreducible, prime, and primal omnific integers via \prettyref{prop:lift-primal}, which allows to transfer these properties from $\KRR$ to $\ZKGG_\kappa$. In particular, we find that $\sum_{n \in \Nb} \omega^{1+\frac{1}{n+1}} + \omega^{\frac{1}{3}} + 1$ is prime in $\oz$ (see \prettyref{exas:primes-primal}).

\subsection*{Acknowledgements}

The authors thank the anonymous referees for the several comments that guided thorough revisions of this paper. The authors would also like to thank the \foreignlanguage{german}{Mathematisches Forschungsinstitut Oberwolfach} for the support during the mini-workshop ``Surreal Numbers, Surreal Analysis, Hahn Fields and Derivations'' in 2016, and to express their gratitude to both organisers and participants for the opportunity to discuss and present a preliminary version of the results in this paper \cite{BEK2016}. Both authors gratefully acknowledge the support of the project FIRB2010 ``New advances in the Model theory of exponentiation'' RBFR10V792, of which the first author was the principal investigator. S.I.\ was partially supported by the PRIN 2017 ``Mathematical logic: models, sets, computability'' 2017NWTM8R and by INdAM-GNSAGA\@. V.M.\ was partially supported by the ERC AdG ``Diophantine Problems'' 267273 and the EPSRC grant EP/T018461/1.
\tableofcontents

\addtocontents{toc}{\protect\setcounter{tocdepth}{1}}
\section{Preliminaries}
\label{sec:preliminaries}

Let us recall some basic definitions and facts, as well as fix the notations that will be used throughout the paper.

\subsection{Generalised power series}
\label{sub:generalised-power}
Let $K$ be a field and $G$ be a totally ordered abelian group. Let $\Kbf$, $\Gbf$ be also a field and a totally ordered abelian group, but whose underlying domains may be a proper classes. We have already introduced in \prettyref{sub:intro-power-series} the set $K((G))$ of generalised power series with coefficients in $K$ and exponents in $G$. We recall a few details about $K((G))$, introduce the \emph{canonical valuation}, and make some early comments about series with finite support.

The ring structure mentioned in the introduction is defined by
\begin{align*}
  \sum_{x \in G} b_xt^x + \sum_{x \in G} c_xt^x &\coloneqq \sum_{x \in G} (b_x + c_x)t^x, \\
  \left(\sum_{x \in G} b_xt^x\right)\left(\sum_{x \in G} c_xt^x\right) &\coloneqq \sum_{x \in G} \left(\sum_{y+z=x}b_yc_z\right)t^x.
\end{align*}
One can easily verify that the right hand sides are well defined and in $K((G))$ (namely, the supports are well ordered, and for the second definition, each term $\sum_{y + z = x}b_yc_z$ contains at most finitely many non-zero summands). Call \textbf{monomials} the series of the form $kt^x$ for $x \in G$ and $k \in K \setminus \{0\}$. Note that by the above definition, $t^x t^y = t^{x+y}$. With these operations, $K((G))$ is a $K$-algebra (for instance, the multiplicative identity is the monomial $t^0$). We identify the elements of $K$ with the monomials of the form $kt^0$.

Furthermore, let $K((G))_\kappa$ denote the subset of the series in $K((G))$ with support of cardinality strictly less than $\kappa$, where $\kappa$ is an uncountable cardinal; equivalently, $K((G))_\kappa$ is the union of $K((H))$ for $H$ ranging over the subgroups of $G$ of size less than $\kappa$. We call such series \textbf{restricted} or \textbf{$\kappa$-bounded}, and they form a $K$-algebra.

We generalise the above to the class $\Kbf$ by letting $\Kbf((G))$ be the class $\bigcup_K K((G))$ for $K$ ranging over the subfields of $\Kbf$ that are sets. For $\Gbf$, we cannot give a general meaning to `$\Kbf((\Gbf))$', as it involves the power of $\Gbf$, but we may define $\Kbf((\Gbf))_\on$ to be the class $\bigcup_G \Kbf((G))$ for $G$ ranging over the subgroups of $\Gbf$ that are sets. We may similarly define $\Kbf((\Gbf))_\kappa$, for $\kappa$ an uncountable cardinal, by restricting the above union to the subgroups of size less than $\kappa$. Note that $\Kbf((G))_\on = \Kbf((G))$.

One can check that every non-zero element of $K((G))$ has a multiplicative inverse (\cite{Hah1907}), thus $K((G))$ is a field. It follows immediately that $\Kbf((\Gbf))_\kappa$ is a field for any $\kappa$ uncountable cardinal and for $\kappa = \on$. Moreover, we have the following.
\begin{fact}
  \label{fact:K((G))-field-ACF-RCF}
  $K((G))$ is a field, and it is:
  \begin{enumerate}
  \item\label{item:fact-K((G))-ACF} \textbf{algebraically closed} (namely, every non-trivial polynomial over $K((G))$ has a root in $K((G))$) if and only if $K$ is algebraically closed and $G$ is divisible;
  \item\label{item:fact-K((G))-RCF} \textbf{real closed} (namely, $K((G))$ becomes algebraically closed after, and only after, adjoining the square root of $-1$) if and only if $K$ is real closed and $G$ is divisible.
  \end{enumerate}
  The same facts hold for $\Kbf((\Gbf))_\kappa$ for any $\kappa$ uncountable cardinal and for $\kappa = \on$.
\end{fact}
The field $K((G))$ is usually called a \textbf{Hahn field}, and $K((G))_\kappa$ is called restricted or $\kappa$-bounded Hahn field. We extend this terminology unchanged to respectively $\Kbf((G))$ and $\Kbf((\Gbf))_\kappa$.

If $S$ is a subset of $G$, then $K((S))$ (the set of the series with support contained in $S$) is clearly a $K$-vector space. If $S$ is a semigroup, then $K((S))$ is closed under product; if $S$ is a monoid, then $K((S))$ also contains $t^0 = 1$ and so is a subring of $K((G))$. In particular, $\kgg$ is a subring of $K((G))$.

These considerations extend to $\Kbf((S))$ in the obvious way, where $\Kbf((S))$ is defined as the union of $K((S))$ for $K$ ranging over the subfields of $\Kbf$ that are sets. If $\Sbf$ is a subclass of a class $\Gbf$, then we can again define $\Kbf((\Sbf))_\kappa$ as the union of $\Kbf((S))$ for $S$ subset of $\Sbf$ of size less than $\kappa$ (if $\kappa$ is a cardinal) or of any size (if $\kappa = \on$). Once again, the above facts about $K((S))$ hold for $\Kbf((\Sbf))_\kappa$ for any uncountable cardinal $\kappa$ and for or $\kappa = \on$; for instance, $\KGG_\kappa$ is a subring of $\Kbf((\Gbf))_\kappa$.

\begin{rem}
  \label{rem:open-induction}
  Let $\Kbf$ be an ordered field and $\Zbf$ denote an \emph{integer part} of $\Kbf$, that is, $\Zbf$ is a subring of $\Kbf$ such that for each $k \in \Kbf$ there is a unique $\lfloor k \rfloor \in \Zbf$ such that $\lfloor k \rfloor \leq k < \lfloor k \rfloor + 1$. Here $\Zbf$ and $\Kbf$ may be proper classes. A typical example is the ring $\mathbb{Z}$ in $\mathbb{R}$, or as mentioned in the introduction, $\oz$ in $\no$.

  Then the ring $\ZKGG_\kappa$ is an integer part of $\Kbf((\Gbf))_\kappa$: given a series $b = \sum_{x \in \Gbf} b_xt^x$ in $\Kbf((\Gbf))_\kappa$, the element $\lfloor b \rfloor = \lfloor b_0 \rfloor + \sum_{x \in \Gbf^{<0}} b_xt^x$ is the unique one in $\ZKGG_\kappa$ satisfying $\lfloor b \rfloor \leq b < \lfloor b \rfloor + 1$. By \cite{She1964}, the non-negative elements of the integer part of a real closed field form a model of Open Induction (the fragment of Peano's arithmetic with induction restricted to quantifier-free formulas). By \prettyref{fact:K((G))-field-ACF-RCF}\prettyref{item:fact-K((G))-RCF}, it follows that the non-negative elements of $\ZKGG_\kappa$, when $\Kbf$ is real closed and $\Gbf$ is divisible, form such a model. In particular, $\oz^{\geq 0}$ is a model of Open Induction.

  Note, moreover, that $\ZKGG_\kappa$ is the unique integer part of $\Kbf((\Gbf))_\kappa$ containing $\Zbf$ that is \emph{truncation closed}, that is, if it contains $\sum_x b_xt^x$, then it contains $\sum_{x < y} b_xt^x$ for all $y \in \Gbf$.

  Since every real closed field $R$ admits an embedding $\iota : R \to \Rb((G))$ with truncation closed image, the preimage of $\Zb + \Rb((G^{<0}))$ through $\iota$ is an integer part of $R$, thus every real closed field admits an integer part \cite{MR1993}. Many questions about integer parts that are truncation closed when embedded into some $\Rb((G))$, such as the existence of irreducible and prime elements, are studied in \cite{BKK2006}.
\end{rem}

Let $G_\infty \coloneqq G \cup \{-\infty\}$. For $x \in G_\infty$, let
\begin{equation*}
  -\infty < x, \quad -\infty + x = x + (-\infty) \coloneqq -\infty.
\end{equation*}
Note the quirk $-\infty < -\infty$, which is convenient to avoid singling out the special value $-\infty$ in some statements. The \textbf{canonical} (or \textbf{$t$-adic}) {valuation} of $K((G))$ is the function $v : K((G)) \to G_\infty$ defined by
\begin{equation*}
  v(b) \coloneqq \begin{cases}
    \min \supp(b) & b \neq 0 \\
    -\infty & b = 0.
  \end{cases}
\end{equation*}

\begin{fact}
  For all $b, c \in K((G))$:
  \begin{enumerate}
  \item\label{item:fact-v-ultra} $v(b + c) \geq \min\{v(b), v(c)\}$;
  \item\label{item:fact-v-mult} $v(bc) = v(b) + v(c)$;
  \item\label{item:fact-v-infty} $v(b) = -\infty$ if and only if $b = 0$.
  \end{enumerate}
\end{fact}

\begin{cor}
  \label{cor:KGG-units}
  The units of $\kgg$, that is to say its invertible elements, are exactly the non-zero elements of $K$.
\end{cor}
\begin{proof}
  If $b, c \in \kgg$ are such that $bc = 1$, then $0 = v(bc) = v(b) + v(c)$, hence $v(b) = v(c) = 0$. Therefore, $b = b_0t^0 = b_0$ where $b_0$ is a non-zero element of $K$.
\end{proof}

The definition of canonical valuation extends naturally to $\Kbf((\Gbf))_\kappa$, where it remains a valuation, for any uncountable $\kappa$ and for $\kappa = \on$, and the above observation about the units remains valid in $\KGG_\kappa$.

\begin{nota}
  Given a subclass $\Sbf \subseteq \Gbf$, let $\Kbf(\Sbf)$ denote the subclass of $\Kbf((\Sbf))_\on$ of the series with finite support.
\end{nota}

Note that the $\Kbf(\Sbf) \subseteq \Kbf((\Sbf))_\kappa$ for any uncountable $\kappa$. Just like for $K((S))$, $\Kbf(\Sbf)$ is always a $\Kbf$-vector space, and if $\Sbf$ is a submonoid of $\Gbf$, then $\Kbf(\Sbf)$ is a subring of $\Kbf((\Sbf))_\on$.

\begin{fact}
  \label{fact:KS-is-UFD}
  Suppose that $S$ is a free commutative monoid on $n$ generators, say $S = \Nb x_1 + \dots + \Nb x_n$ for some $x_1, \dots, x_n \in S$. Then $\Kbf(S)$ is isomorphic to the ring $\Kbf[X_1, \dots, X_n]$ of polynomials in $n$ variables, with the isomorphism mapping $t^{x_i}$ to $X_i$. Then $\Kbf(S)$ is a unique factorisation domain, and its units are the non-zero elements of $\Kbf$.

  Likewise, if $S$ is a free abelian group on $n$ generators, say $S = \Zb x_1 + \dots + \Zb x_n$, then $\Kbf(S)$ is isomorphic to the ring $\Kbf[X_1^{\pm1}, \dots, X_n^{\pm1}]$ of Laurent polynomials in $n$ variables, again mapping $t^{x_i}$ to $X_i$. Hence this $\Kbf(S)$ is also a unique factorisation domains, and this time the units are exactly the monomials of $\Kbf(S)$.
\end{fact}

\subsection{Ordinal arithmetic}
\label{sub:ordinal-arithmetic}

We will make regular use of \emph{Hessenberg natural operations}, \emph{ordinal exponentiation} (only in base $\omega$), \emph{Cantor normal forms}, and the \emph{degree}. We give a minimal self-contained presentation of ordinals, with some examples and no proofs, in order to define all of the above. More details can be found in several sources, such as \cite{Sie1958}.

In order to work correctly with the proper classes $\no$, $\on$, our underlying theory will be the von Neumann-Bernays-Gödel set theory without choice (NBG). All the facts below are theorems of NBG.

We work with \emph{von~Neumann} ordinals, as they are more convenient for surreal numbers. A von~Neumann \textbf{ordinal} is a set $\alpha$ such that the membership relation is a well order on $\alpha$ and that is also \textbf{transitive}, namely every element of $\alpha$ is an (initial) subset of $\alpha$. Every well ordered set $A$ is isomorphic to exactly one ordinal $\alpha$ as an ordered set, in which case we call $\alpha$ the \textbf{order type} of $A$ and write $\alpha = \ot(A)$. This justifies the notation introduced in \prettyref{sub:intro-new-valuation}: for $b \in \KRR$, $\supp(b)$ is well ordered (by definition) and so we have $\ot(b) = \ot(\supp(b))$.

Given two ordinals $\alpha, \beta$, say that $\alpha$ is \textbf{less than} $\beta$, written $\alpha < \beta$, if $\alpha \in \beta$ (in particular $\alpha \subsetneq \beta$, since ordinals are transitive). The (proper) class of all ordinals is denoted by $\on$, and $<$ is a well order on $\on$, namely every non-empty subclass has a minimum. Given an ordinal $\alpha$ or a set (not a proper class) $A$ of ordinals, define:
\begin{align*}
  0 & \coloneqq \varnothing && \text{the least ordinal (\textbf{zero})}; \\
  \Suc(\alpha) & \coloneqq \alpha \cup \{\alpha\} && \text{the least ordinal }>\alpha\text{ (\textbf{successor} of }\alpha\text{)}; \\
  \sup A & \coloneqq \bigcup_{\beta \in A} \beta && \text{the least ordinal }\geq \beta\text{ for all } \beta \in A\text{ (\textbf{supremum} of }A\text{)}.
\end{align*}
One can then define $1 = \Suc(0) = \{0\}$, $2 = \Suc(1) = \{0,1\}$, and so on. Let $\omega$ be the least \textbf{limit} ordinal, namely the least one that is neither a successor nor zero. We identify $\Nb$ with $\omega$. Let $\omega_1$ be the first uncountable ordinal.

Classical ordinal arithmetic can be defined by induction on $\on$. In the literature, the classical operations are usually denoted by $+$ and $\cdot$, but to prevent confusion with the field operations on the surreal numbers, we decorate them with a hat.

\medskip

\begin{center}{
  \renewcommand{\arraystretch}{1.4}
  \begin{tabular}{rcp{0.1\textwidth}p{0.15\textwidth}p{0.25\textwidth}}
     && $\beta = 0$ & $\beta = \Suc(\gamma)$ & $\beta$ limit \\ \hline
     \textbf{sum} $\alpha \hatplus \beta$ & $\coloneqq$ & $\alpha$ & $\Suc(\alpha \hatplus \gamma)$ & $\sup\{\alpha \hatplus \gamma : \gamma < \beta\}$ \\
     \textbf{product} $\alpha \hatdot \beta$ & $\coloneqq$ & $0$ & $\alpha \hatdot \gamma \hatplus \alpha$ & $\sup\{\alpha \hatdot \gamma : \gamma < \beta\}$ \\
     \textbf{exponentiation} $\alpha^{\beta}$ & $\coloneqq$ & $1$ & $\alpha^{\gamma} \hatdot \alpha$ & $\sup\{\alpha^{\gamma} : \gamma < \beta \}$.
  \end{tabular}
}\end{center}

\medskip

For convenience, we let $\on_{\infty} \coloneqq \on \cup \{-\infty\}$, and for $\alpha \in \on_{\infty}$ we set
\begin{equation*}
  -\infty < \alpha, \quad\! -\infty \hatplus \alpha = \alpha \hatplus (-\infty) \coloneqq -\infty, \quad\!
  -\infty \hatdot \alpha = \alpha \hatdot (-\infty) \coloneqq -\infty, \quad\! \alpha^{-\infty} \coloneqq 0.
\end{equation*}
On $\Nb = \omega \subseteq \on$, these operations coincide with the usual sum and product of natural numbers. On $\on$, sum and product are weakly increasing in the first argument and strictly increasing in the second one; likewise, exponentiation is weakly increasing in the base and strictly increasing in the exponent.

Moreover, sum and product are associative, but not commutative. The product is distributive over the sum in the second argument, namely $\alpha \hatdot (\beta \hatplus \gamma) = \alpha \hatdot \beta \hatplus \alpha \hatdot \gamma$; in general, though, $(\beta \hatplus \gamma) \hatdot \alpha$ may differ from $\beta \hatdot \alpha \hatplus \gamma \hatdot \alpha$. We also have $\omega^{\alpha \hatplus \beta} = \omega^\alpha \hatdot \omega^\beta$, as well as $\omega^\alpha \geq \alpha$, for all $\alpha, \beta$. Note that the equation $\omega^\varepsilon = \varepsilon$ has class many solutions, the smallest of which denoted by $\varepsilon_0$.

For all $\alpha \in \on$, there is a maximum $\beta$ such that $\omega^\beta \leq \alpha$, and then a unique $\gamma$ such that $\alpha = \omega^\beta \hatplus \gamma$. Iterating this, we find that there is a unique finite sequence $\beta_1 \geq \beta_2 \geq \ldots \geq \beta_n \geq 0$ of ordinals such that
\begin{equation}
  \label{eq:cantor-normal-form}
  \alpha = \omega^{\beta_1} \hatplus \ldots \hatplus \omega^{\beta_n}.
\end{equation}
The expression on the right-hand side is called \textbf{Cantor normal form} of $\alpha$. The normal form makes it possible to compute sum and product easily. It suffices to recall that both operations are associative, that the product is right distributive over the sum and has identity $1 = \omega^0$, and the following equations:
\begin{align}
  \label{eq:sum-in-cnf}
  \omega^\alpha \hatplus \omega^\beta &= \omega^\beta && \text{if }\alpha < \beta; \\
  \label{eq:product-in-cnf}
  (\omega^{\beta_1} \hatplus \cdots \hatplus \omega^{\beta_n}) \hatdot \omega^\alpha &=
    \omega^{(\max_i\beta_i) \hatplus \alpha} && \text{if } \alpha > 0.
\end{align}

Moreover, $\on$ admits \emph{commutative} operations introduced by Hessenberg \cite{Hes1906}. Given $\alpha = \omega^{\gamma_1} \hatplus \omega^{\gamma_2} \hatplus \cdots \hatplus \omega^{\gamma_n}$ and $\beta = \omega^{\gamma_{n+1}} \hatplus \omega^{\gamma_{n+2}} \hatplus \cdots \hatplus \omega^{\gamma_{n+m}}$ in Cantor normal form, let $\pi$ be a permutation of the integers $1, \ldots, n+m$ such that $\gamma_{\pi(1)} \geq \ldots \geq \gamma_{\pi(n+m)}$. Define
\begin{align*}
  \alpha \oplus \beta &\coloneqq \omega^{\gamma_{\pi(1)}} \hatplus \omega^{\gamma_{\pi(2)}} \hatplus \cdots \hatplus \omega^{\gamma_{\pi(n+m)}} && (\textbf{natural sum}),\\
  \alpha \odot \beta &\coloneqq \bigoplus_{1 \leq i \leq n} \bigoplus_{n+1 \leq j \leq k+m} \omega^{\gamma_j \oplus \gamma_j} && (\textbf{natural product}).
\end{align*}
We extend sum and exponentiation to $\on_{\infty}$ by letting
\begin{equation*}
  (-\infty) \oplus \alpha = \alpha \oplus (-\infty) \coloneqq -\infty, \quad \omega^{-\infty} = 0.
\end{equation*}

\begin{fact}
  The operations $\oplus$ and $\odot$ make $\on$ into an ordered semiring, that is:
  \begin{enumerate}
    \item $\oplus$ is associative, commutative, strictly increasing in both arguments (in particular, it is cancellative), with identity $0$;
    \item $\odot$ is associative, commutative, with identity $1$ and strictly increasing in both arguments on $\on \setminus \{0\}$, whereas $\alpha \odot 0 = 0 \odot \alpha = 0$ for every $\alpha \in \on$;
    \item $\odot$ is distributive over $\oplus$.
  \end{enumerate}
\end{fact}

Note that by definition $\omega^{\alpha} \odot \omega^{\beta} = \omega^{\alpha \oplus \beta}$ for all $\alpha, \beta \in \on_\infty$. Moreover, we have the equality $\omega^{\beta_1} \hatplus \cdots \hatplus \omega^{\beta_n} = \omega^{\beta_1} \oplus \cdots \oplus \omega^{\beta_n}$ if (and only if) the left hand side is in Cantor normal form.

It follows that in Cantor normal form, Hessenberg's arithmetic can be read off as the arithmetic of polynomials in the variables $\omega^{\omega^\alpha}$, for $\alpha \in \on$, with coefficients in $\Nb$. Thus we have the suggestive isomorphism
\[ (\on, \oplus, \odot) \cong (\Nb[X_0,X_1,\dots,X_\alpha,\dots], +, \cdot) \]
where we map $\omega^{\omega^\alpha}$ to the variable $X_\alpha$.

For $\alpha \in \on$, let the \textbf{degree} of $\alpha$ be the exponent $\beta_1$ in its Cantor Normal form when $\alpha > 0$, or $-\infty$ if $\alpha = 0$. In other word, $\deg(\alpha) = \max\{\beta \in \on_\infty : \omega^\beta \leq \alpha\}$. This justifies the notation of \prettyref{sub:intro-new-valuation}: for $b \in \KRR$ we have $\deg(b) = \deg(\ot(b))$.

\begin{fact}
  The degree is a multiplicative valuation on $\on$: for all $\alpha, \beta \in \on$
  \begin{enumerate}
    \item $\deg(\alpha \oplus \beta) \leq \max\{\deg(\alpha),\deg(\beta)\}$ (in fact, this is always an equality);
    \item $\deg(\alpha \odot \beta) = \deg(\alpha) \oplus \deg(\beta)$;
    \item $\deg(\alpha) = -\infty$ if and only if $\alpha = 0$.
  \end{enumerate}
\end{fact}

We conclude with some technical considerations about order types that will be used in \prettyref{sec:degree}. Given an ordered set $A$ and two subsets $B, C \subseteq A$ (or more generally $A, B, C$ classes), write $B < C$ if $x < y$ for all $x \in B$ and $y \in C$. Furthermore, if $A$ is a subclass of an ordered abelian group, let $B + C \coloneqq \{x + y \,:\, x \in B, y \in C\}$. We have the following.

\begin{fact}
  \label{fact:ot-ABC}
  Let $A$ be a well ordered set. Then:
  \begin{enumerate}
  \item\label{item:fact-ot-side-union} for all $\alpha, \beta \in \on$, $\ot(A) = \alpha \hatplus \beta$ if and only if there are $B, C \subseteq A$ such that $B < C$, $\ot(B) = \alpha, \ot(C) = \beta$, and $A = B \cup C$;
  \item\label{item:fact-ot-union} for all $B, C \subseteq A$, $\ot(B \cup C) \leq \ot(B) \oplus \ot(C)$ (\cite[Lem.\ 4.1]{Ber2000});
  \item\label{item:fact-ot-sum} if $A$ is a subset of an ordered abelian group, then for all $B, C \subseteq A$, $\ot(B + C) \leq \ot(B) \odot \ot(C)$ (\cite[Lem.\ 4.5]{Ber2000}).
  \end{enumerate}
\end{fact}

\subsection{Surreal numbers}
\label{sub:surreal}
For most of this paper, the reader only needs to know that $\no = \Rb((\Gbf))_\on$, where $\Gbf$ is the additive group of $\no$ itself, and the subscript $\on$ means that we only take the series the supports of which are sets, rather than proper classes; in other words, we treat $\on$ as an uncountable cardinal larger than any cardinal that is a set. We also write $\omega$ in place of $t^{-1}$, and just write
\[ \no = \Rb((\no))_\on. \]
Moreover, when proving \prettyref{prop:no-G-sigma-cofinality}, we will use that $\no$ contains an isomorphic copy of $\on$ which is unbounded in $\no$. Every result in this paper can be deduced from the above observations only.

\emph{Simplicity} will also appear as an optional ingredient in the definition of infinite product, in case one does not want to use a global axiom of choice. The reader may well skip this section and only come back to learn about simplicity or for additional background.

Here we give some minimal definitions and examples about $\no$, without proofs, in order to justify the identification $\no = \Rb((\no))_\on$ and give a flavour of the constructions in $\no$. We follow the approach of \cite{Gon1986}; all the details can be found there.

\begin{figure}[t]
  \centering
    \begin{tikzpicture}[radius=1pt,x=1.4cm,y=0.85cm]
      \tikzset{
        every node/.style={scale=0.8},
        pics/surreal/.style={code={\draw node [fill=white,inner sep=0.06cm] {\color{blue}{#1}};}},
        edge-/.style={to path={
          -- (\tikztotarget) node [midway,fill=white,draw,circle,inner sep=.03cm] {\tiny $0$}}},
        edge/.style={to path={
          -- (\tikztotarget) node [midway,fill=white,draw,circle,inner sep=.03cm] {\tiny $1$}}},
        edge--/.style={to path={
          -- (\tikztotarget) node [midway,fill=white,draw,circle,inner sep=.03cm] {\tiny $1$}}}
      }

      \coordinate (0) at (0,0);
      \foreach \s in {,-} {

        \coordinate (\s1) at (\s3/2,1);
        \coordinate (\s2) at (\s5/2,2);
        \coordinate (\s3) at (\s3,3);

        \coordinate (\s1/2) at (\s3/4,2);

        \coordinate (\s1/4) at (\s3/8,3);
        \coordinate (\s3/4) at (\s9/8,3);
        \coordinate (\s3/2) at (\s2,3);

        \coordinate (\s pi4) at (\s75/64,4.5);   
        \coordinate (\s sqrt2) at (\s61/32,4.5); 
        \coordinate (\s e) at (\s183/64,4.5);    

        \coordinate (\s omega) at (\s7/2,4.5);
        \coordinate (\s1/omega) at (\s1/8,4.5);

        \coordinate (\s omegaomega) at (\s29/8,5);
        \coordinate (\s epsilon) at (\s59/16,5.5);
        \coordinate (\s omega1) at (\s119/32,6);

        \draw [edge\s] (0) to (\s1) to (\s2) to (\s3);
        \draw [edge\s-] (\s1) to (\s1/2);
        \draw [edge\s-] (\s1/2) to (\s1/4);
        \draw [edge\s] (\s1/2) to (\s3/4);
        \draw [edge\s-] (\s2) to (\s3/2);

        \draw [dashed] (\s3/4) -- ++(\s3/16,1/2) -- ++(-\s3/32,1/4) -- (\s pi4); 
        \draw [dashed] (\s3/2) -- ++(-\s1/4,1/2) -- ++(\s1/8,1/4) -- ++(\s1/16,1/4) -- (\s sqrt2);
        \draw [dashed] (\s3) -- ++(-\s1/4,1/2) -- ++(\s1/8,1/4) -- ++(-\s1/16,1/4) -- (\s e); 
        \draw [dashed] (\s3) -- (\s omega);
        \draw [dashed] (\s1/4) -- ++(-\s1/8,1/2) -- (\s1/omega);

        \draw [dashed] (\s omega) -- (\s omegaomega) -- (\s epsilon) -- (\s omega1);

        \draw [dotted] (\s1/4) -- ++(\s1/8,1/2);
        \draw [dotted] (\s3/4) -- ++(-\s3/16,1/2);
        \draw [dotted] (\s3/2) -- ++(\s1/4,1/2);
        \draw [dotted] (\s omega) -- ++(-\s1/8,1/2);

        \draw [dotted] (\s3/4) ++(\s3/16,1/2) -- ++(\s1/8,1/4);
        \draw [dotted] (\s3/2) ++(-\s1/4,1/2) -- ++(-\s1/8,1/4);
        \draw [dotted] (\s3) ++(-\s1/4,1/2) -- ++(-\s1/8,1/4);

        \foreach \t in {,-} {
          \draw [dotted] (\s omega1) -- ++(\t\s1/64,1/2);
          \draw [dotted] (\s3/4) ++(-\s3/16,1/2) -- ++(\t\s1/16,1/4);
          \draw [dotted] (\s3/2) ++(\s1/4,1/2) -- ++(\t\s1/8,1/4);
          \draw [dotted] (\s1/4) ++(\s1/8,1/2) -- ++(\t\s1/16,1/4);
          \foreach \x in {pi4,sqrt2,e,1/omega}
            \draw [dotted] (\s\x) -- ++(\t\s1/32,1/2);
        }

        \foreach \x in {1,2,3}
          \draw (\s\x) pic{surreal={\small $\s\x$}};
        \foreach \x in {1,3}
          \foreach \y in {2,4}
              \draw (\s\x/\y) pic{surreal={\tiny $\frac{\s\x}{\y}$}};

        \draw (\s omega) pic{surreal={\small $\s\omega$}};
        \draw (\s 1/omega) pic{surreal={\small $\s\origfrac{1}{\omega}$}};
        \draw (\s sqrt2) pic{surreal={\tiny $\s\sqrt{2}$}};
        \draw (\s pi4) pic{surreal={\small $\s\frac{\pi}{4}$}};
        \draw (\s e) pic{surreal={\small $\s e$}};
        \draw (\s omegaomega) pic{surreal={\tiny $\s\omega^\omega$}};
        \draw (\s epsilon) pic{surreal={\tiny $\s\varepsilon_0$}};
        \draw (\s omega1) pic{surreal={\tiny $\s\omega_1$}};
      }
      \draw (0) pic{surreal={\small $0$}};

      \fill (-4,0) circle;
      \fill (-4,1) circle;
      \fill (-4,2) circle;
      \fill (-4,3) circle;
      \fill (-4,4.5) circle;
      \fill (-4,5) circle;
      \fill (-4,5.5) circle;
      \fill (-4,6) circle;
      \draw (-4,0) node [left] {\color{blue}{$0$}};
      \draw (-4,1) node [left] {\color{blue}{$1$}};
      \draw (-4,2) node [left] {\color{blue}{$2$}};
      \draw (-4,3) node [left] {\color{blue}{$3$}};
      \draw (-4,4.5) node [left] {\color{blue}{$\omega$}};
      \draw (-4,5) node [left] {\color{blue}{$\omega^\omega$}};
      \draw (-4,5.5) node [left] {\color{blue}{$\varepsilon_0$}};
      \draw (-4,6) node [left] {\color{blue}{$\omega_1$}};
      \draw (-4,0) -- (-4,3);
      \draw [dashed] (-4,3) -- (-4,6);
      \draw [dotted] (-4,6) -- (-4,6.5);

      \draw (-4.7,2.75) node {\rotatebox{90}{\color{red}partial, well founded order $<_s$}};
      \draw[->,color=red] (-4.5,0.5) -- (-4.5,5);
      \draw[->,color=red] (-2,-0.6) -- (2,-0.6);
      \draw (0,-1) node {\color{red}total, dense order $<$};

    \end{tikzpicture}
  \caption{The tree of surreal numbers.}
  \label{fig:surreal-numbers}
\end{figure}

A \textbf{surreal number} is a function $s : \alpha \to 2 = \{0,1\}$ for some ordinal $\alpha \in \on$. Let $\no$ denote the class of all surreal numbers. By \cite{Con1976}, it is common to represent the elements of $\no$ as the nodes of a binary rooted tree with edges labelled by $0$, $1$ and paths of any ordinal length (see Figure~\ref{fig:surreal-numbers}). Each node represents the function whose domain is the length of the path from the root to the node, and whose values are the labels on the edges of the path. For instance, the function $s : 3 = \{0,1,2\} \to \{0,1\}$ taking values $s(0) = 1$, $s(1) = 0$, $s(2) = 0$ is represented by the node $\frac{1}{4}$.

Given $b, c \in \no$ distinct surreal numbers, let $\alpha$ be the least ordinal where they differ, in the sense that either $b(\alpha) \neq c(\alpha)$, or $\alpha \notin \dom(b)$, or $\alpha \notin \dom(c)$. We say that $b$ is \textbf{less than} $c$, written $b < c$, if $b(\alpha) = 0$ or $c(\alpha) = 1$. This is a dense linear order on $\no$. In Figure~\ref{fig:surreal-numbers}, the order can be read by scanning the tree from left to right, for instance, $0 < \frac{1}{\omega} < \frac{1}{2} < 1$.

We also say that $b$ is \textbf{simpler} than $c$, written $b <_s c$, if $\dom(b) < \dom(c)$ and $b$ is the restriction of $c$ to $\dom(b)$. Write $b \leq_s c$ if $b <_s c$ or $b = c$. In Figure~\ref{fig:surreal-numbers}, $b \leq_s c$ when $b$ is on the path from the root to $c$ (for instance, $0 <_s 1 <_s \frac{1}{2} < \frac{1}{\omega}$, whereas $-1$ and $1$ are not $<_s$-comparable). This is a partial order, and it is \textbf{well founded}, namely every non-empty subclass of $\no$ contains $<_s$-minimal elements.

\begin{fact}
  \label{fact:L-R}
  If a non-empty subclass $C \subseteq \no$ is convex, namely for all $x,y \in C$ and $x < z < y$ we have $z \in C$, then the $<_s$-minimal element is unique, and we shall call it the \textbf{simplest} element of $C$. If $L, R$ are subsets of $\no$ with $L < R$, the convex class $\{b \in \no : L < b < R\}$ is non-empty, and in particular it contains a unique simplest element, denoted by $L \mid R$.
\end{fact}

For instance,
\[ 0 = \varnothing \mid \varnothing, \quad 1 = \{0\} \mid \varnothing, \quad \frac{1}{2} = \{0\} \mid \{1\}, \quad \omega = \{0,1,2,3,\dots\} \mid \varnothing. \]
Note that the above fact may fail if $L$ and $R$ are proper classes: there is no surreal $b$ such that $L < b < R$ when $L = \no^{<0}$ and $R = \no^{\geq 0}$.

Since $<_s$ is well founded, we can define various functions by induction on $<_s$. For the sake of brevity, we only show the definition of the sum. Given $b \in \no$, write $L_b \coloneqq \{ b' \in \no : b' <_s b, b' < b \}$ and $R_b \coloneqq \{ b'' \in \no : b'' <_s b, b < b'' \}$. We have $b = L_b \mid R_b$. We define, by induction on $<_s$,
\begin{align*}
  -b &\coloneqq \{ -b'' : b'' \in R_b \} \mid \{ -b' : b' \in L_b \} \\
  b + c & \coloneqq \{ b' + c, b + c' : b' \in L_b, c' \in L_c \} \mid \{ b'' + c, b + c'' : b'' \in R_b, c'' \in R_c \}
\end{align*}
To show that the definitions are well posed, one needs to ensure that the classes $L$ and $R$ respectively on the left and right of the $\mid$ symbol are sets and that $L < R$.

\begin{fact}
  The maps $(b,c) \mapsto b + c$, $b \mapsto -b$ are well defined, and make $\no$ into a divisible ordered abelian group. Moreover, there is an inductively defined map $(b, c) \mapsto bc$ making $\no$ into an ordered field.
\end{fact}

\begin{fact}
  The surreal numbers with either finite domain or with domain $\omega$ and not eventually constant form a complete real closed field, thus isomorphic to $\Rb$.

  The surreal numbers that are constantly $1$ form an isomorphic copy of $\on$ with Hessenberg's natural sum and product, among which those of finite domain form a copy of $\Nb$.
\end{fact}

In light of the above, we shall freely identify $\Rb$ and $\on$ with their isomorphic copies inside $\no$. The labels in \prettyref{fig:surreal-numbers} correspond to this identification.

\begin{fact}
  \label{fact:omega-map}
  There is an inductively defined map $x \in \no \mapsto \omega^x \in \no^{>0}$ such that for all $x, y \in \no$, $\omega^{x+y} = \omega^x\omega^y$, and if $x < y$, then $0 \leq k\omega^x < \omega^y$ for all $k \in \Nb$. Moreover, the map $\alpha \mapsto \omega^\alpha$ coincides with ordinal exponentiation in base $\omega$ when $\alpha \in \on$.
\end{fact}

One can define, by induction on the order type of a series, a natural map $\iota : \Rb((\no))_\on \to \no$, where the monomial $t^{-x}$ is mapped to $\omega^x$ for $x \in \no$. We omit its definition here.

\begin{fact}
  \label{fact:surreal-as-series}
  $\iota : \Rb((\no))_\on \to \no$ is an ordered ring isomorphism. In particular, $\no$ is a real closed field by \prettyref{fact:K((G))-field-ACF-RCF}.
\end{fact}

We shall freely identify $\no = \Rb((\no))_\on$. As mentioned in the introduction, we shall write the series with $\omega$ in place of $t^{-1}$ to emphasise that we are talking about surreal numbers, as opposed to series in some $\Kbf((\Gbf))_\kappa$.

\begin{rem}
  Under the identification $\no = \Rb((\no))_\on$, we have $\on = \Nb(-\on)$, namely the ordinal numbers are (identified with) the series with finite support, coefficients in $\Nb$, and exponents in $-\on$. The series corresponding to an ordinal is simply its Cantor normal form.
\end{rem}

\subsection{Archimedean classes}
\label{sub:archimedean}

Let $\Gbf$ be a class equipped with the structure of divisible ordered abelian group. We shall use the following notation for the Archimedean relation in $\Gbf$: for $x, y \in \Gbf$, we write
\begin{itemize}
  \item $x \preceq y$ if there is some non-zero $n \in \Nb$ such that $|x| \leq n|y|$;
  \item $x \asymp y$ if $y \preceq x \preceq y$;
  \item $x \prec y$ if $x \preceq y$ and $x \not\asymp y$.
\end{itemize}
Then $\preceq$ is a total \textbf{quasi-order} (namely, it is a reflexive, transitive, and total), thus in particular $\asymp$ is an equivalence relation. Let $[x]$ denote the $\asymp$-equivalence class of $x \in \Gbf$, and we call it the \textbf{Archimedean class} of $x$. The relation $\preceq$ induces a total order on the class of Archimedean classes $\Gbf_{/\asymp}$. In line with our previous conventions, if $\Abf, \Bbf \subseteq \Gbf$ and $x \in \Gbf$, we write $\Abf \prec x$ to mean that $y \prec x$ for all $y \in \Abf$, $\Abf \prec \Bbf$ if $y \prec z$ for all $y \in \Abf$ and $z \in \Bbf$, and give analogous meanings to $\Abf \preceq x$, $\Abf \preceq \Bbf$.

Now let $\Kbf$ be a field, $\Zbf$ be a subring of $\Kbf$, both possibly proper classes, and $\kappa$ be an uncountable cardinal or $\kappa = \on$. We can use the Archimedean classes of $\Gbf$ to reduce to the case of series with real exponents. See also \cite[\S~12]{Ber2000}, \cite{BKK2006} for considerations along these lines.

\begin{nota}
  Let $\sigma$ be an Archimedean class of $\Gbf$. Define:
  \begin{itemize}
    \item the groups $\Gbf_{\prec \sigma} \coloneqq \{x \in \Gbf : [x] \prec \sigma\}$, $\Gbf_{\preceq \sigma} \coloneqq \{x \in \Gbf : [x] \preceq \sigma\}$ for $\sigma \neq [0]$, and $\Gbf_{\prec 0} = \Gbf_{\preceq 0} \coloneqq \{0\}$.
    \item the field $\Lbf_\sigma \coloneqq \Kbf((\Gbf_{\prec \sigma}))_\kappa \subseteq \Kbf((\Gbf))_\kappa$;
    \item the ring $\Sbf_\sigma \coloneqq \Lbf_\sigma \cap (\ZKGG_\kappa) = \Zbf + \Kbf((\Gbf_{\prec \sigma}^{<0}))_\kappa \subseteq \ZKGG_\kappa$.
  \end{itemize}
\end{nota}

Note that the above groups, fields, and rings may also be proper classes.

\begin{fact}
  \label{fact:isomorphism-K-H-sigma}
  For every $\sigma \in \Gbf_{/\asymp}$ we have:
  \begin{enumerate}
    \item\label{item:fact-isom-K-H-div} $\Gbf_{\prec \sigma},\Gbf_{\preceq \sigma}$ are divisible convex subgroups of $\Gbf$;
    \item\label{item:fact-isom-K-H-complement} there is an ordered group complement $H_\sigma$ of $\Gbf_{\prec \sigma}$ into $\Gbf_{\preceq \sigma}$;
    \item\label{item:fact-isom-K-H-arch} $H_\sigma$ is Archimedean when $\sigma \neq [0]$, and $H_\sigma = \{0\}$ otherwise;
    \item\label{item:fact-isom-K-H-isom} the map $\iota_\sigma : \Kbf((\Gbf_{\preceq \sigma}))_\kappa \to \Lbf_\sigma((H_\sigma))$ defined by
    \[ \iota_\sigma : b = \sum_{x \in \Gbf_{\preceq \sigma}} b_xt^x \mapsto  \sum_{x \in H_\sigma} t^x \left(\sum_{y \in \Gbf_{\prec \sigma}} b_{x+y}t^y\right) \]
    is a field isomorphism;
    \item\label{item:fact-isom-K-H-image} the restriction of $\iota_\sigma$ to $\ZKGG_\kappa$ is a ring isomorphism with image $\SLHH$.
  \end{enumerate}
\end{fact}
The above are just standard considerations in the theory of divisible ordered abelian groups and Hahn series. Note that since $H_\sigma$ is Archimedean, it embeds into $\Rb$, and in particular it is always a set.

In the case of $\oz$, so when the group $\Gbf$ is the additive group of $\no$, we also have some additional information.

\begin{prop}
  \label{prop:H-sigma-is-R}
  For every non-zero $\sigma \in \no_{/\asymp}$, $H_\sigma \cong \Rb$.
\end{prop}
\begin{proof}
  Let $\sigma \in \no_{/\asymp}$. Note that by construction (\prettyref{fact:omega-map}), if $x < y$, then $\omega^x \prec \omega^y$; moreover, by \prettyref{fact:surreal-as-series}, there is some $x \in \no$ such that $\omega^x \in \sigma$. We may thus write
  \[ \no_{\preceq \sigma} = \left\{ \sum_{y \leq x}b_y\omega^y \right\}, \quad \no_{\prec \sigma} = \left\{ \sum_{y < x}b_y\omega^y \right\}. \]
  It is then clear that the additive group $\Rb\omega^x$ is a complement of $\no_{\prec \sigma}$ in $\no_{\preceq \sigma}$. Since all complements of $\no_{\prec \sigma}$ are isomorphic, $H_\sigma$ is isomorphic to $\Rb$.
\end{proof}

We also observe that the cofinality of a group $\Gbf$, and in particular how it compares to $\kappa$, affects the properties of the associated rings $\ZKGG_\kappa$. This will play a crucial role in \prettyref{sec:on-primal-series}. Given an ordered class $\Abf$ and a subclass $\Bbf \subseteq \Abf$, say that $\Bbf$ is \textbf{cofinal} in $\Abf$ if $\Bbf$ is not bounded from above. Call \textbf{cofinality} of $\Abf$ the least cardinality $\lambda$ of a cofinal subset, if one exists. For instance, $\Zb$ is cofinal in $\Rb$ and they both have cofinality $\aleph_0$. If no subset is cofinal, but there is a cofinal subclass of the form $\{x_\alpha \mid \alpha \in \on\}$, we say that $\Abf$ has cofinality $\on$, otherwise we say that $\Abf$ has cofinality $\infty$. By convention, we write $\alpha < \on < \infty$.

Clearly, $\Abf$ has the same cofinality as any of its cofinal subclasses. In particular, $\no$ has cofinality $\on$, since $\on$ itself is cofinal in $\no$. Likewise, since the interval $\no^{<b}$ is isomorphic to $\no$ as an ordered class for any $b \in \no$, the cofinality of $\no^{<b}$ is also $\on$.

\begin{prop}
  \label{prop:no-G-sigma-cofinality}
  For every non-zero $\sigma \in \no_{/\asymp}$, $\no_{\prec \sigma}$ has cofinality $\on$.
\end{prop}
\begin{proof}
  Recall that $\no_{\prec \sigma} = \left\{ \sum_{y < x}b_y\omega^y \right\}$ for some $x \in \no$ from the proof of \prettyref{prop:H-sigma-is-R}. It follows that $\{\omega^y : y < x\}$ is cofinal in $\no_{\prec \sigma}$, thus $\no_{\prec \sigma}$ has the same cofinality as $\no^{<x}$, and thus it has cofinality $\on$.
\end{proof}

The cofinality affects the arithmetic $\ZKGG_\kappa$ in a way that will be particularly relevant in \prettyref{sub:lifting-primal}. Below, given an integral domain $\Rbf$, let $\Frac(\Rbf)$ denote its fraction field. We thank Pietro Freni for noticing an error in an earlier version of the statement below.

\begin{prop}
  \label{prop:frac-Z-is-K}
  $\Frac(\ZKGG_\kappa) = \Kbf((\Gbf))_\kappa = t^\Gbf\KGG_\kappa$ if and only if $\Gbf$ has cofinality $\geq \kappa$, or $\Gbf = \{0\}$ and $\Frac(\Zbf) = \Kbf$. In particular, $\Frac(\oz) = \no$ (\cite[Thm.\ 32]{Con1976}).
\end{prop}
\begin{proof}
  The case $\Gbf = \{0\}$ is obvious, so assume that $\Gbf$ is a non-trivial group. Let $\lambda$ be its cofinality. By assumption, $\lambda \geq \aleph_0$.

  Suppose that $\lambda < \kappa$. Let $(x_i)_{i < \lambda}$ be a cofinal sequence in $\Gbf$. After taking a subsequence, we may assume that $x_{i+1} > 2x_i > 0$ for all $i < \lambda$ (where $2x_i = x_i + x_i$). Let $b = \sum_{i < \lambda} t^{x_i}$, and take any $c \in \ZKGG_\kappa$. By cofinality, there exists $i < \lambda$ such that $x_i > -v(c)$. It follows that for all $j,k < \lambda$ satisfying $k > \max\{i,j\}$ (thus in particular $x_k \geq 2x_i, 2x_j$) we have
  \[ \supp(t^{x_k}c) \geq x_k + v(c) > \max\{2x_j, 2x_i\} + v(c) > x_j \geq \supp(t^{x_j}c). \]
  In particular, $0 < x_k + v(c) \in \supp(bc)$, thus $bc \notin \ZKGG_\kappa$. Since $c$ was arbitrary, it follows that $b \notin \Frac(\ZKGG_\kappa)$.

  For the converse, assume that $\lambda \geq \kappa$. Let $b \in \Kbf((\Gbf))_\kappa$. Since the cardinality of $\supp(b)$ is less than $\kappa$, $\supp(b)$ is not cofinal, so there exists $x \in \Gbf^{\leq 0}$ such that $-x > \supp(b)$. It follows that $\supp(t^xb) < 0$, hence $t^xb \in \ZKGG_\kappa$. This holds for any $b$, thus $\Frac(\ZKGG_\kappa) = \Kbf((\Gbf))_\kappa = t^\Gbf\KGG_\kappa$.
\end{proof}

\subsection{Pre-Schreier domains and GCD domains}

Throughout the paper, we will use that \emph{GCD domains} are pre-Schreier domains (which are the domains with the refinement property), and that the rings $\Kbf(\Gbf)$, $\KG$ of series with finite support are GCD domains. We assume that $\Kbf$ is a field and $\Gbf$ an ordered abelian group, both possibly proper classes.

Let $R$ be an integral domain. Recall from the introduction that $b \in R$ is primal  if for all $c, d \in R$, if $b$ divides $cd$, then there are $b_1, b_2 \in R$ such that $b = b_1b_2$ and $b_1$ divides $c$, $b_2$ divides $d$. For instance, $b \in R$ is prime if and only if it is irreducible and primal.

An integral domain is pre-Schreier if and only if every element is primal (this is actually the original definition from \cite{Zaf1987}). Indeed, suppose that every element is primal, and that $bc = de$. Since $b$ is primal, we can write $b = fg$ so that $f$ divides $d$ and $g$ divides $e$. After dividing both sides, we find $c = \frac{d}{f}\,\frac{e}{g}$, and so we find the desired $f,g,h,i$ on setting $h = \frac{d}{f}$ and $i = \frac{e}{g}$. Conversely, suppose that $b$ divides a product $de$. Then we can write $bc = de$ for some $c$. Then we can write $b = fg$, where $f$ divides $d$ and $g$ divides $e$, hence $b$ is primal.

Given $b, c \in R$, we say that $d$ is a \textbf{greatest common divisor} of $b$ and $c$ if $d$ divides $b$ and $c$, and for any other $e \in R$ dividing $b$ and $c$, $e$ divides $d$. We say that $R$ is a \textbf{GCD domain} if all $b, c \in R$ have a greatest common divisor.

\begin{fact}
  \label{fact:GCD-is-pre-Schreier}
  Every GCD domain is pre-Schreier (see e.g.\ \cite[Thm.\ 2.5]{Coh1968}).
\end{fact}

\begin{fact}
  \label{fact:KG-GCD-domain}
  For all fields $K$ and ordered abelian groups $G$, $\kg$ and $K(G)$ are GCD domains, so in particular pre-Schreier domains (see e.g.\ \cite[Thm.\ 6.4]{GP1974}). The units of $\kg$ are the non-zero elements of $K$, and the units of $K(G)$ are exactly the monomials. The same facts follow for $\KG$, $\Kbf(\Gbf)$.
\end{fact}

A quick explanation is that every ordered abelian group can be written as a directed union of finitely generated free abelian groups, thus $\Kbf(\Gbf)$ is a directed union of unique factorisation domains by \prettyref{fact:KS-is-UFD}. Likewise, $\Gbf^{\leq 0}$ is a directed union of finitely generated free commutative monoids, and so $\KG$ is a directed union of unique factorisation domains. A directed union of pre-Scheier domains is clearly a pre-Schreier domain, hence so are $\Kbf(\Gbf)$ and $\KG$.

\begin{rem}
  \label{rem:shepherdson}
  In \cite{She1964}, Shepherdson not only proves that (the non-negative part of) the integer part of a real closed field is a model of Open Induction, but also exhibits the ring $\Zb + \Qb^{\textrm{rc}}(\Qb^{<0})$ as one such example, where $\Qb^{\textrm{rc}}$ is the field of real algebraic numbers. This was the first recursive non-standard model of a significant fragment of Peano's arithmetic. Using the above comments, it is easy to verify that $\Zb + \Qb^{\textrm{rc}}(\Qb^{<0})$ is a GCD domain, thus in particular pre-Schreier, and moreover its only primes are the primes of $\Zb$.
\end{rem}

\begin{rem}
  \label{rem:KGG-geq0-GCD}
  The ring $\Kbf((\Gbf^{\geq 0}))_\kappa$ is also a notable example of GCD domain, where $\kappa$ is an uncountable cardinal or $\kappa = \on$ (note that we are taking \emph{positive} exponents here). Given a non-zero $b \in \Kbf((\Gbf^{\geq 0}))_\kappa$, there are unique $k \in \Kbf$ and $\varepsilon \in \Kbf((\Gbf^{>0}))_\kappa$ such that $b = kt^{v(b)}(1 + \varepsilon)$. We then observe that $k(1 + \varepsilon)$ is invertible, with inverse
  \[ k^{-1}(1 - \varepsilon + \varepsilon^2 - \dots) \in \Kbf((\Gbf^{\geq 0}))_\kappa. \]
  The fact that the above infinite sum represents a well defined series in $\Kbf((\Gbf))_\on$ is a well known fact that goes back to \cite{Neu1949}, and if the support of $\varepsilon$ has cardinality less than $\kappa$, than so does its inverse by an easy computation. It follows that given two non-zero series $b, c \in \Kbf((\Gbf^{\geq 0}))_\kappa$, $b$ divides $c$ if and only if $v(b) \leq v(c)$. In particular, $t^{\min\{v(b),v(c)\}}$ is a greatest common divisor of $b$ and $c$.
\end{rem}

\subsection{Ritt's theorem}

Let $\Kbf$ be a field and $\Gbf$ be a divisible ordered abelian group, both possibly proper classes. The factorisation theory of $\KG$ can be described in a little more detail, after rephrasing \cite{Rit1927} to talk about series with finite support rather than exponential polynomials. See for instance \cite{EP1997} for a rephrasing of Ritt's theorem on rings of the form $R(G)$, where $R$ is unique factorisation domain and $G$ is a divisible ordered abelian group. The case of $\KG$, even just for $\Kbf$, $\Gbf$ sets, is slightly different and it seems to be missing from the literature. Since it is an important ingredient in the proof of \prettyref{main:KRR}, we give a short account here.

Given a set $S \subseteq \Gbf$, let $\langle S \rangle_\Qb$ denote its linear span over $\Qb$ in $\Gbf$, seen as a $\Qb$-vector space. This coincides with the notation used in \prettyref{main:KRR} when $\Gbf = \Rb$. First, let us formulate the main theorem of \cite{EP1997} in a manner convenient for our purposes. In order to do so, we first observe the following, where $\supp(b) - v(b)$ means $\supp(b) + (-v(b))$.

\begin{lem}
  \label{lem:KG-supp(bc)}
  Let $b, c \in \Kbf(\Gbf)$. Then
  \[ \lspan{\supp(bc) - v(bc)} = \lspan{(\supp(b) - v(b)) \cup (\supp(c) - v(c))}. \]
\end{lem}
\begin{proof}
  Note that $\supp(b) - v(b) = \supp(t^{-v(b)}b)$. Therefore, after replacing $b$ with $t^{-v(b)}b$, we may directly assume that $v(b) = v(c) = 0$, and since $v(bc) = v(b) + v(c)$, we wish to prove that $\lspan{\supp(bc)} = \lspan{\supp(b) \cup \supp(c)}$. Let $K$ be a subfield of $\Kbf$, $G$ be a divisible subgroup of $\Gbf$, both of which are sets, and such that $b, c \in K(G)$.

  For the sake of notation, let $d = bc$. Recall that by definition of product, $\supp(d) \subseteq \supp(b) + \supp(c)$, hence $\lspan{\supp(d)}$ is contained in $\lspan{\supp(b) \cup \supp(c)}$.

  For the other inclusion, let $x$ be any element of $G$ outside of $\lspan{\supp(d)}$ (if there is none, we are done). Let $H$ be a complement of $\lspan{x}$ in $G$ containing $\lspan{\supp(d)}$, which exists since $G$ is divisible. For $q \in \Qb$, set
  \[ b_q = \sum_{y \in qx + H} b_yt^y, \quad
     c_q = \sum_{y \in qx + H} c_yt^y, \quad
     d_q = \sum_{y \in qx + H} d_yt^y. \]
  Observe that by construction, $d_q = \sum_{q'+q''=q}b_{q'}c_{q''}$. Moreover, $d_q = 0$ for every $q \neq 0$, since $H \supseteq \lspan{\supp(d)}$. Finally, since $0 \in \supp(b) \cap \supp(c)$ we have at least $b_0c_0 \neq 0$.

  Now, if $p', p''$ are maximal such that $b_{p'} \neq 0, c_{p''} \neq 0$, the sum
  \[ d_{p'+p''} = \sum_{q'+q''=p'+p''}b_{q'}c_{q''}\]
  contains only one non-zero term, hence $p' + p'' = 0$. Since $p',p'' \geq 0$, we find that $p' = p'' = 0$. Likewise, if $p', p''$ are minimal such that $b_{p'} \neq 0, c_{p''} \neq 0$, we must have $p' = p'' = 0$. It follows that $b_q = c_q = 0$ for every $q \neq 0$, which implies that $x \notin \lspan{\supp(b) \cup \supp(c)}$, as desired.
\end{proof}

\begin{prop}[Ritt's theorem for $\Kbf(\Gbf)$]
  \label{prop:ritt-EP}
  Let $b \in \Kbf(\Gbf)$. There are $n \in \Nb$ and $c_1, \dots, c_n \in \Kbf(\Gbf)$ such that $b = c_1 \cdots c_n$, where each $c_i$ satisfies exactly one of the following:
  \begin{enumerate}
    \item\label{item:prop-ritt-EP-dim-2} $c_i$ is irreducible and $\langle\supp(c_i) - v(c_i)\rangle_\Qb$ has dimension $\geq 2$;
    \item\label{item:prop-ritt-EP-dim-1} $\langle\supp(c_i) - v(c_i)\rangle_\Qb$ has dimension $1$;
  \end{enumerate}
  and the supports of the $c_i$'s satisfying \prettyref{item:prop-ritt-EP-dim-1} are pairwise $\Qb$-linearly independent.

  Moreover, the factors $c_i$ as above are unique up to reordering and up to multiplication by monomials.
\end{prop}
\begin{proof}
  By \cite[Thm.\ 1]{EP1997}, $b$ can be written as a product of irreducible factors and of factors of the form $P(t^x)$, where each $P$ is a polynomial over $\Kbf$ and the exponents $x$ are pairwise $\Qb$-linearly independent. The latter ones clearly satisfy condition \prettyref{item:prop-ritt-EP-dim-1} of the conclusion.

  Now consider all the factors $c$ such that the dimension of $\lspan{\supp(c)-v(c)}$ is 1. These are exactly the factors that can be written as $t^yQ_i(t^z)$ for some non-constant polynomial $Q$ over $\Kbf$, and where $z$ is in $\lspan{\supp(c)-v(c)}$. When two such factors share the same linear space $\lspan{\supp(c)-v(c)}$, we multiply them together and obtain a single factor of the same form. We repeat this until we find a factorisation in which all factors satisfy either \prettyref{item:prop-ritt-EP-dim-2} or \prettyref{item:prop-ritt-EP-dim-1}. This shows the existence of the desired factorisation.

  We can easily prove uniqueness from the fact that $\Kbf(\Gbf)$ is a GCD domain (\prettyref{fact:KG-GCD-domain}). Indeed, all irreducible factors, in particular the ones in \prettyref{item:prop-ritt-EP-dim-2}, are prime, thus unique up to reordering and up to units. Now suppose that $c$ is a factor satisfying \prettyref{item:prop-ritt-EP-dim-1}. Suppose that $b = c_1' \cdots c_m'$ is another factorisation of $b$ of the desired form. Since $c$ is primal, we can write $c = d_1 \cdots d_m$ for some $d_1, \dots, d_m \in \Kbf(\Gbf)$ such that $d_i$ divides $c_i'$ for all $i = 1, \dots, m$. By \prettyref{lem:KG-supp(bc)}, for each $i$ the space $\lspan{\supp(d_i) - v(d_i)}$ is either $\lspan{\supp(c) - v(c)}$ or $0$ (in which case $d_i$ is a unit), and such space must be a subset of $\lspan{\supp(c_i') - v(c_i')}$. By construction, there is exactly one $i$ such that $\lspan{\supp(d_i) - v(d_i)} = \lspan{\supp(c_i') - v(c_i')}$, hence $d_i$ is equal to $c_i'$ up to multiplication by a unit. It follows easily that our desired factorisation is unique up to reordering and up to multiplication by units.
\end{proof}

It is easy to deduce the analogous statement for $\KG$.

\begin{prop}[Ritt's theorem for $\KG$]
  \label{prop:ritt}
  Let $b \in \KG$. There are $n \in \Nb$, $x \in \Gbf$, and $c_1, \dots, c_n \in \KG$ such that $b = t^xc_1 \cdots c_n$, where each $c_i$ satisfies exactly one of the following:
  \begin{enumerate}
    \item\label{item:prop-ritt-dim-2} $c_i$ is irreducible and $\langle\supp(c_i)\rangle_\Qb$ has dimension $\geq 2$;
    \item\label{item:prop-ritt-dim-1} $\langle\supp(c_i)\rangle_\Qb$ has dimension $1$ and $0 \in \supp(c_i)$;
  \end{enumerate}
  and the supports of the $c_i$'s satisfying \prettyref{item:prop-ritt-dim-1} are pairwise $\Qb$-linearly independent.

  Moreover, $x$ is unique, and the factors $c_i$ as above are unique up to reordering and up to multiplication by elements of $\Kbf$.
\end{prop}
\begin{proof}
  First, write $b = c_1' \cdots c_n'$ for some $c_i' \in \Kbf(\Gbf)$ satisfying the conclusion of \prettyref{prop:ritt-EP}. Let $c_i = t^{-\max\supp(c_i')}c_i'$, so that $c_i \in \KG$. We obtain
  \[ b = t^{\max\supp(c_1') + \dots + \max\supp(c_n')}c_1 \cdots c_n. \]
  We claim that the above factorisation is the desired one.

  First, we observe that $0 \in \supp(c_i)$ for all $i = 1, \dots, n$. It follows in particular that $\lspan{\supp(c_i)} = \lspan{\supp(c_i') - v(c_i')}$.

  Therefore, for every $i$, if $c_i'$ satisfies condition \prettyref{item:prop-ritt-dim-1} of \prettyref{prop:ritt-EP}, we have that $\lspan{\supp(c_i)}$ has dimension 1, so it satisfies condition \prettyref{item:prop-ritt-dim-1} of the desired conclusion. Moreover, such supports are pairwise $\Qb$-linearly independent.

  Now suppose instead that $c_i'$ satisfies condition \prettyref{item:prop-ritt-dim-2} of \prettyref{prop:ritt-EP}. Then $\lspan{\supp(c_i)}$ has dimension at least 2. Moreover, $c_i$ is irreducible in $\Kbf(\Gbf)$, since it is just $c_i'$ multiplied by a unit of $\Kbf(\Gbf)$. Suppose that $c_i = de$ for some $d, e \in \KG$. Then one of $d, e$ is a unit of $\Kbf(\Gbf)$, say $d = kt^x$ for some $k \in \Kbf$ and $x \in \Gbf$. But $0 = \max\supp(c_i) = \max\supp(d) + \max\supp(e) \leq 0$, hence $x = 0$. Therefore, $c_i$ is irreducible in $\KG$.

  For the uniqueness, suppose we are given a factorisation as in the conclusion. If $c_i$ is irreducible, then $0 \in \supp(c_i)$, or it would be non-trivially divisible by $t^x$ for any $\supp(c_i) < x < 0$. Therefore, $0 \in \supp(c_i)$ for all $i = 1, \dots, n$.

  In particular, $x = \max\supp(b)$, hence it is unique. Moreover, by \prettyref{prop:ritt-EP}, the factors $c_i$ are unique up to reordering and multiplication by monomials. But since $0 \in \supp(c_i)$ for every $i = 1, \dots, n$, the only monomials by which we can multiply are elements of $\Kbf$, as desired.
\end{proof}

\begin{rem}
  \label{rem:factorisations-of-fractional-polynomials}
  The factorisations of the series $c_i$ such that $\lspan{\supp(c_i)}$ has dimension 1 and $0 \in \supp(c_i)$ depend on the arithmetic of the underlying field $\Kbf$. In some cases, $c_i$ can be irreducible, in others it may not have any irreducible factors. We illustrate this with a few examples.

  Fix one such series $c = c_i$ and let $x$ be a negative generator of $\lspan{\supp(c)}$. Then $c \in \Kbf(\lspan{x}^{\leq 0}) \cong \bigcup_N \Kbf[X^{\frac{1}{N}}]$. For simplicity, let us assume that $c \in 1 + \Kbf(\Gbf^{<0})$, and that we only look at factorisations into factors in $1 + \Kbf(\Gbf^{<0})$. Any other factorisation of $c$ can be transformed into one like this after multiplying by appropriate units, namely non-zero elements of $\Kbf$.

  When $\Kbf$ is algebraically closed, we can split $c$ into finitely many factors of the form $1 - at^y$ where $a \in \Kbf$ is non-zero and $y \in \lspan{x}$ is negative. Every factorisation of $1 - at^y$ lies in some $\Kbf\left(\frac{1}{N}\mathbb{N}y\right)$, hence it can be refined to a product of series of the form $1 - \zeta_Na^{\frac{1}{N}}t^{\frac{y}{N}}$, where $a^{\frac{1}{N}}$ is an $N$-th root of $a$ and $\zeta_N$ is an $N$-th root of unity. In turn, every factorisation of $c$ can be refined into a product of factors as above.

  When $\Kbf$ is real closed, one can obtain a similar description. Recall that $\Kbf$ becomes algebraically closed after adding a square root of $-1$. Given a factorisation of $c_i$, one can go to the algebraic closure, refine it into a product of series of the form $1 - \zeta_Na^{\frac{1}{N}}t^{\frac{y}{N}}$, then multiply the pairs of factors that are Galois conjugate over $\Kbf$ to find a refinement into factors in $\KG$ of the forms $1 - at^y$ or $1 - at^y + bt^{2y}$.

  At the other extreme, consider for instance $\Kbf = \Qb$. In this case, $c$ has a maximal divisor of the form $Q(t^x)$, where $Q \in \Qb[X]$ is a divisor of $X^N - 1$ for some $N \in \Nb$. One can verify (for instance via Kummer theory) that $\frac{c_i}{Q(t^x)}$ admits a factorisation into irreducibles. On the other hand, $Q(t^x)$ admits no irreducible factors by simple Galois theoretic considerations on roots of unity.
\end{rem}

\subsection{Berarducci's order value}
\label{sub:order-value}
Let us recall the definition of the function $v_J$ from \cite{Ber2000}. It will only be used directly in the proof of \prettyref{main:prime}, so the reader may skip this subsection and come back when needed. Let $\Kbf$ be a field of characteristic 0, possibly a proper class. Let $\Jbf$ be the (proper) ideal of $\KRR$ generated by the series of the form $t^x$ for $x < 0$. Given $b \in \KRR$, the \textbf{order value} of $b$ is
\begin{equation}
  \label{eq:vJ}
  v_J(b) \coloneqq \begin{cases}
    0 & \text{if } b \in \Jbf, \\
    1 & \text{if } b \in (\Jbf + \Kbf) \setminus \Jbf, \\
    \min\{\ot(c) : c \equiv b \mod \Jbf + \Kbf\} & \text{otherwise}.
  \end{cases}
\end{equation}

The key and by far most difficult result of \cite{Ber2000} is that $v_J$ is a multiplicative semi-valuation.

\begin{fact}[{\cite[Lem.\ 5.5,\ Thm.\ 9.7]{Ber2000}}]
  \label{fact:order-value}
  For all $b, c \in \KRR$ we have:
  \begin{enumerate}
  \item\label{item:vJ-ultra} $v_J(b + c) \leq \max\{v_J(b), v_J(c)\}$;
  \item\label{item:vJ-mult} $v_J(bc) = v_J(b) \odot v_J(c)$;
  \item\label{item:vJ-infty} $v_J(b) = 0$ if and only if $b \in \Jbf$.
  \end{enumerate}
\end{fact}

\begin{rem}
  \label{rem:vJ-deg-inequality}
  Note that by construction $v_J(b) \leq \ot(b)$. Since $v_J(b)$ is always of the form $\omega^\alpha$ for some $\alpha \in \on_\infty$ (see \cite[Rem.\ 5.3]{Ber2000}), we find that $v_J(b) \leq \omega^{\deg(b)}$. This inequality is often strict. For instance, the reader can easily verify directly from the definitions that
  \[ \sum_{(n,m) \in \Nb^2} t^{-1 - \frac{1}{(n+1)(m+1)}} + \sum_{n \in \Nb} t^{-\frac{1}{n+1}} \]
  has degree $2$ and order value $\omega$.
\end{rem}

\section{The degree is a multiplicative valuation}
\label{sec:degree}

We are now ready to tackle the main theorems of this paper. From this point forward, we work with the following data:
\begin{itemize}
  \item $\Kbf$ a field of characteristic $0$, possibly a proper class;
  \item $\Gbf$ a divisible ordered abelian group, possibly a proper class;
  \item $\kappa$ an uncountable cardinal, or $\kappa = \on$.
\end{itemize}

The purpose of this section is to prove \prettyref{main:degree}, while collecting additional useful observations along the way.

\subsection{Preliminary inequalities}
We first establish some easy inequalities satisfied by the degree. We also introduce a sibling function $\sup : \KRR \to \Rb_\infty = \Rb \cup \{-\infty\}$, satisfying similar inequalities.

\begin{prop}[{\cite[Remark 5.4]{Ber2000}}]
  \label{prop:ot}
  For all $b, c \in \Kbf((\Gbf))_\kappa$ we have:
  \begin{enumerate}
  \item $\ot(b + c) \leq \ot(b) \oplus \ot(c)$;
  \item $\ot(bc) \leq \ot(b) \odot \ot(c)$;
  \end{enumerate}
\end{prop}
\begin{proof}
  Let $b, c \in \Kbf((\Gbf))_\kappa$. Note that $\supp(b + c) \subseteq \supp(b) \cup \supp(c)$ and $\supp(bc) \subseteq \supp(b) + \supp(c)$. The inequalities then follow at once from \prettyref{fact:ot-ABC}\prettyref{item:fact-ot-union}--\prettyref{item:fact-ot-sum}.
\end{proof}

\begin{cor}
  \label{cor:deg}
  For all $b, c \in \Kbf((\Gbf))_\kappa$ we have:
  \begin{enumerate}
  \item\label{item:deg-ultra-early} $\deg(b + c) \leq \max\{\deg(b), \deg(c)\}$;
  \item\label{item:deg-submult} $\deg(bc) \leq \deg(b) \oplus \deg(c)$.
  \end{enumerate}
\end{cor}

The above inequalities, combined with \cite{LM2017}, are already enough to prove \prettyref{main:degree}. Indeed, note that $\omega^{\deg(b)}$ is the first term of the Cantor Normal Form of $\ot(b)$. Thus, by \cite[Cors.\ 4.3,\ 4.7]{LM2017}, we have $\omega^{\deg(bc)} = \omega^{\deg(b)}\odot\omega^{\deg(c)}$, which implies $\deg(bc) = \deg(b) \oplus \deg(c)$. However, we prefer to give a more self-contained proof.

\begin{defn}
  \label{def:sup}
  Given $b \in \KRR^{\neq 0}$, let the \textbf{supremum} of $b$ be $\sup(b) \coloneqq \sup(\supp(b))$. We also let $\sup(0) \coloneqq -\infty$.
\end{defn}

\begin{prop}
  \label{prop:sup}
  For all $b, c \in \KRR$ we have:
  \begin{enumerate}
  \item\label{item:sup-ultra-early} $\sup(b + c) \leq \max \{\sup(b), \sup(c)\}$;
  \item\label{item:sup-submult} $\sup(bc) \leq \sup(b) + \sup(c)$.
  \end{enumerate}
\end{prop}
\begin{proof}
  Immediate from the definitions of sum and product.
\end{proof}

Thus, for \prettyref{main:degree}, we need to prove that $\deg(bc) = \deg(b) \oplus \deg(c)$ for all $b, c \in \KRR$. Along the way we also prove that $\sup(bc) = \sup(b) + \sup(c)$, thus both functions are multiplicative valuations on $\KRR$.

\begin{rem}
  \label{rem:monoid-is-omega-1}
  For $\Gbf = \Rb$, the maps $\ot$ and $\deg$ take values among the \emph{countable} ordinals. Indeed, each well ordered subset of $\Rb$ is countable (because, for instance, we can pick a distinct rational number within each pair of successive elements of such a set, thus constructing an injective map from the set to the rational numbers). The set of countable ordinals is denoted by $\omega_1$, which is itself an ordinal (the least uncountable one).

  Conversely, every countable ordinal is represented by a well ordered subset of $\Rb$ (or even $\Rb^{\leq 0}$), so every ordinal in $\omega_1$ is the order type of some series in $\Kbf((\Rb))$ (or $\KRR$). One can easily deduce that the same is true for the degree.
\end{rem}

\subsection{Truncations and principal series}

We now translate \prettyref{fact:ot-ABC}\prettyref{item:fact-ot-side-union} into a few statements about series.

\begin{prop}
  \label{prop:b-c+d}
  For all $b \in \Kbf((\Gbf))_\kappa$ and $\alpha, \beta \in \on$, $\ot(b) = \alpha \hatplus \beta$ if and only if there are $c, d \in \Kbf((\Gbf))_\kappa$ such that $\supp(c) < \supp(d)$, $\ot(c) = \alpha$, $\ot(d) = \beta$ and $b = c + d$.
\end{prop}
\begin{proof}
  Let $A = \supp(b)$. By \prettyref{fact:ot-ABC}\prettyref{item:fact-ot-side-union}, $\ot(A) = \alpha \hatplus \beta$ if and only if there are $B, C \subseteq A$ such that $B < C$, $\ot(B) = \alpha$, $\ot(C) = \beta$, and $A = B \cup C$. Write $b = \sum_{x \in \Gbf} b_x t^x$. Then clearly
  \[ b = \sum_{x \in \Gbf} b_x t^x = \sum_{x \in A} b_x t^x = \sum_{x \in B} b_x t^x + \sum_{x \in C} b_x t^x. \]
  The conclusion follows at once on setting $c = \sum_{x \in B} b_x t^x$, $d = \sum_{x \in C} b_x t^x$: indeed, $\supp(c) = B$, $\supp(d) = C$, so $\supp(c) < \supp(d)$,  $\ot(c) = \alpha$, $\ot(d) = \beta$ and $b = c + d$.
\end{proof}

When a series $b$ is written as a sum $b = c + d$ as in the above proposition, we shall call $c$ and $d$ \emph{truncations} of $b$. We will use the following notation.

\begin{defn}
  \label{def:truncation}
  Given $b = \sum_x b_x t^x \in \Kbf((\Gbf))_\kappa$ and $y \in \Gbf$, we let the \textbf{truncations} of $b$ at $y$ be:
  \[ b_{\leq y} \coloneqq \sum_{x \leq y} b_x t^x, \quad b_{< y} \coloneqq \sum_{x < y} b_x t^x, \quad b_{\geq y} \coloneqq \sum_{x \geq y} b_x t^x, \quad b_{> y} \coloneqq \sum_{x > y} b_x t^x. \]
\end{defn}

Note that by construction, $b = b_{<y} + b_{\geq y} = b_{\leq y} + b_{> y}$, as well as $\ot(b) = \ot(b_{<y}) \hatplus \ot(b_{\geq y}) = \ot(b_{\leq y}) \hatplus \ot(b_{>y})$. In particular, $\ot(b_{\leq y}) < \ot(b)$ as soon as $b_{\leq y} \neq b$, and consequently $\deg(b_{\leq y}) \leq \deg(b)$.

\subsection{Normal forms}

Given a series $b \in \Kbf((\Gbf))_\kappa$, we can use the Cantor Normal Form of $\ot(b)$ to give a special decomposition of $b$.

\begin{defn}
  \label{def:principal}
  An ordinal $\alpha \in \on$ is \textbf{additively principal} if $\alpha = \omega^{\beta}$ for some $\beta \in \on$; equivalently, by \prettyref{eq:sum-in-cnf} on page~\pageref{eq:sum-in-cnf}, if $\alpha = \beta \hatplus \gamma$ implies $\beta = 0$ or $\gamma = 0$ for all $\beta, \gamma \in \on$, and $\alpha \neq 0$.

  A series $b \in \Kbf((\Gbf))_\kappa$ is \textbf{weakly principal} if $\ot(b)$ is additively principal.

  A series $b \in \KRR$ is \textbf{principal} if it is weakly principal and $\sup(b) = 0$.
\end{defn}

\begin{defn}
  Given $b \in \Kbf((\Gbf))_\kappa$, we call the sum
  \[ b = b_1 + \dots + b_n \]
  the \textbf{weak normal form} of $b$ when:
  \begin{itemize}
  \item $b_i$ is weakly principal for all $i = 1, \dots, n$;
  \item $\ot(b_1) \geq \dots \geq \ot(b_n)$;
  \item $\supp(b_1) < \dots < \supp(b_n)$.
  \end{itemize}
\end{defn}

\begin{rem}
  Note that $\ot(b) = \ot(b_1) \hatplus \dots \hatplus \ot(b_n)$; since each
  $\ot(b_i)$ is additively principal, this sum is by definition the Cantor
  normal form (\prettyref{eq:cantor-normal-form}) of $\ot(b)$.
\end{rem}

\begin{prop}
  \label{prop:weak-normal-form}
  For all series $b \in \Kbf((\Gbf))_\kappa$ there exists a unique weak normal form.
\end{prop}
\begin{proof}
  Let $b \in \Kbf((\Gbf))_\kappa$. There is a unique way of writing $\ot(b) = \omega^{\beta_1} \hatplus \dots \hatplus \omega^{\beta_n}$ with $\beta_1 \geq \dots \geq \beta_n$. By iterating \prettyref{prop:b-c+d}, there are unique $b_1, \dots, b_n$ such that $b = b_1 + \dots + b_n$, $\ot(b_i) = \omega^{\beta_i}$, and $\supp(b_i) < \supp(b_{i+1})$. Conversely, for any weak normal form $b = b_1 + \dots + b_n$, the sum $\ot(b) = \ot(b_1) \hatplus \dots \hatplus \ot(b_n)$ must be the Cantor normal form of $\ot(b)$. Therefore, the weak normal form is unique.
\end{proof}

\begin{cor}
  \label{cor:weakly-principal-tail}
  For all non-zero $b \in \Kbf((\Gbf))_\kappa$ there exists $x \in \Gbf$ such that $b_{\geq x}$ is non-zero and weakly principal.
\end{cor}
\begin{proof}
  Given a non-zero $b \in \Kbf((\Gbf))_\kappa$, it suffices to write $b = b_1 + \dots + b_n$ in weak normal form and let $x = v(b_n)$. Then $b_{\geq x} = b_n$ is non-zero and weakly principal.
\end{proof}

When $\Gbf = \Rb$, we can also use principal series.

\begin{defn}
  \label{def:normal-form}
  Given $b \in \Kbf((\Rb))$, we call the sum
  \[ b = b_1t^{x_1} + \dots + b_nt^{x_n} \]
  the \textbf{normal form} of $b$ when:
  \begin{itemize}
  \item $x_1 \leq \dots \leq x_n$;
  \item $b_i$ is principal for all $i = 1, \dots, n$;
  \item $\ot(b_1) \geq \dots \geq \ot(b_n)$;
  \item $\supp(t^{x_i}b_i) < \supp(t^{x_{i+1}}b_{i+1})$ for all $i = 1, \dots, n-1$.
  \end{itemize}
\end{defn}

\begin{prop}
  \label{prop:normal-form}
  For all series $b \in \KRR$ there exists a unique normal form.
\end{prop}
\begin{proof}
  Let $b \in \KRR$. By \prettyref{prop:weak-normal-form}, we can write uniquely $b = b_1' + \dots + b_n'$ in weak normal form. It now suffices to define $x_i = \sup(\supp(b_i'))$ and $b_i = t^{-x_i}b_i'$. The conclusion then follows trivially.
\end{proof}

\begin{rem}
  \label{rem:P-not-a-ring}
  \prettyref{prop:normal-form} shows that any series in $\KRR$ can be intuitively thought as a series with coefficients in the set of principal series, exponents in $\Rb$, and \emph{finite support}. However, principal series do not form a ring, because the sum of two principal series may not be principal. For instance $\sum_{n \in \Nb} t^{\frac{-1}{n+1}}$ and $t^{-1} + t^{\frac{-1}{2}} - \sum_{n \in \Nb} t^{\frac{-1}{n+1}}$ are both principal (both have order type $\omega$ and support with supremum $0$), but their sum is $t^{-1} + t^{\frac{-1}{2}}$ which has order type $2$ and the supremum of its support is $-\frac{1}{2}$.
\end{rem}

\begin{rem}
  Note that the above proof uses not only that $\Rb$ is complete, but also that $\supp(b_i')$ is bounded from above for $b_i' \in \KRR$. In fact, arbitrary series in $\Kbf((\Rb))$ may not have a normal form. For instance
  \[ b = 1 + t + t^2 + t^3 + \dots \]
  is weakly principal, but there are no $x$ and $b'$ such that $b = t^xb'$ and $b'$ is principal, because such an $x$ would be an upper bound of $\supp(b)$, which is clearly unbounded.
\end{rem}

\begin{rem}
  Suppose that $\alpha$ is order type of a series $b \in \KRR$, with Cantor Normal Form $\alpha = \omega^{\beta_1} \hatplus \cdots \hatplus \omega^{\beta_n}$, and write $b = b_1t^{x_1} + \dots + b_nt^{x_n}$ in normal form. Let $k \leq n$ be the largest integer such that $\beta_k > 0$, or $k = 0$ if no such $k$ exists.

  One can easily verify that
  \[ x_1 < \dots < x_k \leq x_{k+1} < \dots < x_n. \]
  Indeed, $x_i \leq v(b_{i+1}) + x_{i+1} < x_{i+1}$ for $i < k$, while $x_i = x_i + v(b_i) < x_{x+1} + v(b_{i+1}) = x_{i+1}$ for $i > k$. In particular, if $k = n$ or $k = 0$, then the exponents are all distinct. Otherwise, the equality may indeed occur. For instance, the series
  \[ 1 + \sum_{n \in \Nb} t^{\frac{-1}{n+1}} \]
  has normal form
  \[ \left(\sum_{n \in \Nb} t^{\frac{-1}{n+1}}\right)t^0 + t^0. \]
  Here the order type is $\omega \hatplus 1$ and the exponent $0$ appears twice in the normal form.
\end{rem}

\subsection{Multiplicativity of the degree}

We are finally ready to prove \prettyref{main:degree}. First, we recall that the order type function $\ot$ is multiplicative on weakly principal series.

\begin{fact}[{\cite[Corollary 9.9]{Ber2000}}]
  \label{fact:berarducci}
  If $b, c \in \KRR$ are weakly principal, then $\ot(bc) = \ot(b) \odot \ot(c)$.
\end{fact}

One can prove this fact directly from \prettyref{fact:order-value} and \prettyref{prop:ot}, together with the observation that for a principal series $b$ we have $\ot(b) = \vj(b)$. In particular, the product of two weakly principal series $b, c$ is weakly principal, because $\ot(bc)$ is also of the form $\omega^\alpha$ for some ordinal $\alpha$, and moreover, $\deg(bc) = \deg(b) \oplus \deg(c)$. We now leverage this fact to prove that the latter equality holds for every $b, c \in \KRR$.

\begin{lem}
  \label{lem:deg-bc-x}
  Let $b, c \in \KRR$ and $x \in \Rb$. Suppose that $b$ is principal and one of the following holds:
  \begin{enumerate}
    \item\label{item:lem-deg-supp>x} $\supp(c) > x$, or
    \item\label{item:lem-deg-supp-geq-x} $\supp(c) \geq x$ and $\deg(c) > 0$.
  \end{enumerate}
  Then $\deg((bc)_{\leq x}) < \deg(b) \oplus \deg(c)$.
\end{lem}
\begin{proof}
  Let $b$, $c$, $x$ satisfy the hypothesis of case \prettyref{item:lem-deg-supp>x}. Let $y$ be such that $x < y < \supp(c)$ and let $d = b_{\leq x-y}$, $e = b_{>x - y}$, so that $\supp(d) \leq x - y < \supp(e)$. Then $\supp(ce) > y + (x - y) = x$. It follows that $(bc)_{\leq x} = (cd + ce)_{\leq x} = (cd)_{\leq x}$. Therefore,
  \[ \deg((bc)_{\leq x}) = \deg((cd)_{\leq x}) \leq \deg(cd) \leq \deg(c) \oplus \deg(d). \]
  Recall that $\sup(\supp(d)) \leq x - y < 0$. Since $b$ is
  principal, we must have $b \neq d$, so $e \neq 0$, so $\ot(e) > 0$. Since
  $\ot(b) = \ot(d) \hatplus \ot(e)$, it follows that $\ot(d) < \ot(b)$. On the other
  hand, $\ot(b)$ is of the form $\omega^\alpha$, so $\deg(d) < \deg(b)$. In turn,
  $\deg((bc)_{\leq x}) < \deg(b) \oplus \deg(c)$, as desired.

  Now let $b$, $c$, $x$ satisfy \prettyref{item:lem-deg-supp-geq-x}. Write $c = kt^x + c'$, where $c' \neq 0$, $k \in \Kbf$, and $\supp(c') > x$. Then $bc = bkt^x + bc'$. By case \prettyref{item:lem-deg-supp>x}, $\deg((bc')_{\leq x}) < \deg(b) \oplus \deg(c') \leq \deg(b) \oplus \deg(c)$. Moreover, $\deg(bkt^x) \leq \deg(b) < \deg(b) \oplus \deg(c)$. Therefore, $\deg((bc)_{\leq x}) < \deg(b) \oplus \deg(c)$ by \prettyref{cor:deg}\prettyref{item:deg-ultra-early}.
\end{proof}

\begin{prop}
  \label{prop:degree}
  For all $b, c \in \KRR$, $\deg(bc) = \deg(b) \oplus \deg(c)$.
\end{prop}
\begin{proof}
  Let $b, c \in \KRR$. We may assume that $b$, $c$ are both non-zero, otherwise the conclusion is trivial. By \prettyref{prop:normal-form}, we can write $b = b_1t^x + b'$ with $b_1$ principal, $\supp(b') \geq x$ and $\deg(b_1) = \deg(b) \geq \deg(b')$. Moreover, if $\deg(b) = 0$, we may further assume that $\supp(b') > x$. Similarly, we may write $c = c_1t^y + c'$ with $c_1$ principal, $\supp(c') \geq y$, $\deg(c_1) = \deg(c) \geq \deg(c')$, and if $\deg(c) = 0$, then $\supp(c') > y$.

  Trivially, we have
  \[ bc = (b_1t^x + b')(c_1t^y + c') = b_1c_1t^{x+y} + b_1c't^x + b'c_1t^y + b'c'. \]
  After truncating both sides at $x+y$, we get
  \[ (bc)_{\leq x+y} = b_1c_1t^{x+y} + (b_1c')_{\leq y}t^x + (b'c_1)_{\leq x}t^y + (b'c')_{\leq x+y}. \]

  By \prettyref{fact:berarducci}, $\deg(b_1c_1t^{x+y}) = \deg(b_1c_1) = \deg(b_1) \oplus \deg(c_1) = \deg(b) \oplus \deg(c)$. We claim that the other three summands on the right-hand side have smaller degree.

  For the fourth summand, note that $\supp(b'c') \geq x+y$, so $(b'c')_{\leq x+y} \in \Kbf$, which means that its degree is at most $0$. If $\deg(b) > 0$ or $\deg(c) > 0$, then $\deg(b) \oplus \deg(c) > 0$, and we are done. If instead $\deg(b) = \deg(c) = 0$, then actually $\supp(b'c') > x+y$, so $(b'c')_{\leq x+y} = 0$, which has degree $-\infty < 0 = \deg(b) \oplus \deg(c)$, and we are done.

  For the second and third summand, we distinguish two cases. If $\deg(c') < \deg(c)$, then in fact $\deg((b_1c')_{\leq y}) \leq \deg(b_1) \oplus \deg(c') < \deg(b_1) \oplus \deg(c')$, in which case we are done. By symmetry, the same holds for the third summand when $\deg(b') < \deg(b)$.

  Now assume $\deg(c') = \deg(c)$. If $\deg(c) > 0$, then by \prettyref{lem:deg-bc-x}\prettyref{item:lem-deg-supp-geq-x} we have $\deg((b_1c')_{\leq y}) < \deg(b_1) \oplus \deg(c') = \deg(b) \oplus \deg(c)$, and we are done. If $\deg(c) = 0$, then $\supp(c') > y$, so by \prettyref{lem:deg-bc-x}\prettyref{item:lem-deg-supp>x} we have $\deg((b_1c')_{\leq y}) < \deg(b_1) \oplus \deg(c') = \deg(b) \oplus \deg(c)$, and we are done. Again, by symmetry the same holds for the third summand when $\deg(b') = \deg(b)$, concluding the proof of the claim.

  Therefore, by \prettyref{cor:deg}\prettyref{item:deg-ultra-early}, $\deg((bc)_{\leq x + y}) = \deg(b_1c_1t^{x+y}) = \deg(b) \oplus \deg(c)$. Since $\deg((bc)_{\leq x+y}) \leq \deg(bc) \leq \deg(b) \oplus \deg(c)$, it follows that $\deg(bc) = \deg(b) \oplus \deg(c)$, as desired.
\end{proof}

\begin{proof}[Proof of \prettyref{main:degree}]
  The inequality \prettyref{item:main-deg-ultra} $\deg(b + c) \leq \max\{\deg(b), \deg(c)\}$ is \prettyref{cor:deg}\prettyref{item:deg-ultra-early}; the multiplicativity \prettyref{item:main-deg-mult} is the conclusion of \prettyref{prop:degree}. Finally, the conclusion \prettyref{item:main-deg-infty} that $\deg(b) = -\infty$ if and only if $b = 0$ is an immediate consequence of the definition of degree.
\end{proof}

\begin{rem}
  \label{rem:deg-not-mult}
  The map $b \mapsto \deg(\ot(b))$ is well defined on $\KGG_\kappa$ for any group $\Gbf$, but it fails to be multiplicative as soon as $\Gbf$ does not embed into $\Rb$. Indeed, let $x, y \in \Gbf^{<0}$ be such that $y < (n+1)x < 0$ for all $n \in \Nb$, or in other words $x \prec y$ in the notation of \prettyref{sub:archimedean}. Such $x$, $y$ exist exactly when $\Gbf$ does not embed into $\Rb$. Then
  \[ b = t^y \cdot \sum_{n \in \Nb} t^{-nx}, \quad c = 1 - t^x \]
  are both series in $\KGG_\kappa$, with $bc = t^y$, but $\deg(\ot(bc)) = 0 \neq \deg(\ot(b)) \oplus \deg(\ot(c)) = 1$.
\end{rem}

\subsection{Multiplicativity of the supremum}

Another consequence of Berarducci's \prettyref{fact:berarducci} is that the supremum is a multiplicative valuation. It can be easily deduced from the fact that the ideal $\Jbf$ generated by all the monomials $t^x$ for $x \in \Rb^{<0}$ is prime \cite[Cor.\ 9.8]{Ber2000}. Here we include an alternative proof using the results presented in this section.

\begin{prop}
  \label{prop:sup-valuation}
  For all $b, c \in \KRR$ we have:
  \begin{enumerate}
  \item\label{item:sup-ultra} $\sup(b + c) \leq \max\{\sup(b), \sup(c)\}$;
  \item\label{item:sup-mult} $\sup(bc) = \sup(b) + \sup(c)$;
  \item\label{item:sup-infty} $\sup(b) = -\infty$ if and only if $b = 0$.
  \end{enumerate}
\end{prop}
\begin{proof}
  Let $b, c \in \KRR$. \prettyref{item:sup-infty} is a straightforward consequence of the definition of $\sup$, so we shall assume that $b, c \neq 0$. \prettyref{prop:sup} already shows \prettyref{item:sup-ultra}, so we only have to prove \prettyref{item:sup-mult}.

  We first prove the conclusion for $b, c$ principal. In this case, $\sup(b) = \sup(c) = 0$, so we need to show that $\sup(bc) = 0$. Fix some $x < 0$. Then
  \[ (bc)_{\leq 2x} = (b_{\leq x}c_{\leq x} + b_{> x}c_{\leq x} + b_{\leq x}c_{> x} + b_{> x}c_{> x})_{\leq 2x}. \]
  Since $\supp(b_{> x}c_{> x}) > 2x$, we have $(bc)_{\leq 2x} = (b_{\leq x}c_{\leq x} + b_{> x}c_{\leq x} + b_{\leq x}c_{> x})_{\leq 2x}$.  Therefore,
  \[ \ot((bc)_{\leq 2x}) \leq \ot(b_{\leq x}c_{\leq x}) \oplus \ot(b_{> x}c_{\leq x}) \oplus \ot(b_{\leq x}c_{> x}). \]
  Since $b, c$ are principal, the three summands are strictly less than $\ot(b) \odot \ot(c)$; since the ordinal $\ot(b) \odot \ot(c)$ is additively principal, it follows that
  \[ \ot(bc_{\leq 2x}) < \ot(b) \odot \ot(c) = \ot(bc), \]
  where the latter equality is given by \prettyref{fact:berarducci}. This implies that $bc \neq bc_{\leq 2x}$ for all $x < 0$, so in particular $\sup(bc) = 0$, as desired.

  For the general case, write $b = b_1t^{x_1} + \dots + b_nt^{x_n}$ and $c = c_1t^{y_1} + \dots + c_mt^{y_m}$ in normal form. Note that $\sup(b) = x_n$ and $\sup(c) = y_m$. By \prettyref{prop:sup}\prettyref{item:sup-submult},
  \begin{align*}
    \sup(b(c_1t^{y_1} + \dots + c_{m-1}t^{y_{m-1}})) &\leq \sup(b) + y_{m-1} < \sup(b) + \sup(c) \\
    \sup((b_1t^{x_1} + \dots + b_{n-1}t^{x_{n-1}})c) &\leq x_{n-1} + \sup(c) < \sup(b) + \sup(c).
  \end{align*}
  On the other hand, the previous case shows that $\sup(b_nc_m) = 0$, hence
  \[ \sup(b_nt^{x_n}c_mt^{x_m}) = x_n + y_m + \sup(b_nc_m) = x_n + y_m = \sup(b) + \sup(c). \]
  By \prettyref{prop:sup}\prettyref{item:sup-ultra-early}, it follows that
  \[ \sup(bc) = \sup(b_nt^{x_n}c_nt^{y_m}) = \sup(b) + \sup(c) \qedhere \]
\end{proof}

\begin{rem}
  \label{rem:sup-G}
  Given an arbitrary group $\Gbf$ and $b \in \KGG_\kappa$, one may define $\sup'(b)$ as the supremum of $\supp(b)$ in the Dedekind-MacNeille completion of $\Gbf$ (with additional care if $\Gbf$ is a proper class). This map fails to be multiplicative: using the same $\Gbf$, $b$, $c$ of \prettyref{rem:deg-not-mult}, we have $\sup'(b) + \sup'(c) = \sup'(b) > y-x > y = \sup'(t^y) = \sup'(bc)$. One can make $\sup'$ into a multiplicative valuation by taking a suitable quotient of the completion of $\Gbf$. We detail this in \prettyref{prop:sup-G}, even though it is not used in this paper.
\end{rem}

\subsection{Consequences for principal series}

We can now show that the set of principal series is closed under product, and closed under at least some sums.

\begin{prop}
  \label{prop:berarducci}
  If $b, c \in \KRR$ are principal, then $bc$ is principal.
\end{prop}
\begin{proof}
  Let $b, c \in \KRR$ be principal series. By \prettyref{fact:berarducci}, $bc$ is weakly principal, and by \prettyref{prop:sup-valuation}\prettyref{item:sup-mult}, $\sup(bc) = \sup(b) + \sup(c) = 0$, so $bc$ is principal.
\end{proof}

\begin{prop}
  \label{prop:sum-principal}
  For all $b, c \in \KRR$, if $b$, $c$ are principal and $\deg(b + c) = \deg(b)$, then $b + c$ is principal.
\end{prop}
\begin{proof}
  Let $b, c \in \KRR$. If one of $b$, $c$ is zero, the conclusion is trivial, so we may assume that they are both non-zero. Assume that $\deg(b) = \deg(b + c) = \alpha$ for some $\alpha \in \omega_1$. Note in particular that $\deg(c) \leq \alpha$ by the ultrametric inequality \prettyref{cor:deg}\prettyref{item:deg-ultra-early}. If $\alpha = 0$, then $b, c \in \Kbf$, so $b + c \in \Kbf$, and $b + c \neq 0$ since it has degree $0$, so $b + c$ is principal, as desired. Now assume $\alpha > 0$.

  Fix some $x \in \Rb^{<0}$. Since $b, c$ are principal, we have $\sup(b) = \sup(c) = 0$. In particular, $\deg(b_{\leq x}), \deg(c_{\leq x}) < \alpha$, because $\ot(b_{\leq x} + b_{>x}) = \ot(b_{\leq x}) \hatplus \ot(b_{>x})$, from which it follows that $\ot(b_{\leq x}) < \omega^\alpha$ (likewise for $c$). Therefore, $\deg((b + c)_{\leq x}) = \deg(b_{\leq x} + c_{\leq x}) < \alpha$. Since $\deg(b + c) = \alpha$, we have that $(b + c)_{\leq x} \neq b + c$, so $\sup(b + c) > x$. Since $x$ was arbitrary, it follows that $\sup(b + c) = 0$.

  It remains to verify that $\ot(b + c) = \omega^\alpha$. By assumption we have $\ot(b + c) \geq \omega^\alpha$. Moreover, $0 \notin \supp(b + c)$, as $0$ is neither in $\supp(b)$ nor in $\supp(c)$ when $\alpha > 0$. Therefore, $\supp(b + c)$ is the union of the sets $\supp((b + c)_{\leq x})$ for $x \in \Rb^{< 0}$. It follows that $\ot(b + c)$ is the supremum of $\ot((b + c)_{\leq x})$ for $x \in \Rb^{< 0}$. But $\deg(b_{\leq x} + c_{\leq x}) < \alpha$ for all $x \in \Rb^{<0}$, so $\ot(b + c) \leq \omega^\alpha$, so $\ot(b + c) = \omega^\alpha$, as desired.
\end{proof}

\section{The quotient monoid \texorpdfstring{$\RV$}{RV}}
\label{sec:RV-monoid}

In this section, let $\Rbf$ be a commutative ring, possibly a proper class, and $(M, +)$ be an ordered commutative monoid. Let $w : \Rbf \to M_\infty = M \cup \{-\infty\}$ be a multiplicative semi-valuation in the sense we used in \prettyref{sub:intro-new-valuation}, namely a map such that for all $b, c \in \Rbf$ we have
\begin{enumerate}
\item $w(b + c) \leq \max\{w(b), w(c)\}$ (ultrametric inequality),
\item $w(bc) = w(b) + w(c)$ (multiplicativity),
\end{enumerate}
where by convention $m - \infty = -\infty + m = -\infty$ and $-\infty < m$ for all $m \in M_\infty$ (including $-\infty < -\infty$). Recall that $w$ is a multiplicative \emph{valuation} when $w(b) = -\infty$ if and only if $b = 0$.

We shall construct a monoid $\RV$ by taking an appropriate quotient of the multiplicative monoid of $\Rbf$, and an auxiliary ring $\sRV$ containing $\RV$ as multiplicative submonoid. This will be eventually be applied to $\Rbf = \KRR$ and $w = \deg$.

\subsection{The \texorpdfstring{$\RV$}{RV} monoid}
We first build the $\RV$ monoid of a valued ring $(\Rbf, w)$, generalising the construction of the $\RV$ group of valued fields.

\begin{defn}
  Given $b, c \in \Rbf$, we write $b \sim c$ if $w(b - c) < w(b)$.
\end{defn}

We can quickly verify that $\sim$ is an equivalence relation compatible with multiplication and the semi-valuation $w$.

\begin{prop}
  \label{prop:deg-RV}
  For all $b, c \in \Rbf$, if $b \sim c$, then $w(b) = w(c)$.
\end{prop}
\begin{proof}
  By the ultrametric inequality, $w(b) = w(c + (b-c)) \leq \max\{w(c), w(b-c)\}$, so necessarily $w(b) \leq w(c)$. By symmetry, $w(c) \leq w(b)$, hence $w(b) = w(c)$.
\end{proof}

\begin{cor}
  The relation $\sim$ is an equivalence relation.
\end{cor}
\begin{proof}
  Let $b, c, d \in \Rbf$. Clearly, $b \sim b$, as $w(b - b) = -\infty < w(b)$. Moreover, if $b \sim c$, then $w(b) = w(c)$ by \prettyref{prop:deg-RV}, so $w(c - b) = w(b - c) < w(c)$, that is $c \sim b$. Finally, if $b \sim c$ and $c \sim d$, then again $w(c) = w(d) = w(b)$, and $w(b - d) \leq \max\{w(b - c), w(c - d)\} < w(b)$, so $b \sim d$.
\end{proof}

\begin{defn}
  \label{def:RV-monoid}
  Let $\RV$ be the quotient $\Rbf_{/\sim}$, and let $\rv : \Rbf \to \RV$ be the natural quotient map. Given $b \in \Rbf$, we define $w(\rv(b)) \coloneqq w(b)$.
\end{defn}

\begin{prop}
  \label{prop:product-RV}
  The relation $\sim$ is a congruence, that is, for all $b, c, d \in \Rbf$, if $c \sim d$, then $bc \sim bd$.
\end{prop}
\begin{proof}
  Let $b, c, d \in \Rbf$ with $c \sim d$. Then $w(bc - bd) = w(b) + w(c-d) < w(b) + w(c) = w(bc)$, so $bc \sim bd$.
\end{proof}

This immediately implies that the following product is well defined.

\begin{defn}
  Given $B, C \in \RV$, we let $B \cdot C \coloneqq \rv(bc)$, where $\rv(b) = B$, $\rv(c) = C$.
\end{defn}

\begin{prop}
  The multiplication $\cdot$ makes $\RV$ into a commutative monoid with zero (with multiplicative identity $\rv(1)$ and absorbing element $\rv(0)$).
\end{prop}
\begin{proof}
  Let $B, C, D \in \RV$. Pick some $b, c, d \in \Rbf$ such that $B = \rv(b), C = \rv(c), D = \rv(d)$. It follows at once from the definition that $B \cdot C = \rv(bc) = \rv(cb) = C \cdot B$ and that $B \cdot (C \cdot D) = \rv(b(cd)) = \rv((bc)d) = (B \cdot C) \cdot D$. $\rv(1)$ is clearly the neutral element, as $\rv(b) \cdot \rv(1) = \rv(b \cdot 1) = \rv(b)$ for every $b \neq 0$, and $\rv(0)$ is the absorbing one since $\rv(b) \cdot \rv(0) = \rv(0)$ for every $b \in \Rbf$.
\end{proof}

\begin{rem}
  For all $B, C \in \RV$ we have $w(B \cdot C) = w(B) + w(C)$, and $w(B) = -\infty$ if and only if $B = \rv(0)$.
\end{rem}

\subsection{Modules in \texorpdfstring{$\RV$}{RV}}

We shall now verify that $\RV$ also carries a natural notion of (partial) sum. To define the sum, we first partition $\RV$ into subsets of constant valuation.

\begin{defn}
  \label{def:RV-alpha}
  Given $m \in M$, let $\RV_m \coloneqq \{ B \in \RV \,:\, w(B) = m\} \cup \{\rv(0)\}$.
\end{defn}

\begin{rem}
  For all $m, n \in M$, $\RV_m \cdot \RV_n \subseteq \RV_{m+n}$.
\end{rem}

\begin{prop}
  \label{prop:sum-RV}
  For all $b, c, d \in \Rbf$, if $c \sim d$ and $w(b + c) = w(b)$, then $b + c \sim b + d$.
\end{prop}
\begin{proof}
  Let $b, c, d \in \Rbf$, with $c \sim d$ and $w(b + c) = w(b)$. Note in particular that $w(c) = w((b + c) - b) \leq \max\{w(b + c), w(b)\} = w(b + c)$. Then $w((b + c) - (b + d)) = w(c - d) < w(c) \leq w(b + c)$, so $b + c \sim b + d$.
\end{proof}

This implies that the following partial sum is well defined.

\begin{defn}
  Given $m \in M$ and $B, C \in \RV_m$, we define
  \[ B + C \coloneqq \begin{cases}
      \rv(b + c) & \textrm{if }w(b + c) = w(b),\\
      \rv(0) & \textrm{otherwise},
    \end{cases} \]
  where $b, c \in \Rbf$ are such that $B = \rv(b)$ and $C = \rv(c)$.
\end{defn}

We shall verify that each $\RV_m$ is an abelian group with respect to the above sum, and moreover that the product is distributive over the sum.

\begin{lem}
  \label{lem:sum-rv-b1-bn}
  For all $m \in M$ and all $b_1, \dots, b_n \in \Rbf$ either zero or of value $m$, we have
  \[ ((\rv(b_1) + \rv(b_2)) + \dots ) + \rv(b_n) = \begin{cases}
      \rv(b_1 + \dots + b_n) & \textrm{if }w(b_1 + \dots + b_n) = m, \\
      \rv(0) & \textrm{otherwise.}
    \end{cases} \]
\end{lem}
\begin{proof}
  We prove the conclusion by induction on $n$, the case $n = 1$ being trivial. Suppose that $n > 1$ and that the conclusion true for $n - 1$. Let $b_1, \dots, b_n \in \Rbf$ be zero or of valuation $m$, $c = b_1 + \dots + b_{n-1}$, and $C = (\rv(b_1) + \rv(b_2)) + \dots + \rv(b_{n-1})$. We may assume that $b_n \neq 0$, thus that $w(b_n) = m$.

  By inductive hypothesis, if $w(c) \neq m$, then $C = \rv(0)$, in which case $C + \rv(b_n) = \rv(b_n)$, while $c + b_n \sim b_n$, hence $C + \rv(b_n) = \rv(c + b_n)$. If on the other hand $w(c) = m$, we have $C = \rv(c)$. Therefore, $C + \rv(b_n)$ is by definition $\rv(c + b_n)$ if $w(c + b_n) = m$, and $\rv(0)$ otherwise.
\end{proof}

\begin{prop}
  For all $m \in M$, the sum $+$ makes $\RV_m$ into an abelian group.
\end{prop}
\begin{proof}
  Let $m \in M$ and $B, C, D \in \RV_m$. Choose $b, c, d \in \Rbf$ such that $B = \rv(b)$, $C = \rv(c)$, $D = \rv(d)$. If $w(b + c) = m$, then $B + C = \rv(b + c) = \rv(c + b) = C + B$, otherwise $B + C = \rv(0) = C + B$, so the sum is commutative. Moreover, by \prettyref{lem:sum-rv-b1-bn}, if $w(b + c + d) = m$ we have
  \[ B + (C + D) = \rv(b + c + d) = \rv(d + b + c) = D + (B + C) = (B + C) + D, \]
  otherwise $B + (C + D) = \rv(0) = D + (B + C) = (B + C) + D$. Therefore, the sum is associative.

  By definition, $B + \rv(0) = \rv(0) + B = B$, so $\rv(0)$ is the neutral element. Finally, $\rv(-b)$ is the inverse of $B$, as $w(\rv(-b)) = m$ and $B + \rv(-b) = \rv(0)$.
\end{proof}

One could alternative verify that $\RV_m$ inherits the group structure of the quotient $\mathcal{O}_m/\mathcal{M}_m$, where $\mathcal{O}_m$ is the class of the elements $b \in \mathcal{O}$ such that $w(b) \leq m$, and $\mathcal{M}_m$ is the class of the elements such that $w(b) < m$. This structures coincides with the one above.

\begin{prop}
  For all $B \in \RV$, $m \in M$ and $C, D \in \RV_m$, we have $B \cdot (C + D) = B \cdot C + B \cdot D$.
\end{prop}
\begin{proof}
  Let $C, D \in \RV_m$, $B \in \RV_n$ for some $m, n \in M$. Note that $B \cdot C, B \cdot D \in \RV_{m + n}$. Pick $b, c, d \in \Rbf$ such that $B = \rv(b)$, $C = \rv(c)$, $D = \rv(d)$.

  Assume first that $w(c + d) = w(c)$. Then $w(bc + bd) = w(b) + w(c + d) = w(b) + w(c) = w(bc)$. Therefore, $\rv(bc) + \rv(bd) = \rv(bc + bd) = \rv(b) \cdot \rv(c + d)$. Now suppose that $w(c + d) \neq w(c)$. Then $w(bc + bd) = w(b) + w(c + d) \neq w(b) + w(c) = w(bc)$. Therefore, $\rv(bc) + \rv(bd) = \rv(0) = \rv(b) \cdot \rv(c + d)$.
\end{proof}

Recall that $\RV_m \cdot \RV_n \subseteq \RV_{m + n}$ for all $m, n \in M$. It follows at once that $\RV_0$ is closed under multiplication and that $\RV_0 \cdot \RV_m \subseteq \RV_m$ for all $m \in M$.

\begin{cor}
  The sum $+$ and the product $\cdot$ make $\RV_0$ into a ring.
\end{cor}

\begin{cor}
  For all $m \in M$, $\RV_m$ is an $\RV_0$-module.
\end{cor}

The ring $\RV_0$ also has the alternative representation below, which will be convenient for calculations.

\begin{prop}
  \label{prop:classical-residue-ring}
  Let $\mathcal{O} = \{ b \in \Rbf \,:\, w(b) \leq 0\}$, $\mathcal{M} = \{ b \in \Rbf \,:\, w(b) < 0\}$. Let $\res : \mathcal{O} \to \RV_0$ be defined as $\res(b) = \rv(b)$ if $b \notin \mathcal{M}$ and $\res(b) = \rv(0)$ otherwise. Then $\res : \mathcal{O} \to \RV_0$ is a surjective ring homomorphism with kernel $\mathcal{M}$.
\end{prop}
\begin{proof}
  Recall that $\mathcal{M}$ is a (prime) ideal of $\mathcal{O}$. Let $b, c \in \mathcal{O}$.

  By definition, $\rv(b) + \rv(c) = \rv(b + c)$ whenever at least one of $b$, $c$ is in $\mathcal{O} \setminus \mathcal{M}$. If on the other hand $b, c \in \mathcal{M}$, then $b + c \in \mathcal{M}$, hence $\res(b + c) = \rv(0) = \res(b) = \res(c)$. In both cases, $\res(b + c) = \res(b) + \res(c)$.

  Similarly, by definition $\rv(b) \cdot \rv(c) = \rv(bc)$. If $bc \in \mathcal{M}$, then one of $b, c$ is in $\mathcal{M}$ and we find $\res(bc) = \rv(0) = \res(b) \cdot \res(c)$. If $bc \notin \mathcal{M}$, then $b,c \notin \mathcal{M}$. In both cases, $\res(bc) = \res(b) \cdot \res(c)$.

  It follows that $\res$ is a ring homomorphism, and by definition its kernel is $\mathcal{M}$. It is surjective by definition of $\RV_0$.
\end{proof}

\begin{defn}
  \label{def:residue-ring}
  We call $\RV_0$ the \textbf{residue ring} of $(\Rbf, w)$.
\end{defn}

\subsection{Embedding \texorpdfstring{$\RV$}{RV} into a ring}

Since the multiplication of $\RV$ is distributive over the sum defined in each $\RV_m$, it is possible to embed $\RV$ into a ring. More precisely, there is a ring $\sRV$ containing $\RV$, with the universal property that any map from $\RV$ to a ring preserving sum and product factors through $\sRV$. We construct $\sRV$ explicitly as follows.

\begin{defn}
  \label{def:sRV}
  Let $(\sRV, +)$ be the direct sum $\bigoplus_{m \in M} (\RV_m, +)$. We write its elements as $\sum_{m \in M} B_m$, where $B_m \in \RV_m$ for all $m \in M$, and $B_m \neq \rv(0)$ for at most finitely many values of $m$.

  For $\sum_{m \in M} B_m, \sum_{m \in M} C_m \in \sRV$, we define
  \[ \left(\sum_{m \in M} B_m\right) \cdot \left(\sum_{m \in M} C_m\right) \coloneqq \sum_{m \in M} \left(\sum_{\substack{n + o = m\\n, o \in M}} B_n \cdot C_o\right). \]
\end{defn}

\begin{nota}
  $\RV$ embeds naturally into $\sRV$ by sending $B \in \RV$ into the sum $\sum_{m \in M} B_m$ having $B_{w(B)} = B$ and $B_m = 0$ for $m \neq w(B)$. It is immediate from the definition that such embedding preserves sums and products, so we shall identify $\RV$ with its isomorphic copy in $\sRV$.
\end{nota}

\begin{prop}
  The sum $+$ and the product $\cdot$ make $\sRV$ into a (unital) ring (with multiplicative identity $\rv(1)$ and absorbing element $\rv(0)$).
\end{prop}
\begin{proof}
  Immediate from the definitions and from the properties of sum and product in $\RV$.
\end{proof}

\begin{defn}
  We extend $w : \RV \to M_\infty$ to $\sRV$ as follows: for every non-zero $B = \sum_m B_m \in \sRV$, let $w(B) \coloneqq \max\{m \in M : B_m \neq \rv(0)\}$, and let $w(0) \coloneqq -\infty$.
\end{defn}

\begin{prop}
  \label{prop:w-sRV-valuation}
  The function $w : \sRV \to M_\infty$ is a multiplicative semi-valuation:
  \begin{enumerate}
    \item $w(b + c) \leq \max\{w(b), w(c)\}$
    \item $w(bc) = w(b) + w(c)$.
  \end{enumerate}
  Moreover, it is a multiplicative valuation if and only if the original $w : \Rbf \to M_\infty$ is a multiplicative valuation.
\end{prop}
\begin{proof}
  Clear from the definitions.
\end{proof}

One can iterate the above constructions starting from the ring $\sRV$ and the multiplicative semi-valuation $w$. The resulting monoid is the same $\RV$, and the map $\rv : \sRV \to \RV$ simply sends $B = \sum_m B_m$ to $B_{w(B)}$. We will not use these facts and shall omit the relevant verifications.

\subsection{Some examples}

Suppose that $\Rbf$ is a field, in which case $w$ is necessarily a multiplicative valuation. Then the residue ring $\RV_0$ is the usual residue field by \prettyref{prop:classical-residue-ring}, and if we let $\RV^* \coloneqq \RV \setminus \{\rv(0)\}$, the monoid $\RV^*$ is a group sitting in the exact sequence $1 \to \RV_0^\times \to \RV^* \to M \to 0$, where $\RV_0^\times$ is the multiplicative group of the residue field. Moreover, each $\RV_m$ is a one-dimensional $\RV_0$-vector space. In many cases, $\RV$ is isomorphic to $\RV_0^\times \times M$, in which case the ring $\sRV$ is the group ring $\RV_0(M)$. This group ring appears in many common valued rings (for instance, for $\Rbf = \Zb$ and $w = v_p$ the $p$-adic valuation, $\RV_0$ is the finite field $\mathbb{F}_p$ and $\sRV$ is $\mathbb{F}_p(\mathbb{N})$).

\begin{table}[t]
  \begin{center}{
    \renewcommand{\arraystretch}{1.4}
    \begin{tabular}{c|cccc}
      $(R, w)$ & $\RV_0$ & $\RV^*$ & \footnotesize $\RV_m$ (for $m \in M$) & $\sRV$ \\
      \hline
      $(\Qb, v_p)$ & $\mathbb{F}_p$ & $\mathbb{F}_p^\times \times \Zb$ & $\mathbb{F}_p$ & $\mathbb{F}_p(\Zb)$ \\
      $(\Zb, v_p)$ & $\mathbb{F}_p$ & $\mathbb{F}_p^\times \times \Nb$ & $\mathbb{F}_p$ & $\mathbb{F}_p(\Nb)$ \\
      \footnotesize $(\KRR, v)$ & $\Kbf$ & $\Kbf^\times \times \Rb$ & $\Kbf$ & $\KR$ \\
      \scriptsize $(\KRR, \sup)$ & $\KRR_{/\Jbf}$ & \tiny $(\KRR_{/\Jbf})^* \times \Rb$ & $\KRR_{/\Jbf}$ & \scriptsize $\KRR_{/\Jbf}(\Rb^{\leq 0})$ \\
      \scriptsize $(\KRR, \deg)$ & \tiny $\KR$ (\ref{prop:rv-KR-isom}) & $\RV^{*}$ & \tiny $\PRV_m \otimes_\Kbf \KR$ (\ref{prop:RV-alpha-tensor}) & \footnotesize $\sPRV(\Rb^{\leq 0})$ (\ref{prop:sRV-tensor}) \\
      \scriptsize $(\KRR, v_J)$ & $\Kbf$ & $\PRV$ (\ref{rem:RV-of-vJ}) & $\PRV_\alpha$ (for $m = \omega^\alpha$) & $\sPRV$
    \end{tabular}
  }\end{center}

  \caption{Rings $\sRV$ arising from various valued rings.}
  \label{tab:examples}
\end{table}

For the sake of example, \prettyref{tab:examples} lists a few notable residue rings $\RV_0$, monoids $\RV^* = \RV \setminus \{\rv(0)\}$, modules $\RV_m$ and rings $\sRV$ for different valuations. The first four rows can be verified directly from the definitions. The fifth row, along with the definitions of $\PRV$, $\PRV_\alpha$ and $\sPRV$,  will be discussed and proved in Sections \ref{sec:existence}, \ref{sec:uniqueness}. For the last row, see \prettyref{rem:RV-of-vJ}.

\section{Finding a factorisation}
\label{sec:existence}

During this section, we shall work with the ring $\KRR$ and the valuation $\deg$. We apply the constructions of \prettyref{sec:RV-monoid}: the equivalence relation $\sim$, the monoid $\RV$, the modules $\RV_\alpha$ shall always refer to the ones obtained from the ring $\Rbf = \KRR$ with the multiplicative valuation $w = \deg$, where $\alpha$ ranges in the monoid $M = \omega_1$ equipped with Hessenberg's natural sum (see \prettyref{rem:monoid-is-omega-1}). Thus $\rv : \KRR \to \RV$ will denote the multiplicative, degree-preserving quotient map. We will not use the ring $\sRV$ in this section.

\subsection{The residue ring of generalised power series}

First, we check that the residue ring of $\KRR$ is (an isomorphic copy of) the subring of the series with degree at most $0$, that is $\KR$.

\begin{prop}
  \label{prop:rv-KR-isom}
  The residue ring $\RV_0$ is the image $\rv(\KR)$, and the restriction $\rv_{\restriction \KR} : \KR \to \RV_0$ is a ring isomorphism.
\end{prop}
\begin{proof}
  Immediate by \prettyref{prop:classical-residue-ring}, as $\mathcal{O} = \{ b \in \KRR \,:\, \deg(b) \leq 0\} = \KR$ and $\mathcal{M} = \{ b \in \KRR \,:\, \deg(b) < 0\} = \{0\}$.
\end{proof}

\begin{nota}
   With a slight abuse of notation, we shall identify $\RV_0$ with $\KR$, and for instance write $0$ in place of $\rv(0)$. In particular, we shall also say that each $\RV_\alpha$ is a $\KR$-module.
\end{nota}

\subsection{The submonoid of principal elements}

By \prettyref{fact:berarducci}, the product of two weakly principal series is weakly principal. We have also shown that the product of two \emph{principal} series is also principal (\prettyref{prop:berarducci}). Therefore, their images through the map $\rv$ form a multiplicative submonoid of $\RV$.

\begin{defn}
  \label{def:principal-rv}
  Given $B \in \RV$, we say that $B$ is \textbf{principal} if $B = \rv(b)$ for some principal $b \in \KRR$. Let $\PRV$ denote the subset of $\RV$ consisting of $0$ and of all principal elements of $\RV$. Given $\alpha \in \omega_1$, we define $\PRV_\alpha \coloneqq \PRV \cap \RV_\alpha$.
\end{defn}

In other words, $\PRV$ is the image through $\rv$, plus $0$, of the set of principal series mentioned in \prettyref{rem:P-not-a-ring}.

\begin{exa}
  Under the identification $\KR = \RV_0$, we have that $\PRV_0 = \Kbf$. Indeed, a series $p \in \KR$ is principal if and only if $p \in \Kbf \setminus \{0\}$.
\end{exa}

\begin{prop}
  $\PRV$ is closed under the multiplication of $\RV$.
\end{prop}
\begin{proof}
  Recall that if $b, c \in \KRR$ are principal, then $bc$ is principal by \prettyref{prop:berarducci}. Therefore, if $B, C \in \PRV \setminus \{0\}$, then $B = \rv(b), C = \rv(c)$ for some principal $b, c \in \KRR$, so $B \cdot C = \rv(bc)$ is principal as well. If one of $B, C$ is $0$, their product is obviously in $\PRV$ as well.
\end{proof}

We shall now verify that the intersection of $\PRV$ with each $\RV_\alpha$ is also a $\Kbf$-linear space.

\begin{prop}
  For all $\alpha \in \omega_1$, $\PRV_\alpha$ is an additive subgroup of $\RV_\alpha$.
\end{prop}
\begin{proof}
  Let $B, C \in \PRV_\alpha$ for some given $\alpha \in \omega_1$. Pick $b, c \in \KRR$ such that $B = \rv(b), C = \rv(c)$. If $\deg(b + c) = \deg(b)$, then $B + C = \rv(b + c)$, and $b + c$ is principal by \prettyref{prop:sum-principal}, so $B + C$ is principal and in $\RV_\alpha$, hence $B + C \in \PRV_\alpha$. On the other hand, if $\deg(b + c) \neq \deg(b)$, then $B + C = 0 \in \PRV_\alpha$. Therefore, for each $\alpha \in \omega_1$, $\PRV_\alpha$ is closed under sum. Finally, if $B \in \PRV_\alpha$ is of the form $B = \rv(b)$ for some principal $b$, then $-b$ is also principal as $\ot(-b) = \ot(b)$ and $\sup(-b) = \sup(b) = 0$, so $\rv(-b) \in \PRV_\alpha$ and $B + \rv(-b) = 0$.
\end{proof}

\begin{cor}
  For all $\alpha \in \omega_1$, $\PRV_\alpha$ is a $\Kbf$-linear subspace of $\RV_\alpha$.
\end{cor}
\begin{proof}
  Note that if $b \in \KRR$ is principal, then so is $kb$ for any non-zero $k \in \Kbf$, since $\ot(kb) = \ot(b)$ and $\sup(kb) = \sup(b) = 0$, hence $\Kbf \cdot \PRV_\alpha \subseteq \PRV_\alpha$.
\end{proof}

\subsection{Decomposing the modules}
\label{sub:RV-alpha-tensor}
The multiplication map $\mu : \PRV_\alpha \times \KR \to \RV_\alpha$ given by $(B, p) \mapsto B \cdot p$ is $\Kbf$-linear in both arguments. Thus, it has a canonical extension to a $\Kbf$-linear map $\tilde{\mu} : \PRV_\alpha \otimes_\Kbf \KR \to \RV_\alpha$ on the tensor product of $\PRV_\alpha$ and $\KR$ over $\Kbf$. We can show that $\tilde{\mu}$ is an isomorphism, and thus that $\RV_\alpha$ is the \textbf{extension of scalars} of $\PRV_\alpha$ from $\Kbf$ to $\KR$.

We briefly recall that if $\mathbf{V}$, $\mathbf{W}$ are $\Kbf$-linear spaces and $\mathbf{W}$ has a basis, the tensor product $\mathbf{V} \otimes_\Kbf \mathbf{W}$ can be presented as the class of all finite sums of the vectors $v \otimes w$, for $v$ varying among the non-zero vectors $\mathbf{V}$ and $w$ in the given basis of $\mathbf{W}$. The same can be done if we are given a basis of $\mathbf{V}$. In particular, since the monomials $t^x$ for $x \in \Rb^{\leq 0}$ form a basis of $\KR$, each element of $\PRV_\alpha \otimes_\Kbf \KR$ has a unique representation as a sum
\[ B_1 \otimes t^{x_1} + \dots + B_n \otimes t^{x_n} \]
where $x_1 < \dots < x_n \leq 0$ are distinct real numbers and $B_1, \dots, B_n$ are non-zero elements of $\PRV_\alpha$. Another consequences is that if $C_1, \dots, C_k \in \PRV_\alpha$ are $\Kbf$-linearly independent, then for each $\mathcal{B} \in \PRV_\alpha \otimes_\Kbf \KR$ there is at most one choice of $q_1, \dots, q_k \in \KR$ such that
\[ \mathcal{B} = C_1 \otimes q_1 + \dots + C_k \otimes q_k. \]

\begin{prop}
  \label{prop:RV-alpha-tensor}
  For all $\alpha \in \omega_1$, the map $\tilde{\mu} : \PRV_\alpha \otimes_\Kbf \KR \to \RV_\alpha$ is an isomorphism.
\end{prop}
\begin{proof}
  Let $\alpha \in \omega_1$. Pick a non-zero element of $\mathcal{B} \in \PRV_\alpha \otimes_\Kbf \KR$, with image $\tilde{\mu}(\mathcal{B}) = B \in \RV_\alpha$, and write (uniquely)
  \[ \mathcal{B} = B_1 \otimes t^{x_n} + \dots + B_n \otimes t^{x_n} \]
  where $B_i \in \PRV_\alpha$ are non-zero, $x_1 < \dots < x_n$, and $n > 0$.

  Let $b_i \in \KRR$ be some principal series such that $\rv(b_i) = B_i$, so that $B = \rv(b)$ for $b = b_1t^{x_1} + \dots + b_mt^{x_m}$. We may further assume that $\supp(b_j) \geq x_{j-1} -  x_j$ for all $j = 2, \dots, m$, as we are free to replace $b_j$ with $(b_j)_{\geq x_{j-1} - x_j}$ without affecting the value of $\rv(b_j)$. With these choices, $b_1t^{x_1} + \dots + b_mt^{x_m}$ is the normal form of $b$. It follows that $\deg(B) = \deg(b_1) = \alpha$, thus in particular $B \neq 0$. This shows that $\tilde{\mu}$ is injective.

  For surjectivity, pick some $B \in \RV_\alpha$ and some $b \in \KRR$ such that $\rv(b) = B$. Write $b = b_1t^{x_1} + \dots + b_nt^{x_n}$ in normal form, with $b_1, \dots, b_n$ principal. We may further assume $\deg(b_1) = \dots = \deg(b_n) = \alpha$, as we can discard the terms of lower degree without changing the value of $\rv(b)$. Then, by \prettyref{lem:sum-rv-b1-bn},
  \begin{align*}
    B &= \rv(b) = \rv(b_1t^{x_1} + \dots + b_nt^{x_n}) = \rv(b_1t^{x_1}) + \dots + \rv(b_nt^{x_n}) = \\
    &= \rv(b_1) \cdot t^{x_1} + \dots + \rv(b_n) \cdot t^{x_n} = \tilde{\mu}\left(\rv(b_1) \otimes t^{x_1} + \dots + \rv(b_n) \otimes t^{x_n}\right). \qedhere
  \end{align*}
\end{proof}

\subsection{The maximal divisor of finite support for \texorpdfstring{$\RV$}{RV}}

We can now verify that each element of $\RV_\alpha$ has a maximal divisor in $\KR$.

\begin{defn}
  \label{def:div-RV}
  Given $B, C \in \RV$, we say that $B$ \textbf{divides} $C$, and write $B \mid C$, if $C = B \cdot D$ for some $D \in \RV$.
\end{defn}

\begin{rem}
  For all $p, q \in \KR = \RV_0$, $p$ divides $q$ in the sense of \prettyref{def:div-RV} if and only if $p$ divides $q$ in the ring $\KR$.
\end{rem}

\begin{prop}
  \label{prop:pB}
  For all $B \in \RV$, there exists a series $p \in \KR$ such that for all $q \in \KR$, $q \mid B$ if and only if $q \mid p$. Moreover, $p$ is unique up to multiplication by non-zero elements of $\Kbf$, and $p \in \Kbf$ if $B$ is principal.
\end{prop}
\begin{proof}
  Let $B \in \RV$. If $B = 0$, the conclusion is trivial on taking $p = 0$, so assume otherwise. Let $\alpha = \deg(B)$, and write (uniquely)
  \[ B = B_1 \cdot t^{x_1} + \dots + B_n \cdot t^{x_n} \]
  where $x_1 < \dots < x_n$ and $B_1, \dots, B_n$ are non-zero elements of $\PRV_\alpha$. Let $C_1, \dots, C_k$ be a maximal $\Kbf$-linearly independent subset of $B_1, \dots, B_n$, and rewrite the above equality as
  \[ B = C_1 \cdot q_1 + \dots + C_k \cdot q_k \]
  for some $q_1, \dots, q_k \in \KR$. Since $C_1, \dots, C_k$ have been chosen to be $\Kbf$-linearly independent, the series $q_1, \dots, q_k$ are unique.

  Suppose that $q \mid B$ for some $q \in \KR$, namely $B = B' \cdot q$ for some $B' \in \RV_\alpha$. Via the same argument as for $B$, we can find $C_{k+1}, \dots, C_{k+\ell} \in \PRV_\alpha$ that are $\Kbf$-linearly independent over $C_1, \dots, C_k$ and such that
  \[ B' = C_1 \cdot q_1' + \dots + C_{k+\ell} \cdot q_{k+\ell}' \]
  for some (unique again) $q_1', \dots, q_{k+\ell}' \in \KR$. In particular,
  \[ B = C_1 \cdot (qq_1') + \dots + C_{k+\ell} \cdot (qq_{k+\ell}') = C_1 \cdot q_1 + \dots + C_k \cdot q_k. \]
  By the uniqueness of the series $q_i$, we find that $q_{k+1}' = \dots = q_{k+\ell}' = 0$, and that $q$ divides $q_1, \dots, q_k$.

  Conversely, if $q$ divides $q_1, \dots, q_k$, say $q_i = qq_i'$, then
  \[ B = C_1 \cdot (qq_1') + \dots + C_k \cdot (qq_k') = (C_1 \cdot q_1' + \dots + C_k \cdot q_k') \cdot q, \]
  thus $q$ divides $B$.

  It follows that our desired $p$ is exactly a greatest common divisor of $q_1, \dots, q_k$, which exists since $\KR$ is a GCD domain (\prettyref{fact:KG-GCD-domain}). When $B$ is principal, we may assume that $B = C_1$, in which case $p \in \Kbf$. Such $p$ is determined up to multiplication by units of $\KR$, and the units of $\KR$ are the non-zero elements of $\Kbf$.
\end{proof}

\begin{cor}
  \label{cor:pB}
  For all $B \in \sRV$, there exists a $p \in \KR$ such that for all $q \in \KR$, $q \mid B$ if and only if $q \mid p$, and such $p$ is unique up to multiplication by non-zero elements of $\Kbf$.
\end{cor}
\begin{proof}
  Let $B = \sum_\alpha B_\alpha$. By construction, $q \in \KR$ divides $B$ if and only if $q$ divides each $B_\alpha$. Let $p_\alpha \in \KR$ be the series given by \prettyref{prop:pB} applied to $B_\alpha$, for all (finitely many) $\alpha$ such that $B_\alpha \neq 0$. Let $p \in \KR$ to be a greatest common divisor of such $p_\alpha$. Then $q \in \KR$ divides $B$ if and only if $q$ divides $p$. Again, such $p$ is determined up to multiplication by non-zero elements of $\Kbf$.
\end{proof}

\begin{nota}
  \label{nota:pB}
  Given a non-zero $B \in \sRV$, if $p$ is an element of $\KR$ satisfying the conclusion of \prettyref{cor:pB}, then $kp$ also satisfies the conclusion for any non-zero $k \in \Kbf$. Therefore, there is exactly one choice of $p$ such that the coefficient of $t^{\sup(p)}$ is $1$. Let $\p(B)$ be this element. We also let $\p(0) \coloneqq 0$.
\end{nota}

\begin{rem}
  \label{rem:pB}
  For every $B \in \sRV$ we have $\p(B) \mid B$, and when $p \in \KR = \RV_0$, we have $\p(p) = kp$ for some non-zero $k \in \Kbf$. Moreover, $\p(B) = 1$ whenever $B \in \RV$ is principal.
\end{rem}

\begin{rem}
  \label{rem:pB-computation}
  The proof of \prettyref{prop:pB} is also an algorithm to compute the maximal divisor of $\KR$. Suppose that $B \in \RV_\alpha$ is of the form $B = C_1 \cdot q_1 + \dots + C_k \cdot q_k$ for some $\Kbf$-linearly independent elements $C_i \in \PRV_\alpha$. Then $\p(B)$ is a greatest common divisor of $q_1, \dots, q_k$. To find such elements $C_i$, it suffices to write $B = B_1 \cdot t^{x_1} + \dots + B_n \cdot t^{x_n}$ in normal form, and extract a maximal subset of $\Kbf$-linearly independent elements. Likewise, the proof of \prettyref{cor:pB} shows that for $B = \sum_\alpha B_\alpha \in \sRV_\alpha$, $\p(B)$ is a greatest common divisor of $\{\p(B_\alpha) : \alpha \in \omega_1\}$.

  For instance, take
  \[ B = \rv\left(\sum_{n \in \Nb} \left(t^{-\frac{1}{n+1}} + t^{-\frac{2}{n+1}}\right) - \sum_{n \in \Nb} t^{-2 - \frac{1}{n+1}} - \sum_{n \in \Nb} t^{-3 - \frac{2}{n+1}}\right). \]
  Note that $B$ has degree 1. Let $C = \rv\left(\sum_{n \in \Nb} t^{\frac{-1}{n+1}}\right)$ and $D = \rv\left(\sum_{n \in \Nb} t^{\frac{-2}{n+1}}\right)$. Both $C$ and $D$ are principal of degree 1, and one can easily prove that they are $\Kbf$-linearly independent, for instance using the fact that their supports are disjoint. We have
  \[ B = C + D - C \cdot t^{-2} - D \cdot t^{-3} = C \cdot (1 - t^{-2}) + D \cdot (1 - t^{-3}). \]
  A greatest common divisor of $1 - t^{-2}$ and $1 - t^{-3}$ is $t^{-1} - 1$, hence $\p(B) = 1 - t^{-1}$.
\end{rem}

\begin{prop}
  \label{prop:pB-pC-div-pBC}
  For all $B, C \in \sRV$, $\p(B)\p(C) \mid \p(B \cdot C)$.
\end{prop}
\begin{proof}
  Let $B, C \in \sRV$. Take $B', C' \in \RV$ such that  $B = B' \cdot \p(B)$ and $C = C' \cdot \p(C)$. Then $B \cdot C = B' \cdot C' \cdot (\p(B)\p(C))$, so $\p(B)\p(C) \mid B \cdot C$, so by definition $\p(B)\p(C) \mid \p(B \cdot C)$.
\end{proof}

\subsection{The maximal divisor of finite support for \texorpdfstring{$\KRR$}{K((ℝ≤0))}}
The existence of the maximal divisor of finite support can be lifted from the quotient $\RV$ to the ring $\KRR$. In this step, it is crucial that the valuation $\deg$ takes values in the ordinals, as it allows us to reason by induction on the degree.

\begin{prop}
  \label{prop:pb-KRR}
  For all $b \in \KRR$, there exists a $p \in \KR$ such that for all $q \in \KR$, $q \mid b$ if and only if $q \mid p$. Moreover, $p$ in unique up to multiplication by non-zero elements of $\Kbf$.
\end{prop}
\begin{proof}
  Let $b \in \KRR$ and $q \in \KR$. We reason by induction on $\deg(b)$. If $\deg(b) \leq 0$, then $b \in \KR$, so the conclusion is trivial on taking $p = b$.

  Suppose now that $\deg(b) > 0$. Let $B = \rv(b)$. By definition, there is $b' \in \KRR$ such that $B = B' \cdot \p(B)$, where $B' = \rv(b')$. Then $b \sim \p(B)b'$, so we can write $b = \p(B)b' + c$ with $c \in \KRR$ satisfying $\deg(c) < \deg(b)$.

  If $q \mid \p(B)$ and $q \mid c$, then clearly $q \mid b$. Conversely, assume that $q \mid b$. In particular, $q \mid \rv(b) = B$, so $q \mid \p(B)$ by definition of $\p(B)$. In turn, $q \mid (b - \p(B)b') = c$. Therefore, $q \mid b$ if and only if $q \mid \p(B)$ and $q \mid c$. By inductive hypothesis, there exists a $p' \in \KR$ such that $q \mid c$ if and only if $q \mid p'$. The conclusion follows on letting $p$ be a greatest common divisor of $\p(B)$ and $p'$.
\end{proof}

\begin{nota}
  \label{nota:pb}
  Just as in \prettyref{nota:pB}, given a non-zero $b \in \KRR$, let $\p(b)$ be the unique divisor satisfying the conclusion of \prettyref{prop:pb-KRR}, namely $q \mid b$ if and only if $q \mid \p(b)$ for all $q \in \KR$, such that the coefficient of $t^{\sup(p)}$ is $1$. We also let $\p(0) \coloneqq 0$.
\end{nota}

\begin{rem}
  For every $b \in \KRR$, we have $\p(b) \mid b$. When $p \in \KR = \RV_0$, we have $\p(p) = kp$ for some non-zero $k \in \Kbf$ (in fact, this new definition of $\p(p)$ coincides with the one of \prettyref{nota:pB}). If $b$ is principal, then $\p(b) = 1$.
\end{rem}

The above proof is again also an algorithm to find the maximal divisor in $\KR$. Given $b_0 = b \in \KRR$, compute $p_0 = \p(\rv(b_0))$ as in the previous subsection, and thus in particular find some $b_0' \in \KRR$ such that $\rv(b_0') \cdot p_0 = \rv(b_0)$. Take $b_1 = b_0 - p_0b_0'$, where by construction $\deg(b_1) < \deg(b)$. If $b_1 = 0$, we are done, otherwise compute $p_1 = \p(\rv(b_1))$, $b_1'$ such that $\rv(b_1') \cdot p_1 = \rv(b_1)$, and $b_2 = b_1 - p_1b_1'$ as above, so that $\deg(b_2) < \deg(b_1)$. Repeat the procedure until, after finitely many steps, we reach $b_n = 0$. Finally, compute a greatest common divisor of $p_1, \dots, p_{n-1}$ and adjust its coefficients to find $\p(b)$.

\begin{exa}
  Let
\[ b_0 = -\sum_{n \in \Nb} t^{-2 - \frac{1}{n+1}} + t^{-1} + \sum_{n \in \Nb} t^{-\frac{1}{n+1}} + 1.\]
This is a series of order type $\omega + \omega + 1$, thus of degree $1$. To find its maximal divisor of finite support $\p(b_0)$, we first consider $\rv(b_0)$. We have
\[ \rv(b_0) = \rv\left(-\sum_{n \in \Nb} t^{-2 - \frac{1}{n+1}} + \sum_{n \in \Nb} t^{-\frac{1}{n+1}}\right) = \rv\left(\sum_{n \in \Nb} t^{-\frac{1}{n+1}}\right) \cdot (1 - t^{-2}). \]
Since $\rv(\sum_{n \in \Nb} t^{\frac{-1}{n+1}})$ is principal, $p_0 = \p(\rv(b_0)) = (1 - t^{-2})$. We may now write
\[ b_1 = b - (1 - t^{-2}) \sum_{n \in \Nb} t^{-\frac{1}{n+1}} = t^{-1} + 1. \]
Since $p_1 = \p(t^{-1}+1) = t^{-1}+1$, the next computation yields $b_2 = 0$, and we stop. It now suffices to compute a greatest common divisor of $(1 - t^{-2})$ and $t^{-1} + 1$, and we find $\p(b) = t^{-1} + 1$.
\end{exa}

\begin{prop}
  \label{prop:pb-pc-div-pbc}
  For all $b, c \in \KRR$, $\p(b)\p(c) \mid \p(bc)$.
\end{prop}
\begin{proof}
  Let $b, c \in \KRR$. Take $b', c' \in \KRR$ such that $b = \p(b)b$, $c' = \p(c)c$. Then $bc = (\p(b)\p(c))b'c'$, so $\p(c)\p(c) \mid bc$, so by definition $\p(b)\p(c) \mid \p(bc)$.
\end{proof}

\subsection{The factorisation} We can now factor the series in $\KRR$ by induction on the degree.

\begin{prop}
  \label{prop:KRR-fact-exists}
  For all non-zero $b \in \KRR$, there exist $n \in \Nb$, irreducible series $c_1, \dots, c_n \in \KRR$ with infinite support, and $k \in \Kbf$ such that
  \[ b = k \p(b) c_1 \cdots c_n. \]
  Moreover, $n$ is at most the number of terms in the Cantor Normal Form of $\deg(b)$.
\end{prop}
\begin{proof}
  Let $b \in \KRR$ be non-zero and let $b' \coloneqq \frac{b}{\p(b)} \in \KRR$. By maximality of $\p(b)$, we must have $\p(b') = 1$. We work by induction on $\deg(b)$. If $\deg(b) = 0$, then $b' \in \Kbf$, and we are done.

  Assume $\deg(b) = \alpha > 0$. If $b'$ is irreducible, then we are done. Otherwise, $b' = cd$ for some $c, d \in \KRR$ not in $\Kbf$. Since $\p(b') = 1$, we have $\p(c) = \p(d) = 1$ by \prettyref{prop:pb-pc-div-pbc}, so $c, d$ cannot be in $\KR$, hence $\deg(c), \deg(d) > 0$. Since $\deg(b') = \deg(c) \oplus \deg(d)$, it follows that $\deg(c) < \alpha$, $\deg(d) < \alpha$.

  By inductive hypothesis, $c$ and $d$ can be written as products of irreducible series of positive degree. Therefore, $b'$ is also a product of irreducible series of positive degree, that is to say, with infinite support, as desired.

  Moreover, since $\deg(\p(b)) = 0$ by definition, we have $\deg(b) = \deg(c_1) \oplus \dots \oplus \deg(c_n)$, where $\deg(c_i) > 0$ for each $i = 1, \dots, n$. It follows that $n$ is bounded by the number of terms in the Cantor Normal Form of $\deg(b)$.
\end{proof}

The above argument has an algorithmic flavour: one may compute $p$ as described after \prettyref{prop:pB}, then proceed by induction on the degree to find the remaining factors. However, we do not know an algorithm to determine when a series with infinite support is irreducible, and the task seems highly not trivial. We discuss in \prettyref{sec:new-irreducibles-primes} a few techniques to find irreducible series, which are however still limited to series of a certain form.

\begin{rem}
  The bound on $n$ is sharp. More strongly, for every $\alpha \in \omega_1$, if $n$ is the number of terms in its Cantor Normal Form, then there is a series $b \in \KRR$ of degree $\alpha$ that factors into $n$ irreducible series with infinite support. To see this, write $\alpha = \omega^{\beta_1} \hatplus \dots \hatplus \omega^{\beta_n}$, and pick principal series $c_1,\dots,c_n \in \KRR$ of degrees respectively $\omega^{\beta_1}, \dots, \omega^{\beta_n}$ (they exist since the degree is surjective on $\omega_1$ by \prettyref{rem:monoid-is-omega-1}). Each $c_i$ has order type $\omega^{\omega^{\beta_i}}$, and $\sup(c_i) = 0$, so it is irreducible by \cite{Ber2000}, hence $b = c_1 \cdots c_n$ has as many irreducible factors as the number of terms in the Cantor Normal form of its degree.

  On the other hand, the factors may be fewer than $n$. For instance, we know from \cite{PS2006,LM2017} that there are irreducible series in $\KRR$ of order types $\omega^2, \omega^3$. The Cantor Normal Form of their degrees are respectively $1 \hatplus 1$ and $1 \hatplus 1 \hatplus 1$, but their factorisations have only one irreducible term.
\end{rem}

\section{Uniqueness of the factor with finite support}
\label{sec:uniqueness}

As in the previous section, we shall work with the ring $\KRR$ and the valuation $\deg$; the symbols $\sim$, $\RV$, $\RV_\alpha$, $\sRV$, and $\rv : \KRR \to \RV \subseteq \sRV$ shall refer to the objects constructed in \prettyref{sec:RV-monoid}. Recall that we are identifying $\RV_0$ with $\KR$.

\subsection{The subring of principal elements}

\begin{defn}
  An element $\sum_\alpha B_\alpha$ of $\sRV$ is \textbf{principal} if each $B_\alpha$ is either zero or principal (namely in $\PRV$). Let $\sPRV$ denote the set of all principal elements of $\sRV$.
\end{defn}

Note that $\sPRV$ is precisely the subring of $\sRV$ generated by $\PRV$, and it contains $\rv(1)$. Similarly to \prettyref{sub:RV-alpha-tensor}, we note that the map $(B,p) \mapsto B \cdot p$ from $\sPRV \times \KR$ to $\sRV$ is bilinear, thus it extends uniquely to a map $\tilde{\mu} : \sPRV \otimes_\Kbf \KR \to \sRV$.

\begin{prop}
  \label{prop:sRV-tensor}
  The map $\tilde{\mu} : \sPRV \otimes_\Kbf \KR \to \sRV$ is an isomorphism.
\end{prop}
\begin{proof}
  Immediate, since $\tilde{\mu}$ restricts to an isomorphism $\RV_\alpha \to \PRV_\alpha \otimes_\Kbf \KR$ for each $\alpha$ by \prettyref{prop:RV-alpha-tensor}, and $\sPRV$, $\sRV$ are direct sums of respectively $\PRV_\alpha$ and $\RV_\alpha$.
\end{proof}

Therefore, by the properties of the tensor product, every $B \in \sRV$ can be written in a unique way as $B = \sum_{x \in \Rb^{\leq 0}} B_x \cdot t^x$, where $B_x \in \sPRV$ and at most finitely many $B_x$'s are non-zero. In particular, we can canonically identify $\sRV$ with $\sPRV(\Rb^{\leq 0})$, the ring of the series with coefficients in $\sPRV$, exponents in $\Rb^{\leq 0}$, and finite support.

\begin{rem}
  \label{rem:frac-sprv-GCD}
  Consider the ring $\sPRV^{-1} \cdot \sRV$ obtained by localising $\sRV$ with respect to the multiplicative subset $\sPRV \setminus \{0\}$. Since $\sPRV$ is a subring of $\sRV$, we have $\sPRV^{-1} \cdot \sRV = \Frac(\sPRV) \cdot \sRV = \Frac(\sPRV)(\Rb^{\leq 0})$. In particular, $\sPRV^{-1} \cdot \sRV$ is a GCD domain by \prettyref{fact:KG-GCD-domain}.
\end{rem}

\subsection{Divisibility in \texorpdfstring{$\sRV$}{R̂V}}

Given $B, C \in \sRV$, we say that $B$ \textbf{divides} $C$, and write $B \mid C$, if $C = B \cdot D$ for some $D \in \sRV$. First of all, we observe that divisibility in $\sRV$ is just an extension of the notion of divisibility in \prettyref{def:div-RV}, so there is no risk of ambiguity.

\newcommand{\mindeg}{\deg^{-}}

\begin{prop}
  \label{prop:BC-RV-B-C-RV}
  For all non-zero $B, C \in \sRV$, if $B \cdot C \in \RV$, then $B, C \in \RV$.
\end{prop}
\begin{proof}
  Given a non-zero $B = \sum_\alpha B_\alpha \in \sRV$, let $\mindeg(B)$ be the smallest $\alpha$ such that $B_\alpha \neq 0$. We also set $\mindeg(0) = +\infty$. Clearly, for $B \neq 0$, we have $\mindeg(B) \leq \deg(B)$ and $\mindeg(B) = \deg(B)$ if and only if $B \in \RV$. The reader can easily verify that $\mindeg(B + C) \geq \min\{\mindeg(B), \mindeg(C)\}$ for all $B, C$ and $\deg^{-}(B) = +\infty$ if and only if $B = 0$.

  We claim that for all non-zero $B, C \in \sRV$, $\mindeg(B \cdot C) = \mindeg(B) \oplus \mindeg(C)$. Let $B, C \in \sRV$ be non-zero, and write $B \cdot C = \sum_\alpha D_\alpha$. By definition,
  \[ D_\alpha = \left(\sum_{\beta \oplus \gamma = \alpha} B_\beta \cdot C_\gamma\right). \]
  It follows at once that $D_\alpha = 0$ for all $\alpha < \mindeg(B) \oplus \mindeg(C)$, while $D_\alpha = B_\beta \cdot C_\gamma \neq 0$ for $\beta = \mindeg(B)$, $\gamma = \mindeg(C)$ and $\alpha = \beta \oplus \gamma$. Therefore, $\mindeg(B \cdot C) = \mindeg(B) \oplus \mindeg(C)$. In particular, $\mindeg$ is a valuation.

  In particular, if $B, C \in \sRV$ are non-zero and $B \cdot C \in \RV$, then $\mindeg(B \cdot C) = \mindeg(B) \oplus \mindeg(C) = \deg(B \cdot C) = \deg(B) \oplus \deg(C)$. It follows at once that $\mindeg(B) = \deg(B)$ and $\mindeg(C) = \deg(C)$, so $B, C \in \RV$.
\end{proof}

\begin{cor}
  \label{cor:BC-PRV-B-C-PRV}
  For all non-zero $B, C \in \sRV$, if $B \cdot C \in \PRV$ (or $\sPRV$), then $B, C \in \PRV$ (resp.\ $\sPRV$).
\end{cor}
\begin{proof}
  Recall that $\sRV = \sPRV(\Rb^{\leq 0})$ as explained after \prettyref{prop:sRV-tensor}. Then clearly if two non-zero $B, C \in \sRV$ are such that $B \cdot C \in \sPRV$, we must have $B, C \in \sPRV$. If moreover $B \cdot C \in \PRV \subseteq \RV$, then by \prettyref{prop:BC-RV-B-C-RV}, $B, C \in \RV \cap \sPRV = \PRV$.
\end{proof}

\begin{cor}
  For all $B, C \in \RV$, $B$ divides $C$ in the sense of \prettyref{def:div-RV} if and only if $B$ divides $C$ in the ring $\sRV$. In particular, for all $p, q \in \KR$, $p$ divides $q$ in $\sRV$ if and only if $p$ divides $q$ in $\KR$.
\end{cor}

We also observe the following.

\begin{prop}
  \label{prop:B-div-C-div-C-alpha}
  For all $B \in \RV$, $C = \sum_\alpha C_\alpha \in \sRV$, we have that $B \mid C$ if and only if $B \mid C_\alpha$ for all $\alpha \in \omega_1$.
\end{prop}
\begin{proof}
  Let $B \in \RV$, $C = \sum_\alpha C_\alpha \in \sRV$. Clearly, we may assume $B \neq 0$, so let $\beta = \deg(B)$.

  If $B \mid C$, then $C = B \cdot C'$ for some $C' = \sum_\alpha C_\alpha'$. By definition of product, $B \cdot C' = \sum_\alpha \sum_{\beta \oplus \gamma = \alpha} B \cdot C'$. Therefore, we have $C_{\alpha \oplus \beta} = B \cdot C_\alpha'$ for all $\alpha$, and $C_\gamma = 0$ for all ordinal $\gamma$ that cannot be written in the form $\alpha \oplus \beta$. In turn, $B \mid C_\alpha$ for all $\alpha$.

  Conversely, suppose $B \mid C_\alpha$ for all $\alpha$. Then for all $\alpha$ we must have $C_{\alpha \oplus \beta} = B \cdot C_\alpha'$ for some $C_\alpha' \in \RV_\alpha$, and if $\gamma$ is an ordinal not of the form $\gamma = \alpha \oplus \beta$, then $C_\gamma = 0$. It follows at once that $B \cdot \sum_\alpha C_\alpha' = C$, so $B \mid C$, as desired.
\end{proof}

\subsection{All series in \texorpdfstring{$\KR$}{K(ℝ≤0)} are primal in \texorpdfstring{$\KRR$}{K((ℝ≤0))}}

By \prettyref{fact:KG-GCD-domain}, every element of $\KR$ is primal in $\KR$. Recall that in a ring $R$, an element $p \in R$ is primal if whenever $p$ divides a product $bc$ of some $b,c \in R$, there are $p_1,p_2 \in R$ such that $p = p_1p_2$ and $p_1$ divides $b$, $p_2$ divides $c$. We shall now prove that the series in $\KR$ are primal in both $\sRV$ and $\KRR$.

\begin{lem}
  \label{lem:p-div-BC-div-B-RV}
  For all $B \in \RV$ and non-zero $C \in \PRV$ we have $\p(B \cdot C) = \p(B)$.
\end{lem}
\begin{proof}
  Let $p \in \KR$, $B \in \RV$, $C \in \PRV$ with $C \neq 0$. Clearly, we may assume $B \neq 0$. Let $\beta = \deg(B)$ and $\alpha = \deg(B \cdot C)$. As explained in \prettyref{rem:pB-computation}, we may compute $\p(B)$ and $\p(B \cdot C)$ as follows. We can find some $\Kbf$-linearly independent $D_1, \dots, D_k \in \PRV_\beta$ such that
  \[ B = D_1 \cdot q_1 + \dots + D_k \cdot q_k \]
  for some (unique) $q_1, \dots, q_k \in \KR$. Then $\p(B)$ is a greatest common divisor of $q_1, \dots, q_k$. We also have
  \[ B \cdot C = (D_1 \cdot C) \cdot q_1 + \dots + (D_k \cdot C) \cdot q_k. \]
  We now observe that $D_1 \cdot C, \dots, D_k \cdot C$ are $\Kbf$-linearly independent elements of $\PRV_\alpha$. Therefore, $\p(B \cdot C)$ is also a greatest common divisor of $q_1, \dots, q_k$, thus $\p(B \cdot C) = \p(B)$ by comparing the coefficients.
\end{proof}

\begin{lem}
  \label{lem:p-div-BC-div-B}
  For all $B \in \sRV$ and non-zero $C \in \sPRV$ we have $\p(B \cdot C) = \p(B)$.
\end{lem}
\begin{proof}
  Let $B \in \sRV$, $C \in \sPRV$ with $C \neq 0$. Clearly, we may assume $B \neq 0$. Write $B = \sum_\alpha B_\alpha$, $C = \sum_\alpha C_\alpha$. Let $\beta = \deg(B)$ and $\gamma = \deg(C)$. We shall prove by induction on $\beta$ that if $p \in \KR$ divides $B \cdot C$, then it also divides $B$. This implies the conclusion: we find that $p = \p(B \cdot C)$ divides $B$, thus it divides $\p(B)$, while on the other hand $\p(B) \mid \p(B \cdot C)$ by \prettyref{prop:pB-pC-div-pBC}, hence $\p(B \cdot C) = \p(B)$ by comparing the coefficients.

  Write $B \cdot C = \sum_\alpha D_\alpha$. By \prettyref{prop:B-div-C-div-C-alpha}, $p \mid D_{\beta \oplus \gamma} = B_\beta \cdot C_\gamma$, thus by \prettyref{lem:p-div-BC-div-B-RV}, $p \mid B_\beta$. Now let $B' = B - B_\beta$. If $B' = 0$, we are done. Otherwise, note that $p \mid B' \cdot C$ and $\deg(B') < \deg(B)$. By inductive hypothesis, $p \mid B'$, so $p \mid B$, as desired.
\end{proof}

\begin{lem}
  \label{lem:K-alg-cl-sPRV}
  $\Kbf$ is relatively algebraically closed in $\Frac(\sPRV)$, that is, every element of $\Frac(\sPRV)$ which is algebraic over $\Kbf$ is already in $\Kbf$.
\end{lem}
\begin{proof}
  Let $B, C \in \sPRV$ be non-zero elements such that $\frac{B}{C}$ is algebraic over $\Kbf$, namely there is a polynomial $Q(X) = X^d + k_{d-1}X^{d-1} + \dots + k_0 \in \Kbf[X]$ with $d > 0$ such that
  \[ Q\left(\frac{B}{C}\right) = \left(\frac{B}{C}\right)^d + k_{d-1}\left(\frac{B}{C}\right)^{d-1} +\dots + k_0 = 0. \]
  Assume that $d$ is minimal with this property. In particular, $Q(X)$ is irreducible in $\Kbf[X]$. Now rewrite the above equation as
  \[ B^d + k_{d-1}B^{d-1}C + \dots + k_1BC^{d-1} + k_0C^d = 0. \]
  Since the degree is a multiplicative valuation, two distinct non-zero terms of the above sum must have the same degree, that is, $e\deg(B) \oplus f\deg(C) = e'\deg(B) \oplus f'\deg(C)$ for some natural numbers $e,f,e',f'$ such that $e + f = e' + f' = d$ and $e \neq e'$. Since the Hessenberg sum is cancellative and torsion-free, it follows at once that $\deg(B) = \deg(C) = \beta$ for some $\beta \in \omega_1$.

  Now write $B = \sum_\alpha B_\alpha$, $C = \sum_\alpha C_\alpha$. Since $\deg(B - B_\beta), \deg(C - C_\beta) < \beta$, we must have
  \[ B_\beta^d + k_{d-1} B_\beta^{d-1} C_\beta + \dots + k_1 B_\beta C_\beta^{d-1} + k_0 C_\beta^d = 0. \]
  Let $b, c \in \KRR$ be series such that $\rv(b) = B_\beta$, $\rv(c) = C_\beta$. By \prettyref{lem:sum-rv-b1-bn}, the above equality implies that
  \[ \deg(b^d + k_{d-1}b^{d-1}c + \dots + k_1bc^{d-1} + k_0c^d) < d\beta. \]

   Now write $Q(X) = \prod_{i = 1}^d(X - \zeta_i)$ for some $\zeta_i$'s in the algebraic closure of $\Kbf$. Note that the definition of degree is independent of the field of the coefficients, so it can be naturally extended from $\KRR$ to $\Kbf^{\mathrm{alg}} \cdot \KRR = \Kbf^{\mathrm{alg}}((\Rb^{\leq 0}))$ while remaining a multiplicative valuation, where $\Kbf^{\mathrm{alg}}$ is the algebraic closure of $\Kbf$. Therefore, there is some $i = 1, \dots, d$ such that $\deg(b - \zeta_i c) < \beta$.

  In particular, there exists at least one $x \in \supp(b) \cup \supp(c)$ such that the coefficient of $t^x$ in $b - \zeta_i c$ is $0$. If $b_x, c_x$ are the coefficients of $t^x$ in respectively $b$ and $c$, we must have $c_x \neq 0$ and $\zeta_i = \frac{b_x}{c_x}$. In particular, $\zeta_i \in \Kbf$. Since $Q(X)$ is monic irreducible in $\Kbf[X]$, we must have $Q(X) = X - \zeta_i$. In turn, $B = \zeta_iC$, so $\frac{B}{C} \in \Kbf$, as desired.
\end{proof}

\begin{lem}
  \label{lem:p-fact-sPRV-KR}
  For all non-zero $p_1, p_2 \in \Frac(\sPRV)(\Rb^{\leq 0})$, if $p_1p_2 \in \KR$, then there is some $B \in \Frac(\sPRV)$ such that $p_1 \cdot B \in \KR$, $p_2 \cdot B^{-1} \in \KR$.
\end{lem}
\begin{proof}
  Let $p_1, p_2 \in \Frac(\sPRV)(\Rb^{\leq 0})$ be non-zero with $p_1p_2 \in \KR$. For the sake of notation, let $\Lbf = \Frac(\sPRV)$.

  There is a finite set of negative real numbers $x_1, \dots, x_n$ which are $\Zb$-linearly independent, and such that $p_1, p_2 \in \Lbf(H)$, where $H = \Nb x_1 + \dots + \Nb x_n$. Therefore, $\Kbf(H) \cong \Kbf[X_1, \dots, X_n]$ and $\Lbf(H) \cong \Lbf[X_1, \dots, X_n]$, with isomorphisms sending $t^{x_i}$ to the variable $X_i$ (as in \prettyref{fact:KS-is-UFD}). In particular, $\Kbf(H)$ and $\Lbf(H)$ are unique factorisation domains, with groups of units consisting of the non-zero elements of respectively $\Kbf$ and $\Lbf$. Moreover, since $\Kbf$ is relatively algebraically closed in $\Lbf$ by \prettyref{lem:K-alg-cl-sPRV}, each irreducible element of $\Kbf(H)$ remains irreducible in $\Lbf(H)$.

  Let us write $p_1p_2 = q_1 \cdots q_m$ where $q_1, \dots, q_m$ are irreducible elements of $\Kbf(H)$. Thus $p_1$ is a product of some of the factors $q_1, \dots, q_m$ and some unit element $B^{-1} \in \Lbf(H)$. Since the units of $\Lbf(H)$ are the non-zero elements of $\Lbf$, we find that $p_1 \cdot B \in \KR$ for some $B \in \Lbf = \Frac{\sPRV}$. Likewise, $p_2 \cdot C \in \KR$ for some $C \in \Lbf$.

  To conclude, note that $p_1p_2 \cdot B \cdot C \in \KR \subseteq \Lbf(\Rb^{\leq 0})$, hence $B \cdot C \in \Kbf$. It follows that $C \in B^{-1} \cdot \Kbf$, so $p_2 \cdot B^{-1} \in \KR$, as desired.
\end{proof}

\begin{rem}
  By the identification $\sPRV^{-1} \cdot \sRV = \Frac(\sPRV)(\Rb^{\leq 0})$ (\prettyref{rem:frac-sprv-GCD}), \prettyref{lem:p-fact-sPRV-KR} says in particular that for all $p, q \in \KR$, $p$ divides $q$ in the ring $\sPRV^{-1} \cdot \sRV$ if and only if $p$ divides $q$ in the ring $\KR$. Indeed, if $q = p \cdot q'$ for some $q' \in \sPRV^{-1} \cdot \sRV$, then for some $B \in \Frac(\sPRV)$ we have $p \cdot B \in \KR$, $q' \cdot B^{-1} \in \KR$. But then $B \in \Kbf$, so in particular $q' \in \KR$, hence $p$ divides $q$ in $\KR$.
\end{rem}

\begin{cor}
  \label{cor:p-primal-sRV}
  Every $p \in \KR$ is primal in $\sRV$.
\end{cor}
\begin{proof}
  Let $p \in \KR$, and suppose that $p$ divides $B \cdot C$ for some $B, C \in \sRV$. Since $\sPRV^{-1} \cdot \sRV = \Frac(\sPRV)(\Rb^{\leq 0})$ is a GCD domain, we know that there are $p_1, p_2, B', C' \in \sPRV^{-1} \cdot \sRV$ such that $p = p_1p_2$ and $B = p_1 \cdot B'$, $C = p_2 \cdot C'$. By \prettyref{lem:p-fact-sPRV-KR}, we may further assume that $p_1, p_2 \in \KR$.

  Now write $B'$, $C'$ as fractions $B' = \frac{M}{D}$, $C' = \frac{N}{E}$ for some $M, N \in \sRV$ and non-zero $D, E \in \sPRV$. Then $p_1 \mid B \cdot D$, $p_2 \mid C \cdot E$. Since $D, E \in \sPRV$, it follows by \prettyref{lem:p-div-BC-div-B} that $p_1 \mid B$, $p_2 \mid C$, showing that $p$ is primal, as desired.
\end{proof}

\begin{cor}
  \label{cor:pBC-pB-pC}
  For all $B, C \in \sRV$, $\p(BC) = \p(B)\p(C)$.
\end{cor}
\begin{proof}
  Let $B, C \in \sRV$. We already know that $\p(B)\p(C) \mid \p(BC)$ from \prettyref{prop:pB-pC-div-pBC}. We claim that $\p(BC) \mid \p(B)\p(C)$. Recall that by definition $\p(BC) \mid BC$, so by \prettyref{cor:p-primal-sRV} we can write $\p(BC) = p_1p_2$ for some $p_1, p_2 \in \KR$ such that $p_1 \mid B$, $p_2 \mid C$. But then $p_1 \mid \p(B)$, $p_2 \mid \p(C)$, so $\p(BC) \mid \p(B)\p(C)$. Therefore, $\p(BC) = k\p(B)\p(C)$ for some non-zero $k \in \Kbf$. By comparing the coefficients we deduce that $\p(BC) = \p(B)\p(C)$, as desired.
\end{proof}

\begin{prop}
  \label{prop:pbc-pb-pc}
  For all $b, c \in \KRR$, $\p(bc) = \p(b)\p(c)$.
\end{prop}
\begin{proof}
  Let $b, c \in \KRR$. We reason by induction on $\deg(b)$ and $\deg(c)$. We already know that $\p(b)\p(c) \mid \p(bc)$. We claim that $\p(bc) \mid \p(b)\p(c)$. After dividing $b$ and $c$ by $\p(b)$ and $\p(c)$, we may assume that $\p(b) = \p(c) = 1$, so our claim reduces to proving that $\p(bc) \in \Kbf$.

  Let $q$ be a greatest common divisor between $\p(bc)$ and $\p(\rv(b))$. By definition of $\p(\rv(b))$, we can write $b = \p(\rv(b))b' + d$ where $b', d \in \KRR$ are such that $\deg(d) < \deg(b)$. Since $q \mid \p(bc) \mid bc$ and $q \mid \p(\rv(b))$, we must have $q \mid dc$, hence $q \mid \p(dc)$. By inductive hypothesis, $q \mid \p(dc) = \p(d)\p(c) = \p(d)$. Therefore, $q \mid d$. In turn, $q \mid b$, which means that $q \mid \p(b) = 1$, so $q \in \Kbf$. Therefore, $\p(bc)$ and $\p(\rv(b))$ are coprime.

  By symmetry, $\p(bc)$ and $\p(\rv(c))$ are also coprime. On the other hand, $\p(bc) \mid \p(\rv(bc))$, and $\p(\rv(bc)) = \p(\rv(b))\p(\rv(c))$ by \prettyref{cor:pBC-pB-pC}. Since $\p(bc)$ is coprime with both $\p(\rv(b))$ and $\p(\rv(c))$, we must have $\p(bc) \in \Kbf$, proving the claim.

  Therefore, for all $b, c \in \KRR$, $\p(bc) = k\p(b)\p(c)$ for some $k \in \Kbf$. By comparing the coefficients we deduce that $\p(bc) = \p(b)\p(c)$, as desired.
\end{proof}

\begin{cor}
  \label{cor:KR-primal-KRR}
  Every $p \in \KR$ is primal in $\KRR$.
\end{cor}
\begin{proof}
  Let $p \in \KR$ and $b, c \in \KRR$. Suppose that $p \mid bc$. By \prettyref{prop:pbc-pb-pc}, $p \mid \p(b)\p(c)$. Since $\KR$ is a GCD domain, there are $p_1, p_2 \in \KR$ such that $p = p_1p_2$ and $p_1 \mid \p(b) \mid b$, $p_2 \mid \p(c) \mid c$, thus $p$ is primal.
\end{proof}

\subsection{Uniqueness of the factor with finite support}

It now follows at once that if a series $b \in \KRR$ factors into a product of one series of finite support and other irreducible series of infinite support, the factor of finite support is unique up to multiplication by an element of $\Kbf$, which is the final ingredient towards proving \prettyref{main:KRR}.

\begin{thm}
  \label{thm:KRR-fact-unique}
  For all non-zero $b \in \KRR$, there exist $n \in \Nb$, irreducible series $c_1, \dots, c_n \in \KRR$ with infinite support, and $k \in \Kbf$ such that
  \[ b = k \p(b) c_1 \cdots c_n. \]
  Moreover, $n$ is at most the number of terms in the Cantor Normal Form of $\deg(b)$, and $p$ is unique up to multiplication by elements of $\Kbf$.
\end{thm}
\begin{proof}
  Let $b \in \KRR$. The existence of the desired factorisation and the bound on $n$ are the conclusions of \prettyref{prop:KRR-fact-exists}, so we only need to check for uniqueness.

  Suppose that $b = pc_1 \cdots c_n$ for some $c_1, \dots, c_n \in \KRR$ irreducible with infinite support and $p \in \KR$. Since each $c_i$ is irreducible, we have $\p(c_i) = 1$. Moreover, $\p(p) = kp$ for some $k \in \Kbf$. Therefore,
  \[ \p(b) = \p(p)\p(c_1) \cdots \p(c_n) = \p(p) = kp, \]
  and the conclusion follows.
\end{proof}

\begin{proof}[Proof of \prettyref{main:KRR}]
  Simply combine \prettyref{thm:KRR-fact-unique} with \prettyref{prop:ritt}.
\end{proof}

\begin{cor}
  \label{cor:KRR-pS-equiv}
  The following are equivalent:
  \begin{enumerate}
      \item\label{item:KRR-pS?} $\KRR$ is a pre-Schreier domain;
      \item\label{item:KRR-GCD?} $\KRR$ is a a GCD domain;
      \item\label{item:KRR-irred?} in $\KRR$, every irreducible series with infinite support is prime.
  \end{enumerate}
\end{cor}
\begin{proof}
  \prettyref{item:KRR-GCD?} $\Rightarrow$ \prettyref{item:KRR-pS?} is \prettyref{fact:GCD-is-pre-Schreier}. \prettyref{item:KRR-pS?} $\Rightarrow$ \prettyref{item:KRR-irred?} is clear: if $\KRR$ is a pre-Schreier domain, every element is primal, and so every irreducible series is prime.

  Assume now \prettyref{item:KRR-irred?} Then for every $b \in \KRR$ the factorisation $b = pc_1 \cdots c_n$ of \prettyref{thm:KRR-fact-unique} is unique up to reordering the factors and to multiplication by elements of $\Kbf$. It follows that a series $d \in \KRR$ divides $b$ if and only if it is a product of some of factors $c_1, \dots, c_n$, a factor of $p$, and some non-zero element of $\Kbf$. Therefore, given two series $b, c \in \KRR$, their greatest common divisor is a greatest common divisor of $\p(b)$ and $\p(c)$, multiplied by the irreducible series with infinite support that appear in both factorisations up to multiplication by elements of $\Kbf$. Thus $\KRR$ is a GCD domain, proving \prettyref{item:KRR-irred?} $\Rightarrow$ \prettyref{item:KRR-GCD?}.
\end{proof}

\subsection{The non-complete case}
\label{sub:non-complete}
To find a general factorisation theorem for series beyond $\KRR$, we shall encounter rings of the form $\KHH$ where $H$ is a divisible subgroup of $\Rb$, which may or may not be all of $\Rb$; equivalently, $H$ is a divisible Archimedean group, possibly non-complete. When $H \neq \Rb$, the conclusion of \prettyref{thm:KRR-fact-unique} is false if rephrased in $\KHH$, and actually $\KHH$ is \emph{not} a pre-Schreier domain (see \prettyref{prop:G-not-complete-not-pS}).

We may recover a factorisation theorem by weakening the notion of irreducibility: we shall say that a series $b \in \KHH$ is \textbf{almost irreducible} if for every divisor $c \in \KHH$ of $b$, if $c$ is not a monomial, then $\frac{b}{c}$ is.

\begin{rem}
  \label{rem:almost-irred}
  Clearly, all irreducible series are almost irreducible. Moreover, if $b$ is almost irreducible and $\sup(b) = 0$, then $b$ is irreducible, since its only monomial divisors are elements of $\Kbf$. Conversely, if $\sup(b) < 0$, then $b$ is not irreducible: since $H$ is divisible, there is some $x \in H$ such that $\sup(b) < x < 0$, and so $t^x$ is a non-trivial divisor of $b$.
\end{rem}

\begin{lem}
  \label{lem:p-KR-pH-KH}
  For all $p \in \KR$, there exists a unique $p_H \in 1 + \Kbf(H^{<0})$ such that for all $q \in 1 + \Kbf(H^{<0})$, $q$ divides $p$ if and only if $q$ divides $p_H$.
\end{lem}
\begin{proof}
  Clearly, we may assume that $p \neq 0$. Apply \prettyref{prop:ritt} to $p$ and write $p = t^xc_1 \cdots c_n$, where for each $i$, $0 \in \supp(c_i)$, and $\lspan{\supp(c_i)}$ has dimension $1$ or $c_i$ is irreducible. Let
  \[ S = \{i \in \{1,\dots,n\} : \lspan{\supp(c_i)} \subseteq H\} \]
  and let $p_H = k \prod_{i \in S} c_i$, where $k$ is chosen so that $p_H \in 1 + \Kbf(H^{<0})$.

  Suppose that $q \in 1 + \Kbf(H^{<0})$ divides $p$. Since $\KH$ is a GCD domain, we can write $q = q_1q_2$ where $q_1 \in \KH$ is a greatest common divisor of $q$ and $p_H$ and $q_2 = \frac{q}{q_1} \in \KH$ has no common divisor with $p_H$. In particular, $q_2$ divides $\frac{p}{p_H} = t^x\prod_{i \notin S}c_i$. Since $\sup(q_2) = 0$, $q_2$ divides $\prod_{i \notin S}c_i$. Since $\KR$ is a GCD domain, we may write $q_2 = \prod_{i \notin S}r_i$, for $r_i \in \KR$, where each $r_i$ divides the corresponding $c_i$. Note in particular that $\sup(r_i) = 0$ for every $i \notin S$.

  By \prettyref{lem:KG-supp(bc)}, the support of $r_i$ is contained in $H$ (because $r_i$ divides $q$) and in $\lspan{\supp(c_i)}$ (because $r_i$ divides $c_i$) for every $i \notin S$. If $c_i$ is irreducible, then $\supp(r_i) = \supp(c_i)$ or $\supp(r_i) = \{0\}$, because $r_i = kc_i$ or $r_i = k$ for some $k \in \Kbf \setminus \{0\}$. If $c_i$ is not irreducible, then $\lspan{\supp(r_i)} = \lspan{\supp(c_i)}$ or $\supp(r_i) = \{0\}$ by dimension considerations. In both cases, since $\supp(c_i) \nsubseteq H$, we must have $\supp(r_i) = \{0\}$ for $i \notin S$, and therefore $q_2 \in \Kbf$. It follows that $q$ divides $p_H$.

  Finally, if $p'$ is another series in $1 + \Kbf(H^{<0})$ satisfying the conclusion, then by construction $p'$ divides $p_H$ and $p_H$ divides $p'$, thus $p' = kp_H$ for some $k \in \Kbf$. By comparing the coefficients, we must have $k = 1$, thus $p' = p_H$.
\end{proof}

\begin{cor}
  For all $b \in \KRR$ and $q \in 1 + \Kbf(H^{<0})$, $q$ divides $b$ if and only if $q$ divides ${\p(b)}_H$.
\end{cor}

\begin{cor}
  \label{cor:pHb-pHc-div-pHbc}
  For all $p, q \in \KR$, $(pq)_H = p_Hq_H$.
\end{cor}
\begin{proof}
  Let $p, q \in \KR$, and $r \in \KH$. We first observe that $p_Hq_H$ divides $pq$, so $p_Hq_H$ divides $(pq)_H$. For the converse, since $\KH$ is a GCD domain, we can write $(pq)_H = r_1r_2$ for some $r_1, r_2 \in \KH$ such that $r_1$ divides $p$ and $r_2$ divides $q$ in $\KH$. Clearly, we may also assume $r_1, r_2 \in 1 + \Kbf(H^{<0})$. Then $r_1$ divides $p_H$ and $r_2$ divides $q_H$, so $(pq)_H$ divides $p_Hq_H$. By comparing the coefficients, we find that $(pq)_H = p_Hq_H$, as desired.
\end{proof}

\begin{cor}
  For all $b, c \in \KRR$, $\p(bc)_H = {\p(b)}_H{\p(c)}_H$.
\end{cor}

\begin{cor}
  \label{cor:KH-primal-KHH}
  Every $p \in 1 + \Kbf(H^{<0})$ is primal in $\KHH$.
\end{cor}
\begin{proof}
  Let $p \in 1 + \Kbf(H^{<0})$ and $b, c \in \KHH$, and suppose that $p$ divides $bc$. By \prettyref{cor:pHb-pHc-div-pHbc}, $p \mid bc$ if and only if $p \mid {\p(b)}_H{\p(c)}_H$. Since $\KH$ is a GCD domain, there are $p_1, p_2 \in \KH$ such that $p' = p_1p_2$ and $p_1 \mid {\p(b)}_H \mid b$, $p_2 \mid {\p(c)}_H \mid c$, showing that $p$ is primal in $\KHH$. Since every element of $t^H\KHH$ is the product of an element of $\KHH$ and a unit of $t^H\KHH$, it easily follows that $p$ is primal in $t^H\KHH$ as well.
\end{proof}

\begin{thm}
  \label{thm:KHH-fact-unique}
  Let $b \in \KHH$ with $b \neq 0$. Then there are $x \in H^{\leq 0}$, $n \in \Nb$, $c_1, \dots, c_n \in \KHH$, and a unique $p \in 1 + \Kbf(H^{<0})$ such that $b = pt^xc_1 \cdots c_n$, where each $c_i$ is almost irreducible with infinite support.

  If moreover $\sup(b) \in H$, then we may take $c_1, \dots, c_n$ irreducible, in which case $x$ is unique (and equal to $\sup(b)$).
\end{thm}
\begin{proof}
  Let $b \in \KHH$ be a non-zero series. We proceed as in the proof of \prettyref{prop:KRR-fact-exists}. Let $b' = \frac{b}{{\p(b)}_H} \in \KHH$. Note that $\p(b')_H$ is necessarily $1$. We work by induction on $\deg(b)$. If $\deg(b) = 0$, then $b'$ is of the form $kt^x$ for some $k \in \Kbf$ and $x \in H$, hence it is a unit, and we are done.

  Assume $\deg(b) > 0$. If $b'$ is almost irreducible, we are done. Otherwise, $b' = cd$ for some $c, d \in \KHH$ not of the form $kt^x$. Since $\p(b')_H = 1$, $c, d$ are not in $\KH$, hence $\deg(c), \deg(d) > 0$. Since $\deg(b') = \deg(c) \oplus \deg(d)$, it follows that $\deg(c) < \alpha$, $\deg(d) < \alpha$. Note moreover that $\p(c)_H = \p(d)_H = 1$ by \prettyref{cor:pHb-pHc-div-pHbc}.

  By inductive hypothesis, $c$ and $d$ can be written as products of almost irreducible series with infinite support and some $t^{x'}$ with $x' \in H^{\leq 0}$. Therefore, $b'$ is also a product of the same form, as desired.

  For the uniqueness of the factor $p$, suppose that $b = pt^xc_1 \cdots c_n$ is a factorisation of $b$ as in the conclusion. It then suffices to note that
  \[ \p(b)_H = \p(p)_H \p(t^x)_H \p(c_1)_H \cdots \p(c_n)_H = \p(p)_H = p. \]

  Finally, suppose $\sup(b) \in H$ and let $x = \sup(b)$. We then have $b' \coloneqq \frac{b}{\p(b)_H t^x} \in \KHH$. We apply the previous conclusion to $b'$ and find a factorisation $b' = c_1 \cdots c_n$. Since $\sup(b')  = \sup(c_1) + \dots + \sup(c_n) = 0$, we have $\sup(c_1) = \dots = \sup(c_n) = 0$, hence $c_1, \dots, c_n$ are irreducible in $\KHH$. If $b = p' t^y c_1' \cdots c_n'$ is another factorisation of the same form, we have already observed that $p' = \p(b)_H$, and since the series $c_i'$ are irreducible, we have $\sup(c_i') = 0$, hence $x = \sup(b)$.
\end{proof}

Note that \prettyref{thm:KRR-fact-unique} is the special case of the above statement in which $\sup(b) \in H = \Rb$ for every $b \in \KHH$, because $\Rb$ is complete.

\section{New irreducibles and new primes}
\label{sec:new-irreducibles-primes}

As implicitly hinted by the statement of \prettyref{cor:KRR-pS-equiv}, it is still an open question whether all irreducible series of $\KRR$ are prime. In fact, it is not even clear how many series are irreducible in the first place. As explained in \prettyref{sub:previous-results}, \cite{Ber2000} only implies the existence of irreducible series of order types $\omega^{\omega^\alpha}$ and $\omega^{\omega^\alpha}+1$, and \cite{PS2006,LM2017} only find, with substantial effort, irreducible series of order types respectively $\omega^2,\omega^2+1$ and $\omega^3,\omega^3+1$. For primes, the situation is considerably less clear, and we only know of primes of order type $\omega, \omega+1$ from \cite{Pit2001}, plus all the irreducible series of $\KR$ by \prettyref{cor:KR-primal-KRR}.

In this section, we exploit the tools of the previous sections to find new classes of irreducibles and primes, and in particular prove Theorems~\ref{main:irreducible},~\ref{main:prime}.

\subsection{Finding new irreducibles}
\label{sub:criterion-irred}

Recall that Berarducci proved that every series $b \in \KRR$ whose order type $\ot(b)$ is multiplicatively principal, i.e.\ of the form $\omega^{\omega^\alpha}$, and with $\sup(b) = 0$ is irreducible \cite[Thm.\ 10.5]{Ber2000}. Here we prove a similar, but more general irreducibility criterion (\prettyref{main:irreducible}) and show some techniques produce new irreducibles.

\begin{lem}
  \label{lem:irreducibility}
  For all $b \in \KRR$, if $\frac{\rv(b)}{\p(\rv(b))}$ is irreducible and $\p(b) = 1$, then $b$ is irreducible.
\end{lem}
\begin{proof}
  Let $b \in \KRR$ as in the hypothesis, and suppose that $b = cd$ for some $c, d \in \KRR$. Then $\rv(b) = \rv(c) \cdot \rv(d)$. By \prettyref{cor:pBC-pB-pC}, we can divide both sides by $\p(\rv(b)) = \p(\rv(c))\p(\rv(d))$. By the hypothesis, one of $\frac{\rv(c)}{\p(\rv(c))}$, $\frac{\rv(d)}{\p(\rv(d))}$ is a unit, thus it is an element of $\Kbf$. Therefore, one of $c, d \in \KRR$ has finite support. Since $\p(b) = 1$, it follows that one of $c, d$ is a unit.
\end{proof}

The element $\frac{\rv(b)}{\p(\rv(b))}$ is automatically irreducible when $\deg(b)$ is additively principal: in this case, any factorisation of $\frac{\rv(b)}{\p(\rv(b))}$ must have a factor of degree $0$, which must then be an element of $\Kbf$ since we have already divided out the maximal divisor of degree $0$. We thus obtain the following.

\begin{cor}
  \label{cor:irreducibility}
  For all $b \in \KRR$, if $\deg(b)$ is additively principal and $\p(b) = 1$, then $b$ is irreducible.
\end{cor}

The condition $\p(b) = 1$ can now be arranged easily in order to find plenty of irreducible series. For a start, this is always true for the series satisfying the assumptions of \prettyref{main:irreducible}.

\begin{proof}[Proof of \prettyref{main:irreducible}]
  Let $b \in \KRR$ have order type $\omega^{\omega^\alpha} \hatplus \beta$, where $\beta < \omega^{\omega^\alpha}$, and not divisible by $t^x$ for any $x \in \Rb^{<0}$. Thus $\deg(b) = \omega^{\alpha}$ is additively principal.

  By \prettyref{prop:normal-form}, we can write $b = b't^y + b''$ with $b'$ principal of order type $\omega^{\omega^{\alpha}}$ and $b''$ of order type $\beta$ with $\supp(b'') \geq y$. In particular, $\deg(b'') < \omega^{\alpha} = \deg(b')$. It follows that $\rv(b) = \rv(b't^y) = \rv(b') \cdot t^y$. Since $\rv(b')$ is principal, $\p(\rv(b')) = 1$ (see \prettyref{rem:pB}), hence $\p(\rv(b)) = t^y$.

  Since $\p(b) \mid \p(\rv(b)) = t^y$, $\p(b)$ must be of the form $t^x$ for some $x \geq y$, thus $x = 0$, so $\p(b) = 1$. By \prettyref{cor:irreducibility}, $b$ is irreducible.
\end{proof}

A similar argument lets us find irreducible series of order type $\alpha$ whenever $\deg(\alpha)$ is additively principal.

\begin{lem}
  \label{lem:b1-bl-almost-disjoint}
  Let $b_1, \dots, b_\ell \in \KRR$ be series of the same degree $\alpha$. Suppose that $\deg\ot(\supp(b_i) \cap \supp(b_j)) < \alpha$ for all $i \neq j$. Then $\rv(b_1), \dots, \rv(b_\ell)$ are $\Kbf$-linearly independent.
\end{lem}
\begin{proof}
  Let $b_1, \dots, b_\ell \in \KRR$ satisfy the assumptions. Suppose that
  \[ \rv(b_1) \cdot k_1 + \cdots + \rv(b_\ell) \cdot k_\ell = 0 \]
  for some $k_1,\dots,k_\ell \in \Kbf$. By \prettyref{lem:sum-rv-b1-bn},
  \[ \deg(b_1k_1 + \cdots + b_\ell k_\ell) < \alpha. \]
  Let $A = \supp(b_1k_1 + \cdots + b_\ell k_\ell)$.

  Suppose by contradiction that $k_i \neq 0$ for some $i$. For every $x \in \supp(b_i) \setminus A$, there is some $j \neq i$ such that $x \in \supp(b_j)$. It follows that
  \[ \supp(b_i) \subseteq A \cup \bigcup_{j \neq i} (\supp(b_i) \cap \supp(b_j)). \]
  Thus, by \prettyref{fact:ot-ABC}\prettyref{item:fact-ot-union},
  \[ \ot(b_i) = \ot(\supp(b_i)) \leq \ot(A) \oplus \bigoplus_{j \neq i} \ot(\supp(b_i) \cap \supp(b_j)) < \omega^\alpha, \]
  thus $\deg(b_i) < \alpha$, a contradiction.
\end{proof}

\begin{prop}
  \label{prop:irreducible-principal-degree}
  For every $\alpha, \beta \in \omega_1$ with $\omega^{\omega^\beta} \leq \alpha < \omega^{\omega^\beta \hatplus 1}$, there is an irreducible series $b \in \KRR$ such that $\ot(b) = \alpha$.
\end{prop}
\begin{proof}
  Let $\alpha \in \omega_1$ be as in the hypothesis and write
  \[ \alpha = \omega^{\beta_1} \hatplus \cdots \hatplus \omega^{\beta_n} \]
  in Cantor Normal Form, where by assumption $\beta_1 = \omega^\beta$ for some $\beta \in \omega_1$. Let $\ell \geq 1$ be maximal such that $\beta_1 = \dots = \beta_\ell = \deg(\alpha)$. Pick some principal series $b_1, \dots, b_n \in \KRR$ of order types respectively $\omega^{\beta_1}, \dots, \omega^{\beta_n}$, satisfying the property that $\rv(b_1), \dots, \rv(b_\ell)$ are $\Kbf$-linearly independent. For instance, by \prettyref{lem:b1-bl-almost-disjoint}, we could just take $b_1, \dots, b_\ell$ with pairwise disjoint supports, which can be easily arranged.

  Now let
  \[ b = b_1t^{-n+1} + b_2t^{-n+2} + \dots + b_n. \]
  By construction, $\ot(b) = \alpha$ and $\sup(b) = 0$. Moreover, by the assumption of $\Kbf$-linear independence, we have
  \[ \rv(b) = \rv(b_1t^{-n+1} + \dots + b_\ell t^{-n+\ell}) = \rv(b_1) \cdot t^{-n+1} + \dots + \rv(b_\ell) \cdot t^{-n+\ell}. \]
  It follows that $\p(\rv(b)) = t^{-n+\ell}$ by \prettyref{prop:RV-alpha-tensor}. Since $\p(b)$ divides $\p(\rv(b))$, and $\sup(b) = 0$, we must have $\p(b) = 1$. Therefore, $b$ is irreducible by \prettyref{cor:irreducibility}.
\end{proof}

\begin{exa}
  \label{exa:irreducible-omega2+k+1}
  For any $x \in \Gbf^{<0}$, $\sum_{n \in \Nb} t^{\frac{x}{3n+1}}$ and $\sum_{n \in \Nb} t^{\frac{x}{3n+2}}$ are principal series of degree $1$, with disjoint supports, thus $\Kbf$-linearly independent by \prettyref{lem:b1-bl-almost-disjoint}. Then, by following the above proof, we find that
  \[ b = \left(\sum_{n \in \Nb} t^{\frac{x}{3n+1}}\right)t^{-3} + \left(\sum_{n \in \Nb} t^{\frac{x}{3n+2}}\right)t^{-2} + t^{\frac{(k-1)x}{k}} + \dots + t^{\frac{x}{k}} + 1 \]
  is an irreducible series of order type $\omega \hatplus \omega \hatplus k \hatplus 1$ for every $k \in \Nb$.
\end{exa}

Is is reasonable to expect that for \emph{every} $\alpha \in \omega_1$ there is an irreducible series $b \in \KRR$ with order type $\ot(b) = \alpha$. For instance, one would expect that if $\supp(b)$ is chosen sufficiently randomly but with order type $\alpha$, then $b$ is irreducible. It would be interesting to investigate if the techniques of this paper, combined with previous strategies from \cite{PS2006,LM2017}, can shed more light on this problem.

\subsection{A broader criterion for primality}
\label{sub:criterion-prime}

In this subsection, we extend Pitteloud's primality criterion in order to find more prime series of degree $1$. We reuse some of the arguments in \cite{Pit2001}, so we first translate them in our language.

\begin{prop}
  \label{prop:principal-vj}
  A non-zero series $b \in \KRR$ is principal if and only if $\ot(b) = v_J(b)$.
\end{prop}
\begin{proof}
  Let $b \in \KRR$ be non-zero. We distinguish three cases, according to~\prettyref{eq:vJ} in~\prettyref{sub:order-value}.

  \textbf{Case $b \in \Jbf$.} In this case, $b$ is not principal, since $\sup(b) < 0$, and moreover $\ot(b) > 0 = v_J(b)$.

  \textbf{Case $b \in (\Jbf + \Kbf) \setminus \Jbf$.} We have $0 \in \supp(b)$. It follows that $\ot(b) = \alpha \hatplus 1$ for some ordinal $\alpha$. Therefore, $b$ is principal if and only if $\ot(b) = 1 = v_J(b)$.

  \textbf{Case $b \notin \Jbf + \Kbf$.} In particular, $\ot(b)$ is infinite. Let $c \in \KRR$. Note that $b - c \in \Jbf + \Kbf$ if and only if there are $x \in \Rb^{<0}$ and $k \in \Kbf$ such that $c_{\geq x} = b_{\geq x} + k$. Since $\ot(c_{\geq x}) \leq \ot(c)$, we conclude that $v_J(b)$ is the minimum of $\ot(b_{\geq x} + k)$ for $x \in \Rb^{<0}$ and $k \in \Kbf$.

  If $b$ is principal, then $\ot(b_{\geq x}) = \ot(b)$ for any $x \in \Rb^{<0}$, and $\ot(b_{\geq x} + k) = \ot(b) \hatplus 1$ for any non-zero $k \in \Kbf$, so $v_J(b) = \ot(b)$. If $b$ is not principal, $\ot(b) = \omega^{\alpha} \hatplus \beta$ with $0 < \beta \leq \omega^{\alpha}$. Write $b = b' + b''$ with $\supp(b') < \supp(b'')$ and $\ot(b'') = \beta$. By construction, $\sup(b - b'') = \sup(b') < 0$, so $v_J(b) \leq \ot(b'') = \beta < \ot(b)$.
\end{proof}

\begin{defn}[{\cite[p.~1209]{Pit2001}}]
  Given $\alpha \in \omega_1$, let $\Jbf_{\omega^{\alpha}}$ be the $\Kbf$-vector space $\Jbf_{\omega^{\alpha}} \coloneqq \{e \in \KRR \,:\, v_J(e) < \omega^{\alpha}\}$. Moreover, write $b \mid c \mod \Jbf_{\omega^\alpha}$ if there exists $d \in \KRR$ such that $c \equiv bd \mod \Jbf_{\omega^\alpha}$.
\end{defn}

For clarity, note for instance that $\Jbf_{\omega^0} = \Jbf_1 = \Jbf$, $\Jbf_{\omega^1} = \Jbf_\omega = \Jbf + \Kbf$.

\begin{lem}
  \label{lem:rv-Jomega}
  Let $b, c \in \KRR$ be two principal series, with $\deg(c) = \alpha$. Then $\rv(b) = \rv(c)$ if and only if $b \equiv c \mod \Jbf_{\omega^{\alpha}}$.
\end{lem}
\begin{proof}
  If $\rv(b) = \rv(c)$, then $\deg(b - c) < \alpha$, so $\deg(b) = \alpha$ and $v_J(b - c) \leq \omega^{\deg(b - c)} < \omega^{\alpha}$, hence $b - c \in \Jbf_{\omega^{\alpha}}$, or in other symbols, $b \equiv c \mod \Jbf_{\omega^{\alpha}}$. If $\rv(b) \neq \rv(c)$, then $b - c$ is principal of degree $\max\{\deg(b), \deg(c)\}$: indeed, this is trivial if $\deg(b) \neq \alpha$, and it follows from \prettyref{prop:sum-principal} if $\deg(b) = \alpha$. Therefore, $v_J(b - c) = \ot(b - c) = \omega^{\alpha}$, hence $b - c \notin \Jbf_{\omega^{\alpha}}$.
\end{proof}

\begin{rem}
  \label{rem:RV-of-vJ}
  Every element in the quotient space $\Jbf_{\omega^{\alpha+1}} / \Jbf_{\omega^{\alpha}}$ can be represented as the class $b + \Jbf_{\omega^{\alpha}}$ for some principal series $b \in \KRR$ of degree $\alpha$. By \prettyref{lem:rv-Jomega}, it follows at once that the $\Kbf$-vector space $\PRV_\alpha$ can be alternatively presented as the quotient $\Jbf_{\omega^{\alpha+1}}/\Jbf_{\omega^{\alpha}}$. On the other hand, the quotient $\Jbf_{\omega^{\alpha+1}}/\Jbf_{\omega^{\alpha}}$ is also the module $\RV_m$ for the semi-valuation $w = v_J$ and for $m = \omega^{\alpha}$ (the verification is left to the reader). In particular, $\PRV$ is the $\RV$ monoid of the semi-valuation $v_J$.
\end{rem}

\begin{cor}
  \label{cor:rv-mid-Jomega}
  Let $b, c \in \KRR$ be two principal series, with $\deg(c) = \alpha$. Then $b \mid c \mod \Jbf_{\omega^{\alpha}}$ if and only if $\rv(b) \mid \rv(c)$.
\end{cor}
\begin{proof}
  If $b \in \Kbf$, the conclusion is trivial, so assume otherwise. In particular, we also assume $\alpha > 0$. Suppose that $b \mid c \mod \Jbf_{\omega^{\alpha}}$, namely that $c \equiv bd \mod \Jbf_{\omega^{\alpha}}$ for some $d \in \KRR$. Note that for any $x \in \Rb^{<0}$, $bd_{\geq x} \equiv bd \mod \Jbf$, so in particular $c \equiv bd_{\geq x} \mod \Jbf_{\omega^{\alpha}}$, since $\Jbf = \Jbf_1 \subseteq \Jbf_{\omega^{\alpha}}$. If $x$ is sufficiently close to $0$, then $d_{\geq x} = d' + k$ for some principal series $d'$, or possibly $d' = 0$, and $k \in \Kbf$. We replace $d$ with $d' + k$, so that $bd = bd' + bk$ is also principal. Then $\rv(c) = \rv(bd)$ by \prettyref{lem:rv-Jomega}, so $\rv(b) \mid \rv(c)$.

  Conversely, if $\rv(b) \mid \rv(c)$, then there exists $d \in \KRR$ such that $\rv(c) = \rv(b) \cdot \rv(d) = \rv(bd)$. Then $\rv(d)$ is principal by \prettyref{cor:BC-PRV-B-C-PRV}, so we may assume that $d$ is principal. Thus $bd$ is also principal, and $c \equiv bd \mod \Jbf_{\omega^{\alpha}}$ by \prettyref{lem:rv-Jomega}, so $b \mid c \mod \Jbf_{\omega^{\alpha}}$.
\end{proof}

We may thus reinterpret the key step in Pitteloud's proof as a statement about primality in $\sRV$, as \prettyref{cor:rv-mid-Jomega} translates between divisibility modulo $\Jbf_{\omega^\alpha}$ and divisibility in $\RV$.

\begin{prop}[{\cite[Prop.\ 3.2]{Pit2001}}]
  \label{prop:pitteloud}
  Let $a, b, c, d \in \KRR$ be such that $\vj(a) = \omega$ and assume that $a^k b \equiv c^l d \mod \Jbf_{\vj(a^{k}b)}$ with $k,l > 0$. Then $a \mid c \mod \Jbf_{\vj(c)}$ or $a \mid d \mod \Jbf_{\vj(d)}$.
\end{prop}

By \prettyref{cor:rv-mid-Jomega}, the above statement says that if $\rv(a^kb) = \rv(c^ld)$, then $rv(a) \mid \rv(c)$ or $\rv(a) \mid \rv(d)$.

\begin{cor}
  \label{cor:PRV-1-prime-in-PRV}
  For all $B \in \PRV$ of degree $1$ and $C, D \in \PRV$, if $B \mid C \cdot D$, then $B \mid C$ or $B \mid D$.
\end{cor}
\begin{proof}
  Let $B, C, D$ as in the hypothesis. The conclusion is trivial for $B = 0$, so assume $B \neq 0$. Then $B = \rv(b)$ for some principal $b \in \KRR$ of degree $1$, and in particular with $v_J(b) = \omega$. Write $C = \rv(c)$, $D = \rv(d)$ with $c, d \in \KRR$ principal.

  Assume $B \mid C \cdot D$. By \prettyref{cor:rv-mid-Jomega}, this means that $b \mid cd \mod \Jbf_{\vj(cd)}$, so that there exists $e$ such that $be \equiv cd \mod \Jbf_{\vj(cd)}$. Note that we must have $\vj(be) = \vj(cd)$. By \prettyref{prop:pitteloud}, $b \mid c \mod \Jbf_{\vj(c)}$ or $b \mid d \mod \Jbf_{\vj(d)}$. By \prettyref{cor:rv-mid-Jomega}, this means that $B \mid C$ or $B \mid D$.
\end{proof}

\begin{cor}
  \label{cor:PRV-1-prime-in-RV}
  Every $B \in \PRV$ of degree $1$ is prime in $\sPRV$ and in $\sRV$.
\end{cor}
\begin{proof}
  Let $B \in \PRV_1$, $C, D \in \sRV$. Suppose first that $C, D \in \PRV$. Write $C = \sum_{\alpha} C_{\alpha}$, $D = \sum_{\alpha} D_{\alpha}$. Let $\beta = \deg(C)$, $\gamma = \deg(D)$. Then clearly $B$ divides $C_{\beta} \cdot D_{\gamma}$. By \prettyref{cor:PRV-1-prime-in-PRV}, $B \mid C_{\beta}$ or $B \mid D_{\gamma}$. Assume we are in the first case. Then $B$ divides $(C - C_\beta) \cdot D = \sum_{\alpha < \beta} C_{\alpha} \cdot D$. By induction on $\beta$ and $\gamma$, either $B \mid D$, or $B \mid (C - C_{\alpha})$, hence $B \mid C$, proving the conclusion.

  For the general case of $C, D \in \sRV$, it suffices to recall that $\sRV = \sPRV(\Rb^{\leq 0})$, and that any prime of $\sPRV$ remains prime in $\sPRV(\Rb^{\leq 0})$ (for instance, because $\sPRV(\Rb^{\leq 0})$ is a directed union of rings of polynomials over $\sPRV$, as in \prettyref{fact:KS-is-UFD}).
\end{proof}

It is now easy to lift the above result to primality in $\KRR$.

\begin{lem}
  \label{lem:b-irred-rvb-prvb-prime}
  For all $b \in \KRR$, if $\frac{\rv(b)}{\p(\rv(b))}$ is prime (in $\sPRV$) and $\p(b) = 1$, then $b$ is prime.
\end{lem}
\begin{proof}
  Let $b$ as in the hypothesis, and let $c, d \in \KRR$ be such that $b \mid cd$. We shall prove that $b \mid c$ or $b \mid d$ by induction on $\deg(cd)$. Let $B = \frac{\rv(b)}{q}$ where $q = \p(\rv(b))$. By assumption, $B$ is prime, so in particular, $B \mid \rv(c)$ or $B \mid \rv(d)$.

  Suppose that $B \mid \rv(c)$. Then $\rv(b) = q \cdot B \mid q \cdot \rv(c)$. Write $qc = be + f$ so that $\deg(f) < \deg(c)$. If $f = 0$, we are done, so assume otherwise. We have $b \mid fd$. By induction, $b \mid d$, in which case we are done, or $b \mid f$, in which case $b \mid qc$. In the latter case, let $g = \frac{qc}{b}$. By \prettyref{prop:pbc-pb-pc}, $\p(qc) = q\p(c) = \p(b)\p(g) = \p(g)$, thus in particular $\frac{g}{q} \in \KRR$ and $\frac{bg}{q} = c$, thus $b \mid c$, as desired.
\end{proof}

\begin{proof}[Proof of \prettyref{main:prime}]
  Let $b \in \KRR$ be a series of order type $\omega \hatplus k$ for some $k < \omega$ and not divisible by $t^x$ for any $x \in \Rb^{<0}$. By the assumption on the order type of $b$, $\rv(b)$ is weakly principal, so $\frac{\rv(b)}{\p(\rv(b))}$ is principal and of degree $1$, so it is prime by \prettyref{cor:PRV-1-prime-in-RV}. Moreover, $\p(\rv(b)) = t^y$ for some $y \in \Rb^{\leq 0}$. Since $\p(b)$ must divide $\p(\rv(b))$, we must have $\p(b) = t^x$ for some $x \geq y$, hence $x = 0$, so $\p(b) = 1$. Therefore, $b$ is prime by \prettyref{lem:b-irred-rvb-prvb-prime}.
\end{proof}

Once again, the arguments behind the proof of \prettyref{main:prime} can also be used to yield additional primes, although with more effort compared to the analogous work in \prettyref{sub:criterion-irred}. We will now show how to find new primes of degree $1$. We start by finding a few more primes in $\sRV$ of degree $1$, beyond the principal ones.

\begin{lem}
  \label{lem:b-irred-frac-prime}
  Let $\Rbf$ be an integral domain and $\Gbf$ be an ordered abelian group. Let $b = b_1t^{x_1} + \dots + b_nt^{x_n} \in \Rbf(\Gbf^{\leq 0})$  (with $x_1 < \dots < x_n$) be irreducible in $\Rbf(\Gbf^{\leq 0})$. If $b$ is irreducible in $\Frac(\Rbf)(\Gbf^{\leq 0})$, and $b_1$ is prime in $\Rbf$, then $b$ is prime in $\Rbf(\Gbf^{\leq 0})$.
\end{lem}
\begin{proof}
  Suppose that $b \mid cd$ for some $c, d \in \Rbf(\Gbf^{\leq 0})$. Since $b$ is irreducible in $\Frac(\Rbf)(\Gbf^{\leq 0})$, and the latter is a GCD domain by \prettyref{fact:KG-GCD-domain}, we may assume that $b \mid c$ or $b \mid d$ in the ring $\Frac(\Rbf)(\Gbf^{\leq 0})$. Without loss of generality, we may assume to be in the former case.

  Let
  \[ \varepsilon = \frac{b}{b_1}t^{-x_1} - 1 = \frac{b_2}{b_1}t^{x_2 - x_1} + \dots + \frac{b_n}{b_1}t^{x_n - x_1} \in \Rbf\left[\frac{1}{b_1}\right](\Gbf) \subseteq \Frac(\Rbf)(\Gbf), \]
  where $\Rbf\left[\frac{1}{b_1}\right]$ is the ring generated by $\Rbf$ and $\frac{1}{b_1}$, so that $b = b_1t^{x_1}(1 + \varepsilon)$ and $v(\varepsilon) = x_2 - x_1 > 0$, where $v$ is the canonical valuation of the Hahn field $\Frac(\Rbf)((\Gbf))$.

  Recall that since $v(\varepsilon) > 0$, the series $1 - \varepsilon + \varepsilon^2 - \dots$ in also in $\Frac(\Rbf)((\Gbf))$ and it is the multiplicative inverse of $(1 + \varepsilon)$ (see \prettyref{rem:KGG-geq0-GCD}). Therefore,
  \[ \frac{1}{b} = \frac{1}{b_1}t^{-x_1}(1 + \varepsilon)^{-1} = \frac{1}{b_1}t^{-x_1}\left(1 - \varepsilon + \varepsilon^2 - \ldots \right) \in \Rbf\left[\frac{1}{b_1}\right]((\Gbf)) \subseteq \Frac(\Rbf)((\Gbf)). \]
  Now multiply both sides by $c$. Since $b$ divides $c$, $\frac{c}{b}$ is in $\Frac(\Rbf)(\Gbf^{\leq 0})$, hence it is a finite sum of monomials.

  In turn, there is a maximum $m$ such that $\frac{c}{b} \in \frac{1}{b_1^m}\Rbf(\Gbf^{\leq 0})$, hence $b \mid b_1^mc$ in the ring $\Rbf(\Gbf^{\leq 0})$. Now write $be = b_1^mc$ with $e \in \Rbf(\Gbf^{\leq 0})$. Since $b_1$ is prime, and $b$ is irreducible, either $b = ub_1$ for some unit $u$, or $b_1^m$ must divide $e$ and thus $b$ divides $c$. In the former case, we have shown that $b$ is prime; in the latter, that $b \mid c$. Since the product $cd$ was arbitrary, that also implies that $b$ is prime.
\end{proof}

\begin{cor}
  \label{cor:B-sRV1-irred-frac-prime}
  If $B \in \RV$ has degree $1$ and is irreducible in both $\sRV$ and $\Frac(\sPRV)(\Rb^{\leq 0})$, then $B$ is prime in $\sRV$.
\end{cor}
\begin{proof}
  Let $B \in \RV$ as in the hypothesis. Then $B = B_1 \cdot t^{x_1} + \dots + B_n \cdot t^{x_n}$ for some $B_i \in \PRV$ of degree $1$ and $x_1 < \dots < x_n$. By \prettyref{cor:PRV-1-prime-in-RV}, $B_1$ is prime, so by \prettyref{lem:b-irred-frac-prime}, $B$ is prime in $\sRV$.
\end{proof}

\begin{cor}
  \label{cor:b-irred-frac-sRV1-prime}
  For all $b \in \KRR$ of degree $1$, if $\frac{\rv(b)}{\p(\rv(b))}$ is irreducible in both $\sRV$ and $\Frac(\sPRV)(\Rb^{\leq 0})$, and $\p(b) = 1$, then $b$ is prime.
\end{cor}
\begin{proof}
  Just apply \prettyref{cor:B-sRV1-irred-frac-prime} to $\rv(b)$ and use \prettyref{lem:b-irred-rvb-prvb-prime}.
\end{proof}

In turn, constructing series such that $\frac{\rv(b)}{\p(\rv(b))}$ satisfies the above assumptions is not hard, as we show in the following proposition.

\begin{prop}
  \label{prop:prime-degree-1}
  For every $\alpha \in \omega_1$, if $\deg(\alpha) = 1$, then there is a prime series $b \in \KRR$ such that $\ot(b) = \alpha$.
\end{prop}
\begin{proof}
  Let $\alpha$ as in the hypothesis, and write $\alpha = \omega \hatdot \ell \hatplus m$, where $\ell, m \in \Nb$ and $\ell > 0$. Pick some principal series $b_1, \dots, b_\ell \in \KRR$ of order type $\omega$ such that $\rv(b_1), \dots, \rv(b_\ell)$ are $\Kbf$-linearly independent (for instance, by taking $b_1, \dots, b_\ell$ with pairwise disjoint supports, thanks to \prettyref{lem:b1-bl-almost-disjoint}). Pick some real numbers $x_1 < x_2 < \dots < x_{\ell+m-1} < x_{\ell+m} = 0$, and if $\ell > 2$, chose them so that $(x_1 - x_\ell), \dots, (x_{\ell-1} - x_\ell)$ are linearly independent over $\Qb$. Let
  \[ b = b_1t^{x_1} + \dots + b_\ell t^{x_\ell} + t^{x_{\ell+1}} + \dots + t^{\ell+m}. \]

  By construction, $\ot(b) = \alpha$, and
  \[ \rv(b) = \rv(b_1) \cdot t^{x_1} + \dots + \rv(b_\ell) \cdot t^{x_\ell}. \]
  Since $\rv(b_1), \dots, \rv(b_\ell)$ are $\Kbf$-linearly independent, $\p(\rv(b)) = t^{x_\ell}$ by \prettyref{prop:RV-alpha-tensor}.

  Clearly, $\frac{\rv(b)}{t^{x_\ell}}$ is irreducible in $\sRV$: any factorisation must contain a factor of degree $0$, and that can only be an element of $\Kbf$ as we have divided out $\p(\rv(b))$. We claim that it is also irreducible in $\Frac(\sPRV)(\Rb^{\leq 0})$. This is clear if $\ell = 1$. It is also clear for $\ell > 2$, because the exponents $(x_1 - x_\ell), \dots, (x_{\ell-1} - x_\ell)$ are linearly independent over $\Qb$.

  For $\ell = 2$, we use a slightly more delicate argument. Let $B = \rv(b_2)$. Let $\sPRV_B$ denote the ring $\sPRV$ localised with respect to the complement of the ideal generated by $B$. Since $B$ is prime by \prettyref{cor:PRV-1-prime-in-RV}, $\sPRV_B$ is an integral domain. Moreover, for every $C \in \sPRV$ there is a maximum $n$ such that $B^n$ divides $C$ (e.g.\ because $n$ is bounded by the number of terms in the Cantor Normal Form of $\deg(C)$), hence $\sPRV_B$ is a discrete valuation ring. Since by construction $\rv(b_1)$ is not divisible by $B$, $\rv(b_1)$ is a unit in $\sPRV_B$, thus the polynomial $\rv(b_1)X^N + \rv(b_2)$ is irreducible over $\Frac(\sPRV)$ for every $N > 0$ by Eisenstein's criterion. It follows that $\frac{\rv(b)}{t^{x_2}} = \rv(b_1) \cdot t^{x_1 - x_2} + \rv(b_2)$ is irreducible in $\Frac(\sPRV)(\Rb^{\leq 0})$ by \prettyref{fact:KS-is-UFD}.

  On the other hand, $\p(b) = 1$ since it must divide $\p(\rv(b)) = t^{x_\ell}$ and $\sup(b) = 0$, so $b$ is prime by \prettyref{cor:b-irred-frac-sRV1-prime}.
\end{proof}

\begin{exa}
  \label{exa:prime-omega2+k+1}
  To see one example, consider the irreducible series of order type $\omega \hatplus \omega \hatplus k \hatplus 1$
  \[ b = \left(\sum_{n \in \Nb} t^{\frac{x}{3n+1}}\right)t^{3x} + \left(\sum_{n \in \Nb} t^{\frac{x}{3n+2}}\right)t^{2x} + t^{\frac{(k-1)x}{k}} + \dots + t^{\frac{x}{k}} + 1 \]
  from \prettyref{exa:irreducible-omega2+k+1}. Since $b_1 = \sum_{n \in \Nb} t^{\frac{x}{3n+1}}$, $b_2 = \sum_{n \in \Nb} t^{\frac{x}{3n+2}}$ are prime, the above argument shows that $b$ is in fact prime as well.
\end{exa}

\section{From real exponents to omnific integers}
\label{sec:omnific}

We now turn our attention to the ring $\oz$ of omnific integers, and more generally to rings of the form $\ZKGG_\kappa$ for $\Zbf$ a subring of $\Kbf$, possibly a proper class. Recall that we are working with $\Kbf$ a field of characteristic $0$, $\Gbf$ a divisible ordered abelian group, with both possibly proper classes, and $\kappa$ an uncountable cardinal or $\kappa = \on$; the subscript $\kappa$ means that we only take the series with support of cardinality strictly less than $\kappa$ (if $\kappa$ is a cardinal) or with support a set (if $\kappa = \on$). The ring $\oz$ is the special case in which $\Zbf = \Zb$, $\Kbf = \Rb$, $\Gbf$ is the additive group of $\no$ itself, and $\kappa = \on$. Such rings exhibit additional divisibility behaviours:

\begin{fact}[{\cite[Thm.\ 8.6]{Gon1986}}]
  \label{fact:gonshor-arch-classes}
  Let $b \in \ZKGG_\kappa$. If $\supp(b)$ intersects two distinct non-zero Archimedean classes of $\Gbf$, then $b$ is reducible.
\end{fact}

For instance, the omnific integer $\sum_{n \in \Nb} \omega^{\frac{1}{n+1}} + \omega^{\frac{1}{\omega}} + 1$ must be reducible because $\frac{1}{\omega} \prec 1$. Indeed,
\[ b = \sum_{n \in \Nb} \omega^{\frac{1}{n+1}} + \omega^{\frac{1}{\omega}} + 1 = \left(\sum_{n \in \Nb} \sum_{m \in \Nb} (-1)^m\omega^{\frac{1}{n+1} - \frac{m}{\omega}} + 1\right)\left(\omega^{\frac{1}{\omega}} + 1\right). \]
Even though Gonshor's proof of the above statement is written for $\no$, it applies to any $\ZKGG_\kappa$ without modifications. Thus, to find irreducible factors, we shall look at the series $b$ with support intersecting at most one non-zero Archimedean class.

We shall work with the notations of \prettyref{sub:archimedean}.

\subsection{Translating to real powers}

First, we set up some tools to reduce to rings of the form $\Lbf((\Rb^{\leq 0}))$, or more generally $\Lbf((H^{\leq 0}))$ with $H$ Archimedean, so as to be able to apply the results of \prettyref{sec:uniqueness}.

\begin{defn}
  Let $\sigma \in \Gbf_{/\asymp}$. Given $b = \sum_{x \in \Gbf} b_xt^x \in \ZKGG_\kappa$, let $T_\sigma$, $\tau_\sigma$ be the following functions from $\ZKGG_\kappa$ to itself:
  \[ T_{ \sigma}(b) \coloneqq \sum_{[x] \preceq \sigma} b_xt^x, \quad \tau_\sigma(b) \coloneqq \sum_{[x] \prec \sigma} b_xt^x. \]
\end{defn}

Note that obviously $T_\sigma \circ T_\sigma = T_\sigma$, $\tau_\sigma \circ \tau_\sigma = \tau_\sigma$, $T_\sigma \circ \tau_\sigma = \tau_\sigma \circ T_\sigma = \tau_\sigma$; if $\sigma \succ \sigma'$, then $T_\sigma \circ T_{\sigma'} = T_{\sigma'} = T_{\sigma'} \circ T_\sigma$, and likewise for $\tau_\sigma$, $\tau_{\sigma'}$.

\begin{rem}
  For any $b \in \ZKGG_\kappa$ and $\sigma \in \Gbf_{/\asymp}$, $T_\sigma(b)$ and $\tau_\sigma(b)$ are truncations of $b$ in the sense of \prettyref{def:truncation}.
\end{rem}

\begin{prop}
  \label{prop:T-t-sigma-hom}
  For all $\sigma \in \Gbf_{/\asymp}$, $T_{ \sigma}$ and $\tau_{ \sigma}$ are ring homomorphisms.
\end{prop}
\begin{proof}
  Let $b = \sum_{x \in \Gbf} b_xt^x$, $c = \sum_{x \in \Gbf} c_xt^x$ be series. It is clear from the definitions $T_{ \sigma}(b + c) = T_{ \sigma}(b) + T_{ \sigma}(c)$ and $\tau_{ \sigma}(b + c) = \tau_{ \sigma}(b) + \tau_{ \sigma}(c)$.

  As for multiplication, note that for every $x, y \in \Gbf^{\leq 0}$ we have $2\min\{x,y\} \leq x + y \leq \min\{x,y\} \leq 0$, thus $[x + y] = [\min\{x,y\}] = \max\{[x],[y]\}$, which immediately implies that $[x + y] \preceq \sigma$ if and only if $[x] \preceq \sigma$ and $[y] \preceq \sigma$. Likewise, $[x + y] \prec \sigma$ if and only if $[x] \prec \sigma$ and $[y] \prec \sigma$. Thus,
  \[ T_{ \sigma}(bc) = \sum_{[z] \preceq \sigma} t^z\left(\sum_{x+y = z}b_xc_y\right) = \sum_{[x] \preceq \sigma} b_xt^x \sum_{[x] \preceq \sigma} c_xt^x = T_{ \sigma}(b)T_{ \sigma}(c), \]
  and one can similarly show that $\tau_{ \sigma}(bc) = \tau_{ \sigma}(b)\tau_{ \sigma}(c)$.
\end{proof}

\begin{cor}
  \label{cor:deg-sup-sigma-mult}
  For all non-zero $\sigma \in \Gbf_{/\asymp}$, identify $H_\sigma$ with a subgroup of $\Rb$, and let $\deg_\sigma$ and $\sup_\sigma$ be the maps defined as follows for $b \in \ZKGG_\kappa$, with $\iota_\sigma$ from \prettyref{fact:isomorphism-K-H-sigma}:
  \[ \deg_\sigma(b) \coloneqq \deg(\iota_\sigma(T_\sigma(b))) \in \omega_1, \quad {\sup}_\sigma(b) \coloneqq \sup(\iota_\sigma(T_\sigma(b))) \in \Rb. \]
  Then $\deg_\sigma$ and $\sup_\sigma$ are multiplicative semi-valuations.
\end{cor}
\begin{proof}
  Immediate from \prettyref{prop:T-t-sigma-hom}, \prettyref{fact:isomorphism-K-H-sigma}\prettyref{item:fact-isom-K-H-image}, and respectively from \prettyref{main:degree}, \prettyref{prop:sup-valuation}.
\end{proof}

Note that the value of $\deg_\sigma$ does not depend on the embedding of $H_\sigma$ into $\Rb$.

The above observation yields a rather short and direct proof of the main theorem of \cite{Pit2002} (where the group $\Gbf$ is not assumed to be divisible).

\begin{cor}[\cite{Pit2002}]
  \label{cor:pitteloud}
  Let $\Hbf$ be any abelian ordered group (not necessarily divisible, and possibly a proper class). The ideal $\Jbf$ generated by the series $t^x$ for $x \in \Hbf^{<0}$ is prime in $\ZKHHH_\kappa$.
\end{cor}
\begin{proof}
  Suppose that $bc \in \Jbf$ for some $b, c \in \ZKHHH_\kappa$. Since $\Jbf$ is the directed union of the principal ideals generated by $t^y$ for $y \in \Hbf^{<0}$, there exists  $x \in \Hbf^{<0}$ such that $t^x$ divides $bc$.

  Let $\Gbf$ be the divisible hull of $\Hbf$, $\sigma = [x]$, and identify $H_\sigma$ with a subgroup of $\Rb$. We have ${\sup}_\sigma(bc) = {\sup}_\sigma(b) + {\sup}_\sigma(c) \leq {\sup}_\sigma(t^x) < 0$. Therefore, ${\sup}_\sigma(b) < 0$ or ${\sup}_\sigma(c) < 0$.

  Without loss of generality, assume ${\sup}_\sigma(b) < 0$. Then there is some $y \in \Hbf_{\preceq \sigma}^{<0}$ such that ${\sup}_\sigma(b) \leq {\sup}_\sigma(t^y)$. If $y$ is the maximum of $\Hbf^{<0}$, then $t^y$ divides $b$, and we are done. If not, there must be a $z \in \Hbf^{<0}$ such that $y < z < 0$ and $z - y \asymp y$, hence ${\sup}_\sigma(t^y) < {\sup}_\sigma(t^z)$, from which it follows that $t^z$ divides $b$, as desired.
\end{proof}

\subsection{Reducing divisors}
\label{sub:reduction}
By \prettyref{fact:gonshor-arch-classes}, if the support of some $b \in \ZKGG_\kappa$ intersects two non-zero Archimedean classes, then $b$ is reducible. The argument is easy to generalise using the maps $T_\sigma$ and $\tau_\sigma$, and it allows us to study divisibility in $\oz$ by considering the rings $\SLHH$. Recall from \prettyref{sub:archimedean} that for every $\sigma \in \Gbf_{/\asymp}$, we have fixed:
\begin{itemize}
  \item $H_\sigma \cong \Gbf_{\preceq \sigma}/\Gbf_{\prec \sigma}$ (an Archimedean group);
  \item $\Lbf_\sigma = \Kbf((\Gbf_{\prec \sigma}))_\kappa \subseteq \Kbf((\Gbf))_\kappa$ (a field);
  \item $\Sbf_\sigma = \Lbf_\sigma \cap \left(\ZKGG_\kappa\right) \subseteq \ZKGG_\kappa$ (a ring);
\end{itemize}
and a corresponding isomorphism $\iota_\sigma : \Zbf + \Kbf((\Gbf_{\preceq \sigma}^{<0}))_\kappa \to \SLHH$.

The results in this subsection are valid even when $\Gbf$ is not divisible.

\begin{prop}
  \label{prop:b-div-c-div-T}
  For every $b, c \in \ZKGG_\kappa$ with $b \neq 0$, $b$ divides $c$ if and only if $b$ divides $T_\sigma(c)$ for some (equivalently, all) $\sigma \succeq [v(b)]$.
\end{prop}
\begin{proof}
  Let $b,c$ be as in the hypothesis and let $\sigma = [v(b)]$. Write $c = c_0 + c_1$, where $c_0 = T_\sigma(c)$. By construction, $\supp(c_1) \succ \sigma$. We claim that $b$ divides $c_1$, namely that $\frac{c_1}{b}$ is in $\ZKGG_\kappa$, which is clearly equivalent to the conclusion.

  First, we note that $\frac{1}{b}$ is in the subfield $\Kbf((\Gbf_{\preceq \sigma}))_\kappa$, since $\supp(b) \subseteq \Gbf_{\preceq \sigma}$. Therefore, for every negative $x \succ \sigma$ and every exponent $y \in \supp(\frac{1}{b})$, we have $[x] \succ \sigma \succeq [y]$ and in particular $x + y < 0$. It follows at once that $\frac{c_1}{b} \in \Kbf((\Gbf^{<0}))_\kappa \subseteq \ZKGG_\kappa$, as desired.
\end{proof}

\begin{cor}
  \label{cor:T-t-div}
  For every non-zero $b \in \ZKGG_\kappa$ and $\sigma \in \Gbf_{/\asymp}$, each of $\tau_\sigma(b)$ and $T_\sigma(b)$, if non-zero, divides $b$.
\end{cor}

\begin{cor}
  \label{cor:lift-primal-easy}
  Let $b \in \ZKGG_\kappa$ and $\sigma = [v(b)]$. Then $\iota_\sigma(b)$ is primal in $\SLHH$ if and only if $b$ is primal in $\ZKGG_\kappa$.
\end{cor}
\begin{proof}
  For the sake of notation, let $\Rbf = \Zbf + \Kbf((\Gbf_{\preceq \sigma}^{<0}))_\kappa$, $\Sbf = \ZKGG_\kappa$. Note that $\Rbf$ has the property that if $cd \in \Rbf$ for some $c, d \in \Sbf$, then $c, d \in \Rbf$. In particular, if $c, d \in \Rbf$, then $\frac{c}{d} \in \Sbf$ if and only if it does so in $\frac{c}{d} \in \Rbf$. By \prettyref{fact:isomorphism-K-H-sigma}\prettyref{item:fact-isom-K-H-image}, $\iota_\sigma$ is an isomorphism between $\SLHH$ and $\Rbf$, so it suffices to prove that $b$ is primal in $\Sbf$ if and only if $b$ is primal in $\Rbf$.

  Suppose that $b$ is primal in $\Rbf$. Let $c, d \in \Sbf$ such that $\frac{cd}{b} \in \Sbf$. Then by \prettyref{prop:b-div-c-div-T}, $\frac{T_\sigma(c)T_\sigma(d)}{b} \in \Sbf$ too, and since $T_\sigma(cd)$ is in $\Rbf$, we must have $\frac{T_\sigma(c)T_\sigma(d)}{b} \in \Rbf$. Therefore, there are $b_1, b_2 \in \Rbf$ such that $b = b_1b_2$ and $\frac{T_\sigma(c)}{b_1}, \frac{T_\sigma(d)}{b_2} \in \Rbf$. In particular, $\frac{c}{b_1}, \frac{c}{b_2} \in \Sbf$, again by \prettyref{prop:b-div-c-div-T}. Since $c$ and $d$ were arbitrary, we have shown that $b$ is primal in $\Sbf$.

  Conversely, suppose that $b$ is primal in $\Sbf$. Let $c, d \in \Rbf$ such that $\frac{cd}{b}$ in $\Rbf$. Then there are $b_1, b_2 \in \Sbf$ such that $b = b_1b_2$ and $\frac{c}{b_1}, \frac{d}{b_2} \in \Sbf$. In turn, we must have $b_1, b_2 \in \Rbf$, hence also $\frac{c}{b_1}, \frac{d}{b_2} \in \Rbf$, showing that $b$ is primal in $\Rbf$.
\end{proof}

In particular, we can immediately reprove \prettyref{fact:gonshor-arch-classes}: if the support of $b \in \ZKGG_\kappa$ intersects two non-zero Archimedean classes, then $\tau_\sigma(b) \notin \Zbf$, where $\sigma = [v(b)]$, hence $\tau_\sigma(b)$ is a non-trivial divisor of $b$. It is convenient to remove this source of divisibility, as follows.

\begin{defn}
  \label{def:rho}
  Let $\sigma \in \Gbf_{/\asymp}$. Define $\rho_\sigma : \ZKGG_\kappa \to \ZKGG_\kappa$ as
  \[ \rho_\sigma(b) \coloneqq \begin{cases}
    \frac{T_{ \sigma}(b)}{\tau_{ \sigma}(b)} & \text{if } \tau_{ \sigma}(b) \neq 0, \\
    T_{ \sigma}(b) & \text{otherwise}.
  \end{cases} \]
\end{defn}

Note that for every $\sigma \in \Gbf_{/\asymp}$ and $b \in \ZKGG_\kappa$ such that $\rho_\sigma(b) \neq 0$, $\rho_\sigma(b)$ divides $b$. In particular, if $b \in \ZKGG_\kappa$ is irreducible, then $b = u\rho_\sigma(b)$ for some unit $u$ of $\Zbf$; in the case of $b \in \oz$ irreducible, we find $b = \pm \rho_\sigma(b)$. The converse does not hold in general: for instance in $\oz$, we have $\rho_{[0]}(n) = n$ for every $n \in \Zb$, and $\rho_{[1]}(\sum_{n \in \Nb} \omega^{\frac{1}{n}}) = \sum_{n \in \Nb} \omega^{\frac{1}{n}}$ is divisible by every $k \in \Zb$. Moreover, note that we always have $\rho_\sigma \circ \rho_\sigma = \rho_\sigma$.

\begin{prop}
  \label{prop:reduced-equiv-properties}
  Let $b \in \ZKGG_\kappa$ be non-zero. Then the following are equivalent:
  \begin{enumerate}
    \item\label{item:red-is-rho} $b = \rho_\sigma(c)$ for some $c \in \ZKGG_\kappa$ and $\sigma \in \Gbf_{/\asymp}$;
    \item\label{item:red-is-fixed-rho-some-sigma} $\rho_\sigma(b) = b$ for some $\sigma \in \Gbf_{/\asymp}$;
    \item\label{item:red-is-fixed-rho} $\rho_\sigma(b) = b$ for $\sigma = [v(b)]$;
    \item\label{item:red-is-tau-zero-one} $\tau_\sigma(b) \in \{0, 1\}$ for $\sigma = [v(b)]$;
    \item\label{item:red-supp} $\supp(b) \cap \supp(b-1)$ is contained in an Archimedean class.
  \end{enumerate}
\end{prop}
\begin{proof}
  The equivalence \prettyref{item:red-is-rho} $\Leftrightarrow$ \prettyref{item:red-is-fixed-rho-some-sigma} follows at once from $\rho_\sigma \circ \rho_\sigma = \rho_\sigma$. \prettyref{item:red-is-fixed-rho} $\Rightarrow$ \prettyref{item:red-is-fixed-rho-some-sigma} is trivial. For \prettyref{item:red-is-fixed-rho-some-sigma} $\Rightarrow$ \prettyref{item:red-is-fixed-rho}, pick a $\sigma$ such that $\rho_\sigma(b) = b$. If $\sigma = [v(b)]$, we are done; if $\sigma \succ [v(b)]$, then $T_\sigma(b) = \tau_\sigma(b)$, thus $b = \rho_\sigma(b) \in \{0,1\}$, and $b = 1$ satisfies \prettyref{item:red-is-fixed-rho}; if $\sigma \prec [v(b)]$, then $v(\rho_\sigma(b)) \prec v(b)$, a contradiction. We have obtained \prettyref{item:red-is-fixed-rho} in all cases. For \prettyref{item:red-is-fixed-rho} $\Leftrightarrow$ \prettyref{item:red-is-tau-zero-one}, let $\sigma = [v(b)]$, and note that $T_\sigma(\rho_\sigma(b)) = \rho_\sigma(b)$. It follows at once that $\rho_\sigma(b) = b$ if and only if $\tau_\sigma(b) = 1$ or $\tau_\sigma(b) = 0$.

  Now note that for $\sigma = [v(b)]$, $\tau_\sigma(b) = 0$ holds if and only if $b \in \Kbf((\sigma^{<0}))_\kappa$, and therefore $\tau_\sigma(b) = 1$ if and only if $b - 1 \in \Kbf((\sigma^{<0}))_\kappa$. In particular, if $\tau_\sigma(b) \in \{0,1\}$, then $\supp(b) \cap \supp(b - 1) \subseteq \sigma$, proving \prettyref{item:red-is-tau-zero-one} $\Rightarrow$ \prettyref{item:red-supp}. Conversely, suppose that $\supp(b) \cap \supp(b-1) \subseteq \sigma$ for some Archimedean class $\sigma$. Since $\supp(b)$ and $\supp(b - 1)$ can differ by at most the element $0$, we have that at least one of $\supp(b)$ or $\supp(b-1)$ is entirely contained in some class $\sigma$. If $b \neq 1$, then such class is unique and equal to $[v(b)]$, and we find that $\tau_\sigma(b) \in \{0,1\}$ by the previous argument; if $b = 1$, then $\tau_\sigma(b) = 1 \in \{0,1\}$ too. This shows \prettyref{item:red-supp} $\Rightarrow$ \prettyref{item:red-is-tau-zero-one}.
\end{proof}

In particular, $\rho_\sigma$ takes values in $\Zbf + \Kbf((\sigma^{<0}))_\kappa$, or more precisely in $\Kbf((\sigma^{<0}))_\kappa \cup (1 + \Kbf((\sigma^{<0}))_\kappa)$.

\begin{defn}
  \label{def:reduced}
  A non-zero $b \in \ZKGG_\kappa$ is \textbf{reduced} if $\supp(b) \cap \supp(b-1)$ is fully contained in one Archimedean class.
\end{defn}

\begin{exa}
  \label{exa:reduced}
  In the ring $\oz$, checking whether an omnific integer is reduced is a fairly explicit computation. Let $b = \sum_{i<\alpha} b_i\omega^{x_i}$ be an omnific integer. Each $x_i$ (an element of $-\supp(b)$) is itself a surreal number, which we then write as $x_i = \sum_{j<\beta_i} c_{i,j}\omega^{y_{i,j}}$. Then $b$ is reduced if and only if for every $i < \alpha$ we have either $y_{i,0} = y_{0,0}$, which means that $x_i$ and $x_0$ are in the same non-zero Archimedean class, or $x_i = 0$, and if the latter happens for some (necessarily unique) $i > 0$, then $b_i = 1$. Note that such condition is void when $x_0 = 0$, in which case $b$ is always reduced since $\supp(b) \subseteq \{0\}$.

  One may then verify by direct inspection that $\omega + 1$, $\sum_{n \in \Nb} \omega^{\frac{1}{n+1}}$, $\sum_{n \in \Nb} \omega^{1 + \frac{1}{\omega(n+1)}}$, or any element of $\Zb$ are reduced omnific integers, whereas $\sum_{n \in \Nb} \omega^{\frac{1}{n+1}} + \omega^{\frac{1}{\omega}} + 1$, $\omega + 2$ are not reduced.
\end{exa}

\begin{prop}
  \label{prop:reduced-b-div-c}
  Let $b \in \ZKGG_\kappa$ be reduced and $\sigma = [v(b)]$. For all $c \in \ZKGG_\kappa$, $b$ divides $c$ if and only if $b$ divides $\rho_\sigma(c)$.
\end{prop}
\begin{proof}
  Let $c \in \ZKGG_\kappa$. By \prettyref{prop:b-div-c-div-T}, $b$ divides $c$ if and only if $b$ divides $T_{ \sigma}(c)$. We may then directly assume that $c = T_\sigma(c)$.

  Suppose that $b$ divides $c$ and write $c = be$ for some $e \in \ZKGG_\kappa$. Note that $T_\sigma(e) = e$, because $[v(e)] \preceq [v(c)] = \sigma$. If $\tau_{ \sigma}(c) = 0$, then $b$ divides $\rho_\sigma(c) = T_{ \sigma}(c) = c$, as desired. If $\tau_{ \sigma}(c) \neq 0$, then $\tau_{ \sigma}(be) = \tau_\sigma(b)\tau_\sigma(e) \neq 0$, hence we must have $\tau_{ \sigma}(b) = 1$ by \prettyref{prop:reduced-equiv-properties}\prettyref{item:red-is-tau-zero-one}. It follows that $\tau_{ \sigma}(c) = \tau_{ \sigma}(e)$, hence $b\rho_\sigma(e) = \rho_\sigma(c)$, which says that $b$ divides $\rho_\sigma(c)$, as desired.

  For the converse, assume that $b$ divides $\rho_\sigma(c)$. When $\rho_\sigma(c) \neq 0$, $\rho_\sigma(c)$ divides $c$, so $b$ divides $c$. When $\rho_\sigma(c) = 0$, then in fact $c = 0$, and so $b$ is a divisor of $c$.
\end{proof}

\begin{prop}
  \label{prop:lift-primal-easy-reduced}
  Let $b \in 1 + \Kbf((\Gbf^{<0}))_\kappa$ be reduced and $\sigma = [v(b)]$. Then $b$ is primal in $\ZKGG_\kappa$ if and only if $\iota_\sigma(b)$ is primal in $\LHH$.
\end{prop}
\begin{proof}
  By \prettyref{cor:lift-primal-easy}, $b$ is primal in $\ZKGG_\kappa$ if and only if $\iota_\sigma(b)$ is primal in $\SLHH$. Since $b$ is reduced and in $1 + \Kbf((\Gbf^{<0}))_\kappa$, we have in particular that $\tau_\sigma(b) = 1$, thus $\iota_\sigma(b) \in 1 + \Lbf_\sigma((H_\sigma^{<0}))$. One can then verify that a series in $1 + \Lbf_\sigma((H_\sigma^{<0}))$ is primal in $\SLHH$ if and only if it is primal in $\LHH$ by direct computation.
\end{proof}

If we drop the assumption that $b$ is in $1 + \Kbf((\Gbf^{<0}))_\kappa$, whether $b$ is primal or not depends also on the properties of $\Gbf$, $\kappa$, and $\Zbf$. We analyse this phenomenon in \prettyref{sec:on-primal-series}.

\begin{proof}[Proof of \prettyref{main:Gonshor}]
  By \prettyref{prop:lift-primal-easy-reduced} and \prettyref{cor:KR-primal-KRR}, $\omega^{\sqrt{2}} + \omega + 1$ is primal in $\oz$. Since it is also irreducible by \cite{Ber2000}, it is prime.
\end{proof}

\subsection{Pseudo-irreducibles and pseudo-polynomials}
\label{sub:pseudo-stuff}

We can now describe the factorisation theory of reduced series in $\ZKGG_\kappa$, based on \prettyref{thm:KHH-fact-unique} and the fact that the map $\iota_\sigma$ of \prettyref{fact:isomorphism-K-H-sigma} identifies each $b \in \ZKGG_\kappa \setminus \Zbf$ with an element of $\SLHH$, where $\sigma = [v(b)]$. Recall that reduced series are then mapped into $\Zbf + \Lbf_\sigma((H_\sigma^{<0}))$.

We have an additional divisibility phenomenon to control: there are series which are irreducible in $\LHH$, but are reducible in $\Zb + \LHH$, such as the series $\sum_{n \in \Nb} t^{\frac{-x}{n+1}}$ for any $x \in H_\sigma^{<0}$. We handle this with a suitable weakening of the notion of irreducibility.

\begin{defn}
  \label{def:pseudo-stuff}
  Given a reduced $b \in \ZKGG_\kappa \setminus \Zbf$ with $\sigma = [v(b)]$:
  \begin{itemize}
    \item the \textbf{pseudo-support} of $b$, denoted by $\osupp(b)$, is the support of $\iota_\sigma(b)$ (a subset of $H_\sigma \subseteq \Rb$);
    \item $b$ is a \textbf{pseudo-polynomial} if its pseudo-support is finite;
    \item $b$ is a \textbf{pseudo-monomial} if its pseudo-support consists of exactly one point;
    \item $b$ is \textbf{pseudo-irreducible} if $\iota_\sigma(b)$ is irreducible in $\LHH$;
    \item $b$ is \textbf{almost irreducible} if $\iota_\sigma(b)$ is almost irreducible in $\LHH$ (in the sense of \prettyref{sub:non-complete}).
  \end{itemize}
\end{defn}

\begin{exa}
  In the case of $\oz$, we can make the computation of pseudo-supports explicit. Just as in \prettyref{exa:reduced}, let $b = \sum_{i < \alpha} b_i\omega^{x_i}$ be an omnific integer, and write each exponent as $x_i = \sum_{j < \beta_i} c_{i,j}\omega^{y_{i,j}}$. We assume that $b$ is reduced and not in $\Zb$, and note that $\sigma = [v(b)] = [x_0] = [\omega^{y_{0,0}}]$. Recall that $H_\sigma \cong \Rb$, and we choose the isomorphism so that $\omega^{y_{0,0}}$ is mapped to $1$.

  Then the pseudo-support of $b$ is the set of coefficients $\{-c_{i,0} : y_{i,0} = y_{0,0}\}$, plus the element $0$ if there is some (unique) $i$ such that $x_i = 0$. We may then verify by direct inspection that $\omega^2 + \omega + 1$, $\sum_{n \in \Nb}\omega^{\omega + \frac{1}{n+1}} + \omega^\omega + 1$ are pseudo-polynomials of pseudo-supports respectively $\{-2, -1, 0\}$, $\{-1, 0\}$; similarly, $\sum_{n \in \Nb}\omega^{3\omega + \frac{1}{n+1}}$ is a pseudo-monomial since its pseudo-support is $\{-3\}$.

  Likewise, the pseudo-support of $b = \sum_{n \in \Nb}\omega^{\frac{1}{n+1}}$ is $\{-1, -\frac{1}{2}, -\frac{1}{3}, \dots\}$, thus of order type $\omega$ with supremum $0$ (as a subset of $\Rb$); therefore, $b$ is pseudo-irreducible by \cite[Thm.\ 10.5]{Ber2000}, and it is not a pseudo-polynomial.
\end{exa}

We claim that the above types of series all have a good factorisation theory. Moreover, we claim that they are all that we need. First, pseudo-polynomials can be analysed using Ritt's theorem.

\begin{prop}[Ritt's theorem for pseudo-polynomials]
  \label{prop:ritt-pseudo-poly}
  Let $p$ be a pseudo-polynomial in $\ZKGG_\kappa$. There are an integer $n \in \Nb$ and pseudo-polynomials $m, c_1, \dots, c_n \in \ZKGG_\kappa$ such that $p = mc_1 \cdots c_n$, where $m$ is a pseudo-monomial, each $c_i$ satisfies exactly one of the following:
  \begin{enumerate}
    \item\label{item:ritt-pp-dim-2} $c_i$ is irreducible and $\lspan{\osupp(c_i)}$ has dimension $\geq 2$;
    \item\label{item:ritt-pp-dim-1} $\lspan{\osupp(c_i)}$ has dimension $1$ and $0 \in \supp(c_i)$;
  \end{enumerate}
  and the pseudo-supports of the $c_i$'s satisfying \prettyref{item:ritt-pp-dim-1} are pairwise $\Qb$-linearly independent. Moreover, the factors $c_i$ and $m$ as above are unique up to reordering.
\end{prop}
\begin{proof}
  Let $\sigma = [v(p)]$. The conclusion is trivial if $\sigma = [0]$, so we may assume $\sigma \neq [0]$, namely $p \notin \Zbf$. Note that there are unique $k \in \Lbf_\sigma$, $x \in H_\sigma^{\leq 0}$, and a pseudo-polynomial $p'$ such that $\iota_{v(p)}(p) = kt^xp'$ and $0 \in \supp(p')$, $0 \notin \supp(p' - 1)$. Write $\iota_{v(p)}(p') = c_1' \cdots c_n'$ according to \prettyref{prop:ritt}, so that each $c_i'$ is irreducible, or with support generating a space of dimension 1,  with the latter spaces being pairwise linearly independent. After multiplying by the appropriate elements of $\Lbf_\sigma$, we may assume that $0 \notin \supp(c_i' - 1)$ for every $i$, which determines such factors uniquely up to reordering. We find the desired factorisation on letting $m = \iota_\sigma^{-1}(m)$ and $c_i = \iota_\sigma^{-1}(c_i')$.
\end{proof}

Moreover, we have an immediate generalisation of \prettyref{main:Gonshor}.
\begin{prop}
  \label{prop:pseudo-poly-primal-easy}
  Every pseudo-polynomial in $1 + \Kbf((\Gbf^{<0}))_\kappa$ is primal in the ring $\ZKGG_\kappa$.
\end{prop}
\begin{proof}
  Immediate by \prettyref{cor:KH-primal-KHH} and \prettyref{prop:lift-primal-easy-reduced}.
\end{proof}

The factorisation theory for pseudo-monomials, pseudo-irreducibles, and almost irreducibles is even simpler.
\begin{prop}
  \label{prop:fact-pseudo-monomial}
  Every pseudo-monomial is almost irreducible and not pseudo-irreducible.
\end{prop}
\begin{proof}
  Let $m \in \ZKGG_\kappa$ be a pseudo-monomial and $\sigma = [v(m)]$. By definition, $\iota_\sigma(m)$ is a monomial $kt^x$ for some $k \in \Lbf_\sigma$ and $x \in H_\sigma^{<0}$. It now suffices to note that every monomial of $\LHH$ is almost irreducible, because its only divisors are other monomials, and not pseudo-irreducible, for instance because $t^{\frac{x}{2}}$ is a proper divisor of $kt^x$.
\end{proof}

\begin{prop}
  \label{prop:fact-almost-irreducible}
  Let $b \in \ZKGG_\kappa$ be almost irreducible. Then:
  \begin{enumerate}
    \item\label{item:almost-irred-div} for every divisor $c$ of $b$, one of $c$, $\frac{b}{c}$ is a pseudo-monomial, or $v(c) \prec v(b)$ or $v\left(\frac{b}{c}\right) \prec c$;
    \item\label{item:pseudo-irred-div} if $b$ is pseudo-irreducible, then for every divisor $c$ of $b$ we have $v(c) \prec v(b)$ and $\frac{b}{c}$ pseudo-irreducible, or  $v(\frac{b}{c}) \prec v(b)$ and $c$ pseudo-irreducible;
    \item\label{item:all-small-div} if $0 \notin \supp(b)$, then every non-zero $c \in \ZKGG_\kappa$ with $v(c) \prec v(b)$ divides $b$;
    \item\label{item:pseudo-irred-osupp} $b$ is pseudo-irreducible if and only if $\sup\osupp(b) = 0$;
    \item\label{item:irred-osupp} $b$ is irreducible if and only if $0 \in \supp(b)$, or $\sup\osupp(b) = 0$, $\Zbf$ is a field, and $\Gbf_{\prec [v(b)]} = \{0\}$.
  \end{enumerate}
\end{prop}
\begin{proof}
  Let $\sigma \in \Gbf_{/\asymp}$ be the Archimedean class of $v(b)$.

  \prettyref{item:almost-irred-div} Let $c$ be a divisor of $b$. Recall that by assumption of almost irreducibility, one of $\iota_\sigma(c)$, $\iota_\sigma(\frac{b}{c})$ is a monomial, thus of the form $kt^x$ for some $k \in \Lbf_\sigma$ and $x \in H_\sigma$. When $x < 0$, then one of $b$, $\frac{b}{c}$ is a pseudo-monomial; if $x = 0$, then one of $b$, $c$ is in $\Lbf_\sigma$, thus of canonical valuation $\prec v(b)$.

  \prettyref{item:pseudo-irred-div} Let $c$ be a divisor of $b$. When $b$ is pseudo-irreducible, then one of $\iota_\sigma(c)$, $\iota_\sigma(\frac{b}{c})$ is in $\Lbf_\sigma$ while the other is irreducible in $\LHH$, hence either $\frac{b}{c}$ is pseudo-irreducible and $v(c) \prec v(b)$, or $c$ is pseudo-irreducible and $v(\frac{b}{c}) \prec v(b)$.

  \prettyref{item:all-small-div} If $0 \notin \supp(b)$, then $\tau_\sigma(b) = 0$, hence $T_{\sigma'}(b) = 0$ for every $\sigma' \prec \sigma$. By \prettyref{prop:b-div-c-div-T}, $b$ is divisible by every $c \in \ZKGG_\kappa$ such that $v(c) \prec v(b)$.

  \prettyref{item:pseudo-irred-osupp} \prettyref{rem:almost-irred} shows that $\iota_\sigma(b)$ is irreducible in $\Lbf_\sigma((H_\sigma^{\leq 0}))$, or equivalently $b$ is pseudo-irreducible, if and only if $\sup\osupp(b) = 0$.

  \prettyref{item:irred-osupp} If $0 \in \supp(b)$, $\iota_\sigma(b)$ is irreducible in $\LHH$ by \prettyref{item:pseudo-irred-osupp}. Moreover, $\tau_\sigma(b) = 1$, so $\iota_\sigma(b) \in 1 + \Lbf_\sigma((H_\sigma^{<0}))$, and it follows that $\iota_\sigma(b)$ is irreducible in $\SLHH$ as well. This implies that $b$ is irreducible. If instead we know that $0 = \sup\osupp(b)$, $\Zbf$ is a field, and $\Gbf_{\prec \sigma} = \{0\}$, then $\Sbf_\sigma = \Zbf$ is a field, and one deduces similarly that $b$ is irreducible, proving one direction of~\prettyref{item:irred-osupp}. Conversely, if $b$ is irreducible, then $\sup\osupp(b) = 0$ by \prettyref{item:pseudo-irred-osupp}; if moreover $0 \notin \supp(b)$, then $b$ is divisible by every non-zero $k \in \Sbf_\sigma$ by \prettyref{item:all-small-div}, thus every non-zero $k \in \Sbf_\sigma$ is a unit, or in other words, $\Sbf_\sigma$ is a field, which immediately implies that $\Zbf$ is a field and that $\Gbf_{\prec \sigma} = \{0\}$, completing the argument.
\end{proof}

We conclude this subsection by spelling out the immediate generalisation of \prettyref{thm:KHH-fact-unique} for reduced series.
\begin{prop}
  \label{prop:reduced-fact-unique}
  Let $b \in \ZKGG_\kappa \setminus \Zbf$ be reduced. Then there are $n \in \Nb$ and reduced $p, m, c_1, \dots, c_n \in \ZKGG_\kappa$ such that $b = pmc_1 \cdots c_n$, where
  \begin{enumerate}
    \item\label{item:red-pseudo-poly} $p$ is a pseudo-polynomial with $0 \in \supp(p)$ and $v(p) \asymp v(b)$, or $p = 1$;
    \item\label{item:red-pseudo-mono} $m$ is a pseudo-monomial with $v(m) \asymp v(b)$, or $m = 1$;
    \item\label{item:red-almost-irred} each $c_i$ is almost irreducible with infinite pseudo-support and $v(c_i) \asymp v(b)$.
  \end{enumerate}
  Among these factorisations, $p$ is unique. Moreover:
  \begin{enumerate}[label=(\alph*),ref=\alph*]
      \item\label{item:red-sup} if $\sup_\sigma(b) \in H_\sigma$, then we may take $c_1, \dots, c_n$ pseudo-irreducible, in which case $m$ is unique up to multiplication by series $d \in \Kbf((\Gbf))_\kappa$ such that $\supp(d) \prec v(b)$;
      \item\label{item:red-max} if $\osupp(b)$ has a maximum, then we may take $c_1, \dots, c_n$ irreducible, in which case $m$ is unique.
  \end{enumerate}
\end{prop}
\begin{proof}
  Let $\sigma = [v(b)]$. By \prettyref{thm:KHH-fact-unique} applied to $\iota_\sigma(b)$, there is a factorisation $\iota_\sigma(b) = p' t^y c_1'\cdots c_n'$ where $p' = \p(\iota_\sigma(\rho_\sigma(b)))_{H_\sigma} \in 1  + \Lbf_\sigma(H_\sigma^{<0})$, each $c_i'$ is almost irreducible in $\LHH$ with infinite support, and $y \in H_\sigma^{\leq 0}$.

  Since $b$ is reduced, after multiplying $t^y$ by some $k \in \Kbf$, we may assume that $0 \notin \supp(c_i'-1)$ whenever $0 \in \supp(c_i')$. Hence their preimages $p = \iota_\sigma^{-1}(p')$, $c_i = \iota_\sigma^{-1}(c_i')$ are reduced. Moreover, $p$ is a pseudo-polynomial with $v(p) \asymp v(b)$, or $p = 1$, and each $c_i$ is almost irreducible with infinite pseudo-support. We have found the factors of the form \prettyref{item:red-pseudo-poly} and \prettyref{item:red-almost-irred}.

  Now let $m = \iota_\sigma^{-1}(kt^y)$. If $y < 0$, then $m$ is a pseudo-monomial, with $v(m) \asymp v(b)$. If $y = 0$, and so $m = k$, we distinguish two cases: if $0 \notin \supp(b)$, then $0 \notin \osupp(c_i)$ for some $i$, hence after replacing $c_i$ with $mc_i$, we may assume that $m = 1$; otherwise, $0 \in \supp(b)$, hence $\tau_\sigma(b) = \tau_\sigma(p) = \tau_\sigma(c_i) = 1$ for all $i$, and so in particular $k = \tau_\sigma(m) = 1$, as required in \prettyref{item:red-pseudo-mono}. Note that $m$ is reduced in either case.

  If $\sup_\sigma(b) \in H_\sigma$, then by \prettyref{thm:KHH-fact-unique}, we may choose $c_i'$ irreducible, in which case $c_1, \dots, c_n$ are pseudo-irreducible and $y = \sup_\sigma(b)$ is uniquely determined. Conversely, any such factorisation must satisfy $\sup(\iota_\sigma(m)) = y$. In particular, $m$ is uniquely determined up to multiplication by series $d \in \Kbf((\Gbf))_\kappa$ such that $\supp(d) \prec v(b)$, proving \prettyref{item:red-sup}.

  When moreover $\osupp(b)$ has a maximum, we may choose again $c_i'$ irreducible, and we now have $0 \in \supp(c_i')$, thus $c_1,\dots,c_n$ are irreducible by \prettyref{prop:fact-almost-irreducible}. Conversely, if we are given one such factorisation, we must have $\iota_\sigma(m) = kt^y$, where $y = \max\osupp(b)$ and $k$ is the coefficient of $t^0$ in $\frac{\iota_\sigma(b)}{t^y}$, thus $m$ is uniquely determined, proving \prettyref{item:red-max}.
\end{proof}

\subsection{Infinite product}

To conclude our analysis of $\ZKGG_\kappa$, we want to go from reduced series to arbitrary series. We address this by writing every $b \in \ZKGG_\kappa$ as an infinite product of reduced series.

We first present the special case of finite products, which require no additional tools. Suppose that $\supp(b)$ intersect only finitely many Archimedean classes (for instance, because $\supp(b)$ itself is finite). Let $\sigma_1 \succ \dots \succ \sigma_n$ be an enumeration of the classes intersected by $\supp(b)$. We have
\[ b = T_{\sigma_1}(b) = \frac{T_{\sigma_1}(b)}{\tau_{\sigma_1}(b)} \tau_{\sigma_1(b)} = \rho_{\sigma_1}(b) T_{\sigma_2}(b) = \dots = \rho_{\sigma_1}(b) \cdot \dots \cdot \rho_{\sigma_n}(b). \]
Thus, the factorisation theory of such a $b$ is completely controlled by its associated reduced series. For instance, we have the following.
\begin{cor}\label{cor:finite-supp-primal-easy}
  Every $b \in 1 + \Kbf((\Gbf^{<0}))_\kappa$ with finite support is primal in $\ZKGG_\kappa$.
\end{cor}
\begin{proof}
  Take $b$ as in the assumptions. For every $\sigma$ intersected by $\supp(b)$, the corresponding series $\rho_\sigma(b)$ is a pseudo-polynomial with $0$ in its support, and is thus primal in $\ZKGG_\kappa$ by \prettyref{prop:pseudo-poly-primal-easy}, or it is $1$. Since $b$ is a finite product of such primal series, it is primal.
\end{proof}

We now generalise the above finite product to the case of $\supp(b)$ intersecting infinitely many Archimedean classes.

\begin{defn}
  \label{def:Red}
  Given $b \in \ZKGG_\kappa$, let
  \[ \Red(b) \coloneqq (\rho_\sigma(b) : \sigma \in \Gbf_{/\asymp} \text{ such that } \supp(b) \cap \sigma \neq \varnothing). \]
\end{defn}
\begin{prop}
  \label{prop:Rho-uniqueness}
  For all $b, c \in \ZKGG_\kappa$, $\Red(b) = \Red(c)$ if and only if:
  \begin{enumerate}
    \item\label{item:Rho-min} there is a minimum $\sigma$ intersecting $\supp(b)$ and $b = c$, or
    \item\label{item:Rho-no-min} there is no such minimum and $\supp(\frac{b}{c}) \prec \sigma$ for all $\sigma$'s intersecting $\supp(b)$.
  \end{enumerate}
\end{prop}
\begin{proof}
  Let $b, c \in \ZKGG_\kappa$ such that $\Red(b) = \Red(c)$. In particular, $\supp(b)$ and $\supp(c)$ intersect the same Archimedean classes. Furthermore, the set of classes intersected by $\supp(b)$ is reverse well ordered, so we may enumerate them as $\sigma_0 \succ \sigma_1 \succ \dots \succ \sigma_\beta \succ \dots$ for $\beta < \alpha$, where $\alpha$ is some ordinal. Let $d = \frac{b}{c} \in \Kbf((\Gbf))_\kappa$, so that $b = cd$. We shall prove by induction that $\supp(d) \prec \sigma_\beta$ for all $\beta < \alpha$, or $b = c$. By assumption, $\rho_{\beta}(b) = \rho_{\beta}(c)$.

  Suppose that $\supp(d) \prec \sigma_\gamma$ for every $\gamma < \beta$. Let $b_1 = b - T_{\sigma_\beta}(b)$, $c_1 = c - T_{\sigma_\beta}(c)$, so that $\supp(b_1), \supp(c_1) \succ \sigma_\beta$. The Archimedean classes intersected by $\supp(b_1)$ are then exactly the $\sigma_\gamma$'s for $\gamma < \beta$, and the same holds for $\supp(c_1)$. Thus by inductive hypothesis, $\supp(c_1d)$ intersects the same Archimedean classes as $\supp(c_1)$, we have $\supp(d) \prec \supp(b_1), \supp(c_1)$, and moreover $\supp(c_1d) \succ \supp(T_{\sigma_\beta}(c)d)$. It follows that $b_1 = c_1d$ and $T_{\sigma_\beta}(b) = T_{\sigma_\beta}(c)d$, hence $\supp(d) \preceq \sigma_\beta$.

  We distinguish two cases. If $\tau_{\sigma_\beta}(b) = 0$, then $\sigma_\beta$ is the minimum Archimedean class intersected by $\supp(b)$ (equivalently, by $\supp(c)$), thus $\tau_{\sigma_\beta}(c) = 0$, hence
  \[ d = \frac{T_{\sigma_\beta}(b)}{T_{\sigma_\beta}(c)} = \frac{\rho_{\sigma_\beta}(b)}{\rho_{\sigma_\beta}(c)} = 1. \]
  Therefore, $d = 1$, or in other words, $b = c$. If $\tau_{\sigma_\beta}(b) \neq 0$, then $\sigma_\beta$ is not the minimum Archimedean class intersected by $\supp(b)$, thus $\tau_{\sigma_\beta}(c) \neq 0$, and
  \[ d = \frac{T_{\sigma_\beta}(b)}{T_{\sigma_\beta}(c)} = \frac{\rho_{\sigma_\beta}(b)}{\rho_{\sigma_\beta}(c)} \frac{\tau_{\sigma_\beta}(b)}{\tau_{\sigma_\beta}(c)} = \frac{\tau_{\sigma_\beta}(b)}{\tau_{\sigma_\beta}(c)}. \]
  Therefore, $\supp(d) \prec \sigma_\beta$, completing the induction. Thus, we either find a minimum $\beta$ such that $\sigma_\beta$ intersects $\supp(b)$ and $b = c$, reaching conclusion~\prettyref{item:Rho-min}, or the induction runs over all the classes $\sigma_\beta$ intersecting $\supp(b)$, reaching conclusion~\prettyref{item:Rho-no-min}.

  For the converse, in case \prettyref{item:Rho-min} we have $b = c$, so the conclusion is trivial. Suppose we are in case \prettyref{item:Rho-no-min}, so take some $b \in \ZKGG_\kappa$ such that there is no minimum $\sigma$ intersecting $\supp(b)$. Let $c \in \ZKGG_\kappa$ be such that $d = \frac{b}{c} \in \Kbf((\Gbf))_\kappa$ satisfies $\supp(d) \prec \sigma$ for all $\sigma$'s intersecting $\supp(b)$. Note that $\supp(c) = \supp(bd)$ intersects the same Archimedean classes as $\supp(b)$.

  Now fix a $\sigma$ intersecting $\supp(b)$. Since $\sigma$ is not minimal, we must have $\tau_\sigma(b) \neq 0$. Since $\supp(d) \prec \sigma$, one can easily verify that $T_\sigma(c) = T_\sigma(b)d$, $\tau_\sigma(c) = \tau_\sigma(b)d$. It follows that
  \[ \rho_\sigma(c) = \frac{T_\sigma(c)}{\tau_\sigma(c)} = \frac{T_\sigma(b)d}{\tau_\sigma(b)d} = \frac{T_\sigma(b)}{\tau_\sigma(b)} = \rho_\sigma(b). \]

  Therefore, $\Red(b) = \Red(c)$.
\end{proof}

\begin{defn}
  \label{def:inverse-R}
  Let $(b_\sigma : \sigma \in J)$ be some family of non-zero reduced series, where $J$ is a non-empty subset of $\Gbf_{/\asymp}$. We define $\iprod_{\sigma \in J} b_\sigma$ to be any arbitrarily chosen $b \in \ZKGG_\kappa$, if one exists, such that $\Red(b) = (b_\sigma : \sigma \in J)$.
\end{defn}

In other words, $\iprod$ is a section of $\Red$ over its image. The equality $\iprod_{\sigma \in J} b_\sigma = b$ means that $J$ is the set of the $\sigma$'s that intersect the support of $b$, and that $b_\sigma = \rho_\sigma(b)$ for all $\sigma \in J$. If such $b$ exists, and $J$ is finite, we have that
\[ \iprod_{\sigma \in J} b_\sigma = b = \prod_{\sigma \in J} \rho_{\sigma}(b) = \prod_{\sigma \in J} b_\sigma, \]
thus $\iprod$ and $\prod$ agree on the families where they are both defined.

Conversely, for every non-zero $b \in \ZKGG_\kappa$, if we let
\[ J = \{ \sigma \in \Gbf_{/\asymp} : \supp(b) \cap \sigma \neq \varnothing \}, \]
then $\iprod_{\sigma \in J} \rho_\sigma(b)$ is defined, and we may thus write $b = d \iprod_{\sigma \in J} \rho_\sigma(b)$ for a unique $d \in \Kbf((\Gbf))_\kappa$. \prettyref{prop:Rho-uniqueness} implies that $\supp(d) \prec \supp(b)$, and if there is a minimum $\sigma$ intersecting $\supp(b)$, then we must have $d = 1$. In particular, if $0 \in \supp(b)$, then $\sigma = [0]$ is such minimum, and so $d = 1$.

\begin{rem}
  \label{rem:canonical-infinite-product}
  \prettyref{prop:Rho-uniqueness} shows that the choice of value for $b = \iprod_{\sigma \in J} b_\sigma$ is not unique when $J$ has no minimum: for instance, in this case $b$ and $-b$ are both valid choices, as we get $\Red(b) = \Red(-b)$.

  We can make the choice of $b$ canonical if we have additional structure. Write $b = \iprod_{\sigma \in J} b_\sigma = rt^{-x} + b'$ with $v(b') > x$. The range of possible choices for $x$ is precisely the convex class of the elements $y \in \Gbf$ such that $(x-y) \prec \supp(b)$. If we are working in $\no$, we may then require that $r = 1$ and that $x$ is minimal with respect to simplicity. This determines a unique $b$. More generally, if $\Gbf$ is equipped with a truncation closed embedding into a Hahn group, one can require that $x$ is minimal under truncation and that $r = 1$, which again determines a unique choice for $b$.
\end{rem}

\begin{rem}\label{rem:image-of-Red}
  For a family $(b_\sigma : \sigma \in J)$ to be in the image of $\Red$, and thus have a value for $\iprod_{\sigma \in J} b_\sigma$, one must at least have:
  \begin{enumerate}
    \item $v(b_\sigma) \in \sigma$ for all $\sigma \in J$;
    \item $J$ is reverse well ordered;
    \item if $\sigma \in J$ is not the minimum of $J$, then $0 \in \supp(b_\sigma)$.
  \end{enumerate}
  When $J$ is finite, the above conditions are sufficient: one can immediately verify that under those assumptions, $b = \prod_{\sigma \in J} b_\sigma$ satisfies $\Red(b) = (b_\sigma : \sigma \in J)$.

  This is not the case for $J$ infinite. For instance, there is no $b \in \oz$ such that
  \[ \Red(b) = \left(\omega^{\!\sqrt[n]{\omega}} + 1 : \left[\!\sqrt[n]{\omega}\right] \in \no_{/\asymp}\right) \]
  where $n$ ranges over the positive natural numbers. Indeed, if such $b$ exists, we may also define $b_i = \iprod_{n \geq i} \left(\omega^{\!\sqrt[n]{\omega}} + 1\right)$. Note that $v(b_i) \asymp \omega^{\sqrt[i]{\omega}}$ and
  \[ b = b_1 = \omega^\omega b_2 + b_2 = \omega^\omega(\omega^{\!\sqrt{\omega}}b_3 + b_3) + b_2 = \dots. \]
  In each sum $\omega^{\!\sqrt[n]{\omega}}b_{n+1} + b_{n+1}$, the supports of the two summands are disjoint. It follows that $b$ contains monomials from $\omega^\omega b_2$, $\omega^{\omega + \sqrt{\omega}} b_3$, $\omega^{\omega + \sqrt{\omega} + \!\sqrt[3]{\omega}} b_4$, $\dots$, thus its support is not well ordered, a contradiction. It would be interesting to find a reasonable characterisation of the families in the image of $\Red$, although it is not needed for the rest of this paper.
\end{rem}

\begin{defn}
  \label{def:infinite-product}
  Let $(b_i : i \in I)$ be some family of reduced series, where $I$ is some set. We write $\iprod_{i \in I} b_i = b$ when:
  \begin{itemize}
    \item for every $\sigma \in \Gbf_{/\asymp}$, the class $\{i \in I : v(b_i) \in \sigma\}$ is finite;
    \item if we let $J = \{[v(b_i)] : i \in I\}$ and $c_\sigma = \prod_{v(b_i) \in \sigma}b_i$, then $b = \iprod_{\sigma \in J} c_\sigma$ as per \prettyref{def:inverse-R}.
  \end{itemize}
\end{defn}
Note that when the family $(b_i : i \in I)$ is of the form $(b_\sigma : \sigma \in I)$ with $I \subseteq \Gbf_{/\asymp}$ and $v(b_\sigma) \in \sigma$ for all $\sigma \in I$, the above definition of $\iprod_{i \in I}b_i$ coincides with \prettyref{def:inverse-R}.

\begin{thm}
  \label{thm:ZKGG}
  Let $b \in \ZKGG_\kappa$ be non-zero. Then there are a family $(c_i : i \in I)$ of reduced series $c_i \in \ZKGG_\kappa$ and a series $d \in \Kbf((\Gbf))_\kappa$ such that $b = d \iprod_{i \in I} c_i$, where $\supp(d) \prec \supp(b)$ or $d = 1$, and for each $i \in I$, $v(c_i) \nprec \supp(b)$ and $c_i$ is one of the following:
  \begin{enumerate}[label=(\arabic*),ref=\arabic*]
    \item\label{item:pseudo-poly-irred} an irreducible pseudo-polynomial with $\lspan{\osupp(c_i)}$ of dimension $\geq 2$;
    \item\label{item:pseudo-poly-dim1} a pseudo-polynomial with $\langle\osupp(c_i)\rangle_\Qb$ of dimension $1$ and $0 \in \supp(c_i)$;
    \item\label{item:pseudo-mono} a pseudo-monomial;
    \item\label{item:almost-irred} an almost irreducible series with infinite pseudo-support.
  \end{enumerate}
  Moreover:
  \begin{enumerate}[label={(\alph*)},ref=\alph*]
    \item\label{item:pseudo-poly-irred-unique} the factors in \prettyref{item:pseudo-poly-irred} are unique up to reordering;
    \item\label{item:pseudo-poly-dim1-unique} for each $\sigma$, we may assume that the pseudo-supports of the factors $c_i$ as in \prettyref{item:pseudo-poly-dim1} with $[v(c_i)] = \sigma$ are pairwise $\Qb$-linearly independent, in which case they are unique up to reordering;
    \item\label{item:pseudo-mono-irred-min} the factors $c_i$ in \prettyref{item:almost-irred} that are reducible and the ones in \prettyref{item:pseudo-mono} must satisfy $[v(c_i)] = \min\{[x] : x \in \supp(b)\}$ (and thus are at most finitely many);
    \item\label{item:d=1} if $\{[x] : x \in \supp(b)\}$ has a minimum, then we may assume $d = 1$;
    \item\label{item:sup-in-H} if $\{[x] : x \in \supp(b)\}$ has a minimum $\sigma$, and $\sup_\sigma(b) \in H_\sigma$, then we may assume that the factors in \prettyref{item:almost-irred} are pseudo-irreducible, in which case the product of all (finitely many) factors in \prettyref{item:pseudo-mono} is unique up to multiplication by $e \in \Kbf((\Gbf))_\kappa$ with $\supp(e) \prec \supp(c_i)$;
    \item\label{item:no-pseudo-irred} if $\supp(b)$ has a maximum, then we may assume that the factors in \prettyref{item:almost-irred} are irreducible, in which case the product of $d$ and all factors in \prettyref{item:pseudo-mono} is unique.
  \end{enumerate}
\end{thm}
\begin{proof}
  Write $b = d\iprod_{\sigma \in J}\rho_\sigma(b)$ as per \prettyref{def:inverse-R} and discussion thereafter, thus with $J = \{\sigma : \supp(b) \cap \sigma \neq \varnothing\}$. We then apply \prettyref{prop:reduced-fact-unique} to each $\rho_\sigma(b)$ for $\sigma \in J$ and find
  \[ \rho_\sigma(b) = p_\sigma m_\sigma c_{1,\sigma} \cdots c_{n_\sigma,\sigma}. \]
  where $p_\sigma$ is a (unique) reduced pseudo-polynomial with $0 \in \supp(p_\sigma)$, and $[v(p_\sigma)] = \sigma$ or $p_\sigma = 1$; $m_\sigma$ is a pseudo-monomial with $[v(m_\sigma)] = \sigma$ or $m_\sigma = 1$; each $c_{i,\sigma}$ is almost irreducible with infinite support and $[v(c_{i,\sigma})] = \sigma$. We discard the factors $m_\sigma = 1$ for $\sigma \neq [0]$ and the factors $p_\sigma = 1$. We further apply \prettyref{prop:ritt-pseudo-poly} to the remaining factors $p_\sigma \neq 1$. Let $(c_i : i \in I)$ be the resulting collection of reduced series, which are of the required types \prettyref{item:pseudo-poly-irred}, \prettyref{item:pseudo-poly-dim1}, \prettyref{item:pseudo-mono}, \prettyref{item:almost-irred}, and satisfy $b = d\iprod_{i \in I} c_i$.

  Now suppose $b = d\iprod_{i \in I} c_i$ is a factorisation as in the statement. By definition of product, for any $\sigma \in \Gbf_{/\asymp}$, $\{i \in I : v(c_i) \in \sigma\}$ is finite, and it is non-empty if and only if $\supp(b) \cap \sigma \neq \varnothing$, in which case $\prod_{v(c_i) \in \sigma} c_i = \rho_\sigma(b)$. Pick some $i \in I$ and let $\sigma = [v(c_i)]$.

  If $c_i$ is a pseudo-polynomial as in \prettyref{item:pseudo-poly-irred}, \prettyref{item:pseudo-poly-dim1}, then $c_i$ divides $\rho_\sigma(b)$, and conclusions \prettyref{item:pseudo-poly-irred-unique}, \prettyref{item:pseudo-poly-dim1-unique} follow at once from Propositions~\ref{prop:ritt-pseudo-poly},~\ref{prop:reduced-fact-unique}.

  If there is an $x \in \supp(b)$ such that $[x] \prec \sigma$, then $\tau_\sigma(\rho_\sigma(b)) = 1$, hence $0 \in \supp(\rho_\sigma(b))$, which implies $0 \in \supp(c_i)$. In turn, $c_i$ cannot be a pseudo-monomial, as it would be equal to $1$ by \prettyref{prop:reduced-fact-unique}, contradicting the assumption $[x] \prec \sigma$, and if $c_i$ is almost  irreducible, then it is irreducible by \prettyref{prop:fact-almost-irreducible}. Switching to the contrapositive, if $c_i$ is a non-irreducible almost irreducible factor, or is a pseudo-monomial, then $\sigma$ is the minimum of $\{[x] : x \in \supp(b)\}$, and $c_i$ is a factor of $\rho_\sigma(b)$, hence it belongs to a finite subset of the family $(c_i : i \in I)$, proving \prettyref{item:pseudo-mono-irred-min}.

  Now assume that $\{[x] : x \in \supp(b)\}$ has a minimum $\sigma$. If $d \neq 1$, then $0 \notin \supp(b)$, hence $\tau_\sigma(b) = 0$, so $\rho_\sigma(b) = T_\sigma(b)$ and $0 \notin \supp(\rho_\sigma(b))$, thus $0 \notin \supp(c_i)$ for some $c_i$ with $[v(c_i)] = \sigma$, and in this case, we may replace $c_i$ with $dc_i$ and assume $d = 1$, as claimed in \prettyref{item:d=1}.

  Assume moreover that $\sup_\sigma(b) \in H_\sigma$, or equivalently that $\sup_\sigma(\rho_\sigma(b)) \in H_\sigma$. Then by \prettyref{prop:reduced-fact-unique}, we may assume that the factors in \prettyref{item:almost-irred} are pseudo-irreducible, and in that case, the product of all pseudo-monomial factors together with $d$ is the pseudo-monomial factor of $\rho_\sigma(b)$ given by \prettyref{prop:reduced-fact-unique}, up to multiplication by some $e \in \Kbf((\Gbf))_\kappa$ with $\supp(e) \prec v(c_i)$. This shows \prettyref{item:sup-in-H}.

  To conclude, suppose that $\supp(b)$ has a maximum. This implies in particular that $\{[x] : x \in \supp(b)\}$ has a minimum $\sigma$ and that $\supp(\rho_\sigma(b))$ has a maximum. By \prettyref{prop:reduced-fact-unique} again, we may assume that the factors of $\rho_\sigma(b)$ of type \prettyref{item:almost-irred} appearing in the factorisation are irreducible. In this case, the product of all pseudo-polynomial factors and $d$ is unique, proving \prettyref{item:no-pseudo-irred}.
\end{proof}

\begin{proof}[Proof of \prettyref{main:Oz}]
  We apply \prettyref{thm:ZKGG} to $\oz = \Zb + \Rb((\no^{<0}))_\on$. When $0 \in \supp(b)$, $\supp(b)$ in particular has a maximum, and $\sigma = [0]$ is the minimum $\sigma$ such that $\rho_\sigma(b) \neq 0$. Thus, in the factorisation given by \prettyref{thm:ZKGG}, the factors in \prettyref{item:almost-irred} are irreducible, the product of the pseudo-monomials \prettyref{item:pseudo-mono} is unique and in $\Zb$, and $d = 1$. To conclude, it suffices to split the factors in $\Zb$ into irreducible factors.
\end{proof}

\begin{rem}
  In $\oz$, the groups $H_\sigma$ are always isomorphic to $\Rb$ by \prettyref{prop:H-sigma-is-R}, thus the second premise of \prettyref{item:sup-in-H} is automatically true. In particular, the factors of infinite pseudo-support that appear in $\oz$ are always at least pseudo-irreducible, and the product of the pseudo-monomials is unique up to multiplication by surreal numbers with small support as specified in \prettyref{item:sup-in-H}.
\end{rem}

\begin{rem}\label{rem:infinite-products}
  Surreal numbers already have a natural definition of infinite product, thanks to the presence of a (canonical) exponential function:
  \[ \prod_{i \in I} b_i \coloneqq \exp\left(\sum_{i \in I} \log(b_i)\right). \]
  This is only defined when the family $(\log(b_i) : i \in I)$ is \emph{summable}, that is to say, the union of the supports $\supp\left(\log(b_i)\right)$ is well ordered and each of its monomials appears in $\supp(\log(b_i))$ for at most finitely many $i \in I$. We shall  see that $\prod$ and $\iprod$ disagree in general. We shall use the following facts about surreal exponentiation:
  \begin{itemize}
      \item $\log(\omega^{\omega^x})$ is a monomial for every $x \in \no$;
      \item $\log\omega^{\sum_{i < \alpha} r_i\omega^{x_i}} = \sum_{i<\alpha}r_i\log\left(\omega^{\omega^{x_i}}\right)$;
      \item $\log(1 + \varepsilon) = \sum_{n \in \Nb} (-1)^n \frac{\varepsilon^{n+1}}{n+1}$ for all $\varepsilon \prec 1$;
      \item $\exp$ and its inverse $\log$ are strictly increasing.
  \end{itemize}
  We refer the reader to \cite{Gon1986} for the definition of $\exp$ and $\log$ on surreal numbers and proofs of the above properties.

  First of all, the two products have different domains. $\iprod$ is only defined on the families $(b_i : i \in I)$ where every $b_i$ is reduced and if $b_i \succ b_j$, then $0 \in \supp(b_i)$, as explained in \prettyref{rem:image-of-Red}. For instance, the family $(\omega^{\sqrt[n+1]{\omega}} : n \in \Nb)$ is not in the domain of $\iprod$, but
  \[ \prod_{n \in \Nb}\omega^{\sqrt[n+1]{\omega}} = \omega^{\sum_{n \in \Nb}\sqrt[n+1]{\omega}}.\]

  There are also families in the domain of $\iprod$ that do not have natural infinite product. Pick any sequence of positive monomials $x_1 \succ x_2 \succ x_3 \succ \dots \succ \omega$; for instance, $x_n = \sqrt[n]{\omega^{\omega}}$. Let
  \[ b = \omega^{x_1} +
    \omega^{x_2 - \omega} +
    \omega^{x_3 - 2\omega} + \dots + 1 \in \oz. \]
  Let $b_0 \coloneqq \rho_{[0]}(b) = 1$, $b_{n+1} \coloneqq \rho_{[x_{n+1}]}(b)$. Since $0 \in \supp(b)$, we have $b = \iprod_{n \in \Nb} b_n$. For $n > 0$ we have
  \[ b_n = \frac{\omega^{x_n - x_{n+1} + \omega}}{1 + \omega^{-x_{n+1}+x_{n+2}-\omega} + \omega^{-x_{n+1}+x_{n+3}-2\omega} + \dots + \omega^{-x_{n+1} + n\omega}} + 1. \]
  Therefore, for $n > 0$, $\log(b_n) = \log\left(\omega^{x_n - x_{n+1} + \omega}\right) + \varepsilon_n = \log\left(\omega^{x_n}\right) - \log\left(\omega^{x_{n+1}}\right) + \log\left(\omega^{\omega}\right) + \varepsilon_n$ where $\varepsilon_n \prec 1$. Since $\log\left(\omega^{x_n}\right)$, $\log\left(\omega^\omega\right)$ are monomials and $\log\left(\omega^{x_n}\right) \succ \log\left(\omega^\omega\right) \succ 1$, the monomial $\log\left(\omega^\omega\right)$ appears infinitely many times, thus $(\log(b_n) : n \in \Nb)$ is not summable, and so the natural product of $(b_n : n \in \Nb)$ is not defined.

  A minor tweak shows that the two products may even disagree on their common domain when infinite families are involved (we recall that $\prod$ and $\iprod$ agree on finite families, as explained in \prettyref{rem:image-of-Red}). Let
  \[ c = \omega^{x_1 - \omega} +
    \omega^{x_2 - \omega} +
    \omega^{x_3 - \omega} + \dots + 1. \]
  As before, let $c_0 \coloneqq \rho_{[0]}(c)= 1$, $c_{n+1} \coloneqq \rho_{[x_{n+1}]}(c)$, so that $c = \iprod_{n \in \Nb} c_n$. For $n > 0$,
  \[ c_n = \frac{\omega^{x_n - x_{n+1}}}{1 + \omega^{-x_{n+1}+x_{n+2}} + \omega^{-x_{n+1}+x_{n+3}} + \dots + \omega^{-x_{n+1}+\omega}} + 1. \]
  On writing $c_n = \omega^{x_n - x_{n+1}}(1 + \delta_n)$, the support of $\delta_n$ is contained in the additive semigroup generated by $(-x_n + x_{n+1})$, $(-x_{n+1}+\omega)$, and $(-x_{n+1} + x_{n+k+2})$ for $k \in \Nb$. Likewise,
  \[ \log(c_n) = \log\left(\omega^{x_n}\right) - \log\left(\omega^{x_{n+1}}\right) + \varepsilon_n \]
  where $\supp(\varepsilon_n)$ is contained in the same semigroup. In particular, $\supp(\varepsilon_n) \subseteq [x_n] \cup [x_{n+1}]$, thus $\supp(\varepsilon_n) \cap \supp(\varepsilon_{n+k+2}) = \varnothing$ for all $n, k \in \Nb$. It follows that the family $(\log(c_n) : n \in \Nb)$ is summable, and
  \[ \prod_{n \in \Nb} c_n = \exp\left(\sum_{n \in \Nb}\log(c_n)\right) = \omega^{x_1}\exp\left(\sum_{n \in \Nb} \varepsilon_n\right) \asymp \omega^{x_1} \not\asymp \omega^{x_1 - \omega} \asymp c. \]
\end{rem}

\subsection{The \texorpdfstring{$\sup$}{sup} valuation}
\label{sub:sup-valuation}

We report here another consequence of \prettyref{cor:deg-sup-sigma-mult}: the function $b \mapsto \sup(\supp(b))$, where the supremum is taken in the Dedekind-MacNeille completion of $\Gbf$, is a multiplicative valuation, up to taking an appropriate quotient. This will not be used in the rest of the paper, so the reader is free to skip to the next section. The notation is similar to the one of \cite[\S XIII.13]{Bir1979}, but we use $\oplus$ for the sum in the completion of $\Gbf$ to avoid ambiguities; also, we tweak the definition of completion so that it works when $\Gbf$ is a proper class.

Note that here \emph{we do not assume that $\Gbf$ is divisible}.

\begin{defn}
  Given a set $A \subseteq \Gbf$, let $L(A) \coloneqq \{x \in \Gbf : x \leq A\}$, $U(A) \coloneqq \{x \in \Gbf : x \geq A\}$, and $A^\# \coloneqq L(U(A)) \supseteq A$. We call \textbf{restricted Dedekind-MacNeille completion} of $\Gbf$ the class $\overline{\Gbf}_\kappa \coloneqq \{A^\# : A \subseteq \Gbf\} = \{L(A) : A \subseteq \Gbf, A \text{ set}, |A| < \kappa\}$ (where $|A| < \on$ is vacuously true). Given $\eta, \zeta \in \overline{\Gbf}_\kappa$, $x \in \Gbf$, we define the following:
  \[ \eta \leq \zeta \xLeftrightarrow{\text{def}} \eta \subseteq \zeta, \quad \eta \oplus \zeta \coloneqq (\eta + \zeta)^\#, \quad x^\# \coloneqq \{x\}^\#. \]
\end{defn}

\begin{prop}
  \label{prop:dedekind}\phantom{ }
  \begin{enumerate}
  \item\label{item:G-bar-monoid} $\overline{\Gbf}_\kappa$ equipped with $\leq$, $\oplus$ is an ordered commutative monoid with identity $0^\#$;
  \item\label{item:sharp-embedding} $x \mapsto x^\#$ is an ordered group embedding;
  \item\label{item:G-dense} for every $\eta < \zeta$ in $\overline{\Gbf}_\kappa$, there is $x \in \Gbf$ such that $\eta \leq x < \zeta$.
  \end{enumerate}
\end{prop}
\begin{proof}
  Straightforward from the definitions and left to the reader.
\end{proof}
In light of \prettyref{prop:dedekind}(\ref{item:sharp-embedding}), and the fact that $\eta \leq x^\#$ is equivalent to $\eta \leq x$, we shall identify $\Gbf$ with its copy inside $\overline{\Gbf}_\kappa$.

\begin{defn}
  For $\eta, \zeta \in \overline{\Gbf}_\kappa^{\leq 0}$, write $\eta \sim \zeta$ if for every $n \in \Nb$ and $n\eta \leq x,y \leq n\zeta$ (or $n\zeta \leq x,y \leq n\eta$) we have $x - y \prec y$.
\end{defn}

\begin{prop}
  \label{prop:sim-congruence}
  The relation $\sim$ is a congruence on $\overline{\Gbf}_\kappa^{\leq 0}$ with convex equivalence classes.
\end{prop}
\begin{proof}
  In what follows, let $\eta, \zeta, \xi \in \overline{\Gbf}_\kappa^{\leq 0}$.

  Reflexivity and symmetry of $\sim$ are immediate from the definition. For transitivity, say that $\eta \sim \zeta \sim \xi$, with $\eta \leq \xi$. Let $x,y \in \Gbf$ be such that $\eta \leq x \leq y \leq \xi$. If $\zeta \leq \eta$ or $\xi \leq \zeta$, we immediately obtain that $x - y \prec y$, so assume $\eta \leq \zeta \leq \xi$. If $x+y \leq 2\zeta$, then $2\eta \leq 2x, x+y \leq 2\zeta$, hence $2x -(x+y) = x - y \prec 2y \asymp y$. Likewise, if $2\zeta \leq x+y$, then $y - x \prec y$. By applying the same argument with $n\eta$ and $n\xi$ in place of respectively $\eta$ and $\xi$, we obtain that $\eta \sim \xi$. Note that we have also verified that the $\sim$-equivalence classes are convex.

  Now say that $\eta \sim \zeta$, with $\eta \leq \zeta$. We wish to prove that $\eta \oplus \xi \sim \zeta \oplus \xi$ for any $\xi$. We may assume that $\xi \neq 0$. Let $x,y \in \Gbf$ be such that $\eta \oplus \xi \leq x < y \leq \zeta \oplus \xi$. By construction, $y \neq 0$. Suppose by contradiction that there exists some $n \in \Nb$ such that $n(x - y) \leq 2y$, or in other words $nx \leq (n+2)y$. In particular,
  \[ n(\eta \oplus \xi) \leq nx \leq (n+2)y < (n+1)y < ny \leq n(\zeta \oplus \xi). \]
  Since $(n+1)y < n\zeta \oplus n\xi$, there are $z \in n\zeta$ and $w \in n\xi$ such that $(n+1)y \leq z + w$. By assumption, $(n+2)y \geq n\eta + w$, thus $(n+2)y - w \geq n\eta$. On the other hand, $(n+1)y - w \leq z \leq n\zeta$. It follows that $((n+2)y - w) - ((n+1)y - w) = y \prec y$, a contradiction.

  Therefore, $n(x-y) > 2y$ for every $n \in \Nb$, thus $x - y \prec y$. It now suffices to apply the same argument with $n\eta$, $n\zeta$, $n\xi$ in place of respectively $\eta$, $\zeta$, $\xi$ to deduce that $\eta \oplus \xi \sim \zeta \oplus \xi$, as desired.
\end{proof}

\begin{defn}
  Given $b \in \KGG_\kappa$, let $\sup(b)$ be the $\sim$-equivalence class of $\supp(b)^\#$. Let $\tilde{\Gbf}_\kappa^{\leq 0}$ be the quotient of $\overline{\Gbf}_\kappa^{\leq 0}$ by $\sim$.
\end{defn}

\begin{prop}
  \label{prop:sup-G}
  The map $\sup : \KGG_\kappa \to \tilde{\Gbf}_\kappa^{\leq 0}$ is a multiplicative valuation.
\end{prop}
\begin{proof}
  Note that the quotient $\tilde{\Gbf}_\kappa^{\leq 0}$ is a well-defined ordered monoid by \prettyref{prop:sim-congruence}. Let $b,c \in \KGG_\kappa$ be some arbitrary series. Note that $\sup(0) = [\varnothing^\#] \eqqcolon -\infty$, the minimum element of $\tilde{\Gbf}_\kappa^{\leq 0}$. If $b \neq 0$, any $x \in \supp(b)$ is such that $\varnothing^\# \leq 2x, x \leq \supp(b)^\#$, while however $2x - x \not\prec x$, thus $\sup(b) \neq [\varnothing^\#]$. Therefore, $\sup(b) = -\infty$ if and only if $b = 0$.

  Since $\supp(b + c) \subseteq \supp(b) \cup \supp(c)$ and $\supp(bc) \subseteq \supp(b) + \supp(c)$, we have $\sup(b+c) \leq \max\{\sup(b),\sup(c)\}$ and $\sup(bc) \leq \sup(b) \oplus \sup(c)$. Thus, it remains to prove that the latter inequality is an equality. For the sake of notation, let $\eta = \supp(b)^\#$, $\zeta = \supp(c)^\#$, $\xi = \supp(bc)^\#$. We know that $\xi \leq \eta \oplus \zeta$, and we wish to prove that $\xi \sim \eta \oplus \zeta$.

  Let $x, y \in \Gbf$ be such that $n\xi \leq x \leq y \leq n(\eta \oplus \zeta)$. We need to show that $x - y \prec y$. Let $\sigma = [x]$. Let $\pi_\sigma$ be the projection from $\Gbf_{\preceq \sigma}$ to $H_\sigma$. By construction, $\pi_\sigma(y) \geq \pi_\sigma(x) \geq n\sup_\sigma(bc)$. By \prettyref{cor:deg-sup-sigma-mult}, $n\sup_\sigma(bc) = n\sup_\sigma(b) + n\sup_\sigma(c)$. Suppose by contradiction that $x - y \succeq y$, thus equivalently that $\pi_\sigma(x) < \pi_\sigma(y)$. Then $\pi_\sigma(y) > n\sup_\sigma(b) + n\sup_\sigma(c)$, from which it follows that $y > n(\eta \oplus \zeta)$, a contradiction.
\end{proof}

\section{On primal series}
\label{sec:on-primal-series}

Let us continue working on $\ZKGG_\kappa$. We keep using the notations of \prettyref{sub:archimedean}. We shall prove the following.

\begin{thm}
  \label{thm:non-pS}
  Suppose that $\ZKGG_\kappa$ is a pre-Schreier domain. Then the following conditions hold for all $\sigma \in \Gbf_{/\asymp}$:
  \begin{enumerate}
  \item[\textup{(A1)}\textsubscript{$\sigma$}]\label{item:arch-quot} $H_\sigma \cong \Rb$ or $\sigma = [0]$;
  \item[\textup{(A2)}\textsubscript{$\sigma$}]\label{item:big-cof} $\Gbf_{\prec \sigma}$ has cofinality $\geq \kappa$, or $\Gbf_{\prec \sigma} = \{0\}$ and $\Frac(\Zbf) = \Kbf$, or $\sigma = [0]$;
  \item[\textup{(A3)}\phantom{\textsubscript{$\sigma$}}]\label{item:Z-pS} $\Zbf$ is a pre-Schreier domain.
  \end{enumerate}
  Moreover, $\Lbf_\sigma((\Rb^{\leq 0}))$ is a pre-Schreier domain for every $\sigma \in \Gbf_{/\asymp}$.
\end{thm}

The above three conditions are satisfied by $\oz$, for every $\sigma$: the first one by \prettyref{prop:H-sigma-is-R}; the second one by \prettyref{prop:no-G-sigma-cofinality}; the third one because $\Zb$ is a unique factorisation domain.

Conversely, we prove that if the above conditions are satisfied, then primal elements in $\LHH$ yield primal elements of $\ZKGG_\kappa$ via the map $\iota_\sigma$, generalising \cite[Cor.\ 4.3]{BKK2006}. In this way, we will produce many new primal elements in $\oz$ (\prettyref{exas:primes-primal}).

\subsection{Producing non-primal series}
\label{sub:not-pre-schreier}
We start by proving \prettyref{thm:non-pS}. The easiest condition to deal with is \refZpS{}.

\begin{prop}
  \label{prop:Z-not-pS-not-pS}
  If $\ZKGG_\kappa$ is a pre-Schreier domain, then \refZpS{} holds.
\end{prop}
\begin{proof}
  Suppose that $\ZKGG_\kappa$ is a pre-Schreier domain and pick $b \in \Zbf$. Suppose that $b$ divides $cd$ in $\Zbf$ for some $c, d \in \Zbf$, namely $\frac{cd}{b} \in \Zbf$. Then there are $b_1, b_2 \in \ZKGG_\kappa$ such that $b_1b_2 = b$ and $\frac{c}{b_1}, \frac{d}{b_2} \in \ZKGG_\kappa$. Since $v(b_1) = v(b_2) = v(b) = 0$, we have $b_1, b_2 \in \Zbf$. In particular, we also have $\frac{c}{b_1}, \frac{d}{b_2} \in \Zbf$. By letting $c$, $d$ range over $\Zbf$, we find that $b$ is primal in $\Zbf$. Since $b$ was arbitrary, $\Zbf$ is a pre-Schreier domain, proving \refZpS{}.
\end{proof}

For the remaining assumptions \refarchquot{}, \refbigcof{}, we exhibit series that are not primal when one of those assumptions is not satisfied.

\begin{prop}
  \label{prop:G-not-complete-not-pS}
  Let $H$ be a non-trivial, non-discrete Archimedean group that is not complete. Then for every $x \in H^{<0}$ and non-zero $k \in \Kbf$, $kt^x$ is not primal in $\Zbf + \Kbf((H^{<0}))$. In particular, $\Zbf + \Kbf((H^{<0}))$ is not a pre-Schreier domain.
\end{prop}
\begin{proof}
  Suppose without loss of generality that $H \subsetneq \Rb$ and let $x \in H^{<0}$. Since $H$ is a group, $\Rb \setminus H$ is dense in $\Rb$, so we can find some $y \in \Rb \setminus H$ such that $x < y < 0$. Since $H$ is not discrete, it is dense in $\Rb$, so we can find a increasing sequence $(y_n \in H)_{n \in \Nb}$ such that $\sup_{n \in \Nb} y_n = y$. Likewise, we can find a increasing sequence $(z_n \in H)_{n \in \Nb}$ such that $\sup_{n \in \Nb} z_n = x - y$. Let
  \[ b = \sum_{n \in \Nb} t^{y_n}, c = \sum_{n \in \Nb} t^{z_n} \in \Zbf + \Kbf((H^{<0})). \]
  By \prettyref{prop:sup-valuation}, $\sup(bc) = \sup(b) + \sup(c) = y + (x-y) = x \in H$. This implies that $t^x$ divides $bc$. On the other hand, if we write $t^x = d_1d_2$ with $d_1, d_2 \in \KHH$, then $d_i = k_it^{w_i}$ for some $k_i \in \Kbf$ and $w_i \in H^{\leq 0}$ such that $k_1k_2 = 1$, $w_1 + w_2 = x$. Since $y \notin H$, then we either have $w_1 < y$, or $w_2 < x - y$. Therefore, either $d_1$ divides $b$, or $d_2$ divides $c$, but not both. This shows that $t^x$ is not primal.
\end{proof}
\begin{exa}
  \label{exa:RQQ-not-pre-schreier}
  The series $t^{-1}$ is not primal in $\Rb((\Qb^{\leq 0}))$, so $\Rb((\Qb^{\leq 0}))$ is not a pre-Schreier domain, hence also not a GCD domain. On the other hand, note that if $H$ is a discrete Archimedean group, then $\Kbf((H^{\leq 0})) = \Kbf(H^{\leq 0})$, because the only well ordered subsets of $H$ are finite, hence it is a GCD domain.
\end{exa}

\prettyref{prop:G-not-complete-not-pS} suggests that monomials may be the only source of non-primal elements when $H$ is Archimedean. It seems appropriate to localise $\ZKHH$ by monomials, namely work in the ring
\begin{multline*}
  t^H\left(\ZKHH\right) \coloneqq \{ t^xb : x \in H,\ b \in \ZKHH \} = \\
    = \{ b \in \Kbf((H)) : \supp(b) \text{ is bounded}\} = t^H\KHH.
\end{multline*}
In this ring, all monomials are units, and conversely, all units are monomial (since the map $\sup : t^H\KHH \to H$ defined by $\sup(t^xb) \coloneqq x + \sup(b)$ is clearly a valuation, and $v(b) \geq \sup(b)$ for all $b$, thus the units must satisfy $v(b) = \sup(b)$). This prompts the following question:
\begin{que*}
  Let $H$ be a divisible Archimedean group. Is $t^H\KHH$ a pre-Schreier domain?
\end{que*}

\begin{cor}
  \label{cor:H-not-complete-not-pS}
  Let $\sigma \in \Gbf_{/\asymp}$ be such that \refarchquot{} does not hold. Then every pseudo-monomial $m$ with $[v(m)] = \sigma$ (for instance, $t^x$ for some $x \in \sigma^{<0}$) is not primal in $\ZKGG_\kappa$. In particular, $\ZKGG_\kappa$ is not a pre-Schreier domain.
\end{cor}
\begin{proof}
  Let $\sigma$ as in the hypothesis. By \prettyref{prop:G-not-complete-not-pS}, if $m \in \ZKGG_\kappa$ is a pseudo-monomial with $[v(m)] = \sigma$, namely $\iota_\sigma(m) = kt^x$ for some $x \in H_\sigma^{<0}$ and $k \in \Lbf_\sigma$, then $\iota_\sigma(m)$ is not primal in $\SLHH$. Then $m$ is not primal in $\Zbf + \Kbf((\Gbf_{\preceq \sigma}^{< 0}))_\kappa$, and it follows that $m$ is not primal in $\ZKGG_\kappa$ by \prettyref{prop:b-div-c-div-T}.
\end{proof}

\begin{prop}
  \label{prop:G-smallcof-not-pS}
  If there is some $\sigma \in \Gbf_{/\asymp}$ such that \refbigcof{} does not hold, then $\ZKGG_\kappa$ is not a pre-Schreier domain.
\end{prop}
\begin{proof}
  Let $\sigma$ be as in the hypothesis. Then $\Lbf_\sigma \neq \Frac(\Sbf_\sigma)$: when $\Gbf_{\prec \sigma} \neq \{0\}$, it follows from \prettyref{prop:frac-Z-is-K}, otherwise simply because $\Sbf_\sigma = \Zbf$ and $\Lbf_\sigma = \Kbf$. Let $\eta \in \Lbf_\sigma \setminus \Frac(\Sbf_\sigma)$, and set
  \[ b = t^{-x}, \quad c = \eta t^{-x}, \quad d = \sum_{n \in \Nb} t^{-\frac{x}{n+1}} \]
  where $x$ is some element of $\Gbf^{<0}$ with $x \in \sigma$. Note that $d$ is pseudo-irreducible.

  We have that $b$ divides $cd$, since $\frac{cd}{b} = \eta\sum_{n \in \Nb} t^{\frac{-x}{n+1}} \in \ZKGG_\kappa$. Now suppose that $b = b_1b_2$ for some $b_1,b_2 \in \ZKGG_\kappa$ with $b_2$ dividing $d$. Since $b$ is a pseudo-monomial, $b_2$ is a pseudo-monomial with $v(b_2) \asymp x$ or it satisfies $v(b_2) \prec x$ by Propositions~\ref{prop:fact-pseudo-monomial},~\ref{prop:fact-almost-irreducible}\prettyref{item:almost-irred-div}. Since $d$ is pseudo-irreducible, we must be in the latter case by \prettyref{prop:fact-almost-irreducible}\prettyref{item:pseudo-irred-div}, hence $b_2 \in \Sbf_\sigma$. It follows that $\frac{c}{b_1} = \frac{cb_2}{t^{-x}} = \eta b_2 \in \Lbf_\sigma \setminus \Sbf_\sigma$. It follows that $\frac{c}{b_1} \notin \ZKGG_\kappa$, which means that $b_1$ does not divide $c$. This shows that $b$ is not primal in $\ZKGG_\kappa$.
\end{proof}

\begin{exa}
  A notable example is $\Zb + \Kbf((\Rb^{<0}))$ for any $\Kbf \neq \Qb$. In this case, for $\sigma = [1]$, we have $\Gbf_{\prec \sigma} = \{0\}$ and $\Frac(\Zb) \neq \Kbf$. As the above proof shows, $t^{-1}$ divides $(\sqrt{2}t^{-1})\left(\sum_{n \in \Nb}t^{\frac{-1}{n+1}}\right)$, but there is no factorisation $t^{-1} = b_1b_2$ such that $b_1$ divides $\sqrt{2}t^{-1}$ and $b_2$ divides $\sum_{n \in \Nb}t^{\frac{-1}{n+1}}$. Thus, $\Zb + \Kbf((\Rb^{<0}))$ is neither a pre-Schreier domain, nor in particular a GCD domain.
\end{exa}

\begin{proof}[Proof of \prettyref{thm:non-pS}]
  By Propositions~\ref{prop:Z-not-pS-not-pS},~\ref{prop:G-smallcof-not-pS}, and \prettyref{cor:H-not-complete-not-pS}.
\end{proof}

\subsection{Lifting primal elements}
\label{sub:lifting-primal}

First, we generalise the primality transfer result \prettyref{prop:lift-primal-easy-reduced}. We start with an abstract transfer lemma.

\begin{lem}
  \label{lem:primal}
  Let $\Lbf$ be a field and $\Abf$ be a commutative $\Lbf$-algebra equipped with an $\Lbf$-algebra morphism $\pi : \Abf \to \Lbf$. Let $\Sbf \subseteq \Lbf$ be a subring. An element $b \in \pi^{-1}(\Sbf)$ is primal in $\pi^{-1}(\Sbf)$ if and only if:
  \begin{enumerate}
    \item\label{item:pi(b)-not-0} either $\pi(b) \neq 0$ is primal in $\Sbf$ and $b$ is primal in $\Abf$,
    \item\label{item:pi(b)-is-0} or $\pi(b) = 0$ and $b$ is primal in $\pi^{-1}(\Frac(\Sbf))$.
  \end{enumerate}
\end{lem}
\begin{proof}
  Let $b \in \pi^{-1}(\Sbf)$. In a small abuse of notation, we shall write $\frac{c}{d} \in R$ to mean that there is some $e \in R$ such that $de = c$, and let $\frac{c}{d}$ denote any such witness $e$ (of which there may be several if $d$ is a zero divisor). Let us prove the left-to-right direction first.

  Let us work in the case $\pi(b) \neq 0$. Suppose that $b$ is primal in $\pi^{-1}(\Sbf)$. Say that $\frac{cd}{\pi(b)} \in \Sbf$ for some $c, d \in \Sbf$. By assumption, $b' = \frac{b}{\pi(b)} \in \pi^{-1}(1) \subseteq \pi^{-1}(\Sbf)$, and $\frac{cdb'}{b} = \frac{cd}{\pi(b)} \in \Sbf \subseteq \pi^{-1}(\Sbf)$. Since $b$ is primal in $\pi^{-1}(\Sbf)$, there are $b_1,b_2 \in \pi^{-1}(\Sbf)$ such that $b = b_1b_2$ and $\frac{c}{b_1}, \frac{db'}{b_2} \in \pi^{-1}(\Sbf)$. Thus $\pi(b_1)$, $\pi(b_2)$ divide respectively $c$, $d$ in $\Sbf$, and since $c$, $d$ were arbitrary, this shows that $\pi(b)$ is primal in $\Sbf$. Now suppose that $\frac{cd}{b} \in \Abf$ for some $c, d \in \Abf$.

  If $\pi(c) \neq 0$, let $c' = \frac{c}{\pi(c)} \in \Abf$, otherwise define $c' = c$. Likewise for $d'$. Since $\frac{cd}{b} \in \Abf$, we also have $\frac{c'd'}{b} \in \Abf$, where $\pi(c'), \pi(d') \in \{0,1\} \subseteq \Sbf$. Therefore, there are $b_1, b_2 \in \pi^{-1}(\Sbf)$ such that $b = b_1b_2$ and $\frac{c'}{b_1}, \frac{d'}{b_2} \in \Abf$, hence in particular $\frac{c}{b_1}, \frac{d}{b_2} \in \Abf$, showing that $b$ is primal in $\Abf$. This proves the left-to-right direction of \prettyref{item:pi(b)-not-0}.

  When $\pi(b) = 0$, suppose that $\frac{cd}{b} \in \pi^{-1}(\Frac(\Sbf))$ for some $c, d \in \pi^{-1}(\Frac(\Sbf))$. Define $c'$, $d'$ as in the previous paragraph. We have that $\frac{c'd'}{b} \in \pi^{-1}(\Frac(\Sbf))$, and since $\pi(c'), \pi(d') \in \{0,1\} \subseteq \Sbf$, there are $b_1, b_2 \in \pi^{-1}(\Sbf)$ such that $b = b_1b_2$ and $\frac{c'}{b_1}, \frac{d'}{b_2} \in \pi^{-1}(\Sbf)$. In turn, $\frac{c}{b_1}, \frac{d}{b_2} \in \pi^{-1}(\Frac(\Sbf))$, proving the left-to-right direction of \prettyref{item:pi(b)-is-0}.

  For the right-to-left direction, pick some $c,d$ in $\pi^{-1}(\Sbf)$ such that $e = \frac{cd}{b} \in \pi^{-1}(\Sbf)$. By assumption, there are $b_1,b_2,f,g \in \Abf$ such that $b = b_1b_2$ and $f = \frac{c}{b_1}, g = \frac{d}{b_2} \in \Abf$; when $\pi(b) = 0$, we may further assume $b_1,b_2,f,g \in \pi^{-1}(\Frac(\Sbf))$. Note that $e = fg$.

  We claim that there exists $\eta \in \Frac(\Sbf)$ such that $b_1\eta^{-1}, b_2\eta, f\eta, g\eta^{-1} \in \pi^{-1}(\Sbf)$. This implies that $\frac{c}{b_1\eta^{-1}}, \frac{d}{b_2\eta} \in \pi^{-1}(\Sbf)$, while $b = (b_1\eta^{-1}) (b_2\eta)$. In particular, since $c, d$ were arbitrary, the claim implies that $b$ is primal in $\pi^{-1}(\Sbf)$.

  For \prettyref{item:pi(b)-not-0}, namely $\pi(b) \neq 0$, we can find non-zero $b_1', b_2' \in \Sbf$ such that $\pi(b) = b_1'b_2'$ and $\frac{\pi(c)}{b_1'}, \frac{\pi(d)}{b_2'} \in \Sbf$. We choose $\eta = \frac{\pi(b_1)}{b_1'}$, so that
  \[ \pi(b_1\eta^{-1}) = b_1', \quad \pi(b_2\eta) = \frac{\pi(b)}{\pi(b_1\eta^{-1})}= b_2', \quad \pi(f\eta) = \frac{\pi(c)}{b_1'}, \quad \pi(g\eta^{-1}) = \frac{\pi(d)}{b_2'} \]
  which are all in $\Sbf$ by construction, as desired.

  Under \prettyref{item:pi(b)-is-0}, the condition $\pi(b) = 0$ forces $\pi(b_1) = 0$ or $\pi(b_2) = 0$, and without loss of generality, we may assume $\pi(b_1) = 0$. When $\pi(g) = 0$, pick some $\eta \in \Sbf$ such that $\eta\pi(b_2),\eta\pi(f) \in \Sbf$. Then:
  \[ \pi(b_1\eta^{-1}) = 0, \quad \pi(b_2\eta) = \eta\pi(b_2), \quad \pi(f\eta) = \eta\pi(f), \quad \pi(g\eta^{-1}) = 0 \]
  which are all in $\Sbf$ by construction. If instead $\pi(g) \neq 0$, we choose $\eta = \pi(g)$. Then:
  \[ \pi(b_1\eta^{-1}) = 0, \quad \pi(b_2\eta) = \pi(b_2g) = \pi(d), \quad \pi(f\eta) = \pi(fg) = \pi(e), \quad \pi(g\eta^{-1}) = 1 \]
  which are all in $\Sbf$ by construction. This concludes the proof of the claim.
\end{proof}

We thus obtain the following generalisation of \prettyref{prop:lift-primal-easy-reduced} and of \cite[Cor.\ 4.3]{BKK2006}.

\begin{prop}
  \label{prop:lift-primal}
  Let $b \in \ZKGG_\kappa$ be reduced and $\sigma = [v(b)]$, with $b \notin \Zbf$. Assume that $0 \in \supp(b)$, or that \refbigcof{} holds. Then $b$ is primal in $\ZKGG_\kappa$ if and only if $\iota_\sigma(b)$ is primal in $\LHH$.
\end{prop}
\begin{proof}
  Note that since $b \notin \Zbf$, we have $\sigma \succ [0]$. Recall that $\iota_\sigma$ restricts to an isomorphism between $\Zbf + \Kbf((\Gbf_{\preceq \sigma}^{<0}))_\kappa$ and $\SLHH$ by \prettyref{fact:isomorphism-K-H-sigma}\prettyref{item:fact-isom-K-H-image}. If $b$ is primal in $\ZKGG_\kappa$, then $b$ is primal in $\Zbf + \Kbf((\Gbf_{\preceq \sigma}^{<0}))_\kappa$, because for any factorisation $b = b_1b_2$ with $b_1, b_2 \in \ZKGG_\kappa$ we must have $v(b_1) \preceq \sigma$, $v(b_2) \preceq \sigma$. Thus, $\iota_\sigma(b)$ is primal in $\SLHH$. Since $b$ is reduced, the coefficient of exponent $0$ of $\iota_\sigma(b)$ is either $0$ or $1$, and one can easily verify in this case that $\iota_\sigma(b)$ is also primal in $\LHH$.

  For the converse, assume that $\iota_\sigma(b)$ is primal in $\LHH$. Suppose that $b$ divides some product $cd$ with $c,d \in \ZKGG_\kappa$. In particular, $b$ divides $T_\sigma(cd) = T_\sigma(c)T_\sigma(d)$, hence $\iota_\sigma(b)$ divides $\iota_\sigma(T_\sigma(cd))$.

  We now apply \prettyref{lem:primal} to $\iota_\sigma(b)$ with $\Lbf = \Lbf_\sigma$, $\Abf = \LHH$, $\pi : \Abf \to \Lbf$ the map projecting each series to its coefficient of exponent $0$, and $\Sbf = \Sbf_\sigma$. Note that by construction, $\pi^{-1}(\Sbf) = \SLHH$; moreover, $\pi(\iota_\sigma(c)) = \tau_\sigma(c)$ for every $c \in \ZKGG_\kappa$.

  If $0 \in \supp(b)$, then $\pi(\iota_\sigma(b)) = \tau_\sigma(b) = 1$ because $b$ is reduced, and in particular $\pi(\iota_\sigma(b))$ is primal in $\Sbf = \Sbf_\sigma$. Otherwise, $\tau_\sigma(b) = 0$, and by \refbigcof{}, either $\Frac(\Sbf_\sigma) = \Lbf_\sigma$ by \prettyref{prop:frac-Z-is-K}, or $\Gbf_{\prec \sigma} = \{0\}$ and $\Frac(\Sbf_\sigma) = \Frac(\Zbf) = \Kbf = \Lbf_\sigma$. In all cases, the assumptions of \prettyref{lem:primal} are satisfied, and we obtain that $\iota_\sigma(b)$ is primal in $\SLHH$, and so $b$ is primal in $\Zbf + \Kbf((\Gbf_{\preceq \sigma}^{<0}))_\kappa$.

  Therefore, there are $b_1, b_2 \in \ZKGG_\kappa$ such that $b = b_1b_2$ with $b_1,b_2$ dividing respectively $T_\sigma(c)$, $T_\sigma(d)$, hence dividing respectively $c$, $d$, as desired.
\end{proof}

We may then generalise \prettyref{prop:pseudo-poly-primal-easy} to all pseudo-polynomials.

\begin{cor}
  \label{cor:pseudo-poly-primal}
  Let $p \in \ZKGG_\kappa$ be a pseudo-polynomial. Assume that \refarchquot{} and \refbigcof{} hold for $\sigma = [v(p)]$, or that $0 \in \supp(p)$. Then $p$ is primal in $\ZKGG_\kappa$. In particular, all pseudo-polynomials of $\oz$ are primal in $\oz$.
\end{cor}
\begin{proof}
  Let $p \in \ZKGG_\kappa$ and $\sigma = [v(p)]$. Then:
  \begin{itemize}
  \item when $0 \in \supp(b)$, $p$ is primal in $\ZKGG_\kappa$ by \prettyref{prop:pseudo-poly-primal-easy};
  \item when \refarchquot{}, \refbigcof{} hold, $\iota_\sigma(p)$ is primal in $\Lbf_\sigma((\Rb^{\leq 0}))$ by \prettyref{cor:KR-primal-KRR}, thus $p$ is primal in $\ZKGG_\kappa$ by \prettyref{prop:lift-primal}. \qedhere
  \end{itemize}
\end{proof}

\begin{cor}\label{cor:finite-supp-primal}
  Let $p \in \ZKGG_\kappa$ be a series with finite support and $\sigma \in \Gbf_{/\asymp}$ be minimal intersecting $\supp(b)$. Assume \refarchquot{}, \refbigcof{}, and if $0 \in \supp(b)$, assume \refZpS{}. Then $p$ is primal in $\ZKGG_\kappa$. In particular, all omnific integers with finite support are primal in $\oz$.
\end{cor}
\begin{proof}
  Just repeat the argument for \prettyref{cor:finite-supp-primal-easy} with \prettyref{cor:pseudo-poly-primal} and \refZpS{} in place of \prettyref{prop:pseudo-poly-primal-easy}.
\end{proof}

\begin{exa}\label{exa:finite-supp-primal}
  For instance $\omega^{\sqrt{2}} + \omega + 1$, $\omega^{2+\frac{1}{\omega}} + \omega$, $\omega + \omega^{\frac{1}{\omega}}$ have finite support, thus they are primal in $\oz$, providing an alternative proof of \prettyref{main:Gonshor}.
\end{exa}

\begin{cor}
  \label{cor:primal-in-oz}
  Let $b \in \ZKGG_\kappa$ be reduced with pseudo-support of order type $\omega \hatplus k$ for some $k \in \Nb$, and $\sigma = [v(b)]$. If $0 \in \osupp(b)$ and \refbigcof{} holds, or if \refarchquot{} and \refbigcof{} hold, then $b$ is primal. If $0 \in \supp(b)$, then $b$ is prime.
\end{cor}
\begin{proof}
  Let $b \in \ZKGG_\kappa$ be reduced with pseudo-support of order type $\omega \hatplus k$ for some $k \in \Nb$ and let $\sigma = [v(b)]$. Let $b' = \frac{\iota_\sigma(b)}{t^x} \in \Lbf_\sigma((\Rb^{\leq 0}))$, where $x = \sup(\osupp(b)) \in \Rb$, so that $b'$ is not divisible by any $t^y$ with $y \in \Rb^{<0}$.

  By \prettyref{main:prime}, $b'$ is prime in $\Lbf_\sigma((\Rb^{\leq0}))$. If $0 \in \osupp(b)$, then $x = 0$, and in particular $b' \in \LHH$, hence $b'$ is prime in $\LHH$, thus it is primal in $\ZKGG_\kappa$ by \prettyref{prop:lift-primal} and $0 \in \supp(b)$ or \refbigcof{}. When $0 \in \supp(b)$, $b$ is also irreducible by \prettyref{prop:fact-almost-irreducible}, hence it is prime in $\ZKGG_\kappa$. If on the other hand $0 \notin \osupp(b)$, \refarchquot{} and \prettyref{cor:KR-primal-KRR} ensure that $t^x$ and $b'$ are primal in $\LHH = \Lbf_\sigma((\Rb^{\leq 0}))$, thus $b$ is primal in $\ZKGG_\kappa$ by \refbigcof{} and \prettyref{prop:lift-primal}.
\end{proof}

\begin{cor}
  If $b \in \oz$ is reduced with pseudo-support of order type $\omega \hatplus k$ for some $k \in \Nb$, then $b$ is primal, and it is prime if and only if $0 \in \supp(b)$.
\end{cor}
\begin{proof}
  Simply recall that in $\oz$, \refarchquot{} and \refbigcof{} hold, thus any $b$ as in the assumptions is primal, and it is prime if $0 \in \supp(b)$. Conversely, if $b$ is prime, then it is irreducible, thus $0 \in \supp(b)$.
\end{proof}

\begin{exas}
  \label{exas:primes-primal}
  The omnific integer $1 + \sum_{n \in \Nb}\omega^{\frac{1}{n+1}}$ is a prime of $\oz$, since it is clearly reduced, its pseudo-support has order type $\omega \hatplus 1$, and it contains $0$ in its support (recall that this could also be deduced from \cite{BKK2006,Pit2001}). Similarly, the omnific integer $\sum_{n \in \Nb}\omega^{\frac{1}{n+1}}$ is primal in $\oz$, but obviously not prime as it is not irreducible. Note moreover that thanks to \prettyref{prop:lift-primal}, every reduced prime series found using the results of \prettyref{sec:new-irreducibles-primes} lifts to a primal or prime series; for instance, from \prettyref{exa:prime-omega2+k+1}, we obtain that
  \begin{gather*}
    \left(\sum_{n \in \Nb} \omega^{\frac{1}{3n+1}}\right)\omega^3 + \left(\sum_{n \in \Nb} \omega^{\frac{1}{3n+2}}\right)\omega^2 + \omega^{\frac{k-1}{k}} + \dots + \omega^{\frac{1}{k}} + 1
  \end{gather*}
  is prime in $\oz$.
\end{exas}

\printbibliography

\end{document}